\documentclass[12pt]{amsart}
\usepackage{amsmath}
\usepackage{amsthm}
\usepackage{amsfonts}
\usepackage{amscd}
\usepackage{amssymb}
\usepackage{amstext}
\usepackage{hyperref}
\usepackage{epsfig}

\numberwithin{equation}{section}

\allowdisplaybreaks

\newtheorem{theorem}{Theorem}[section]
\newtheorem{proposition}{Proposition}[section]
\newtheorem{lemma}{Lemma}[section]
\newtheorem{remark}{Remark}[section]
\newtheorem{definition}{Definition}[section]
\newtheorem{corollary}{Corollary}[section]

\renewcommand{\epsilon}{\varepsilon}

\newcommand{\abs}[1]{\left\vert #1\right\vert}
\newcommand{\1}[1]{{\mathbf 1}{\{#1\}}}
\newcommand{\R}{\mathbb{R}}
\newcommand{\N}{\mathbb{N}}
\newcommand{\Z}{\mathbb{Z}}
\newcommand{\T}{\mathcal{T}}

\newcommand{\reff}{R_{\text{eff}}}

\def\fff#1{&{{\pageref{#1}}}\cr}
\def\hfff#1{\label{#1}}

\title[]{Scaling limit for the ant in a simple labyrinth}
\date{}
\author[G.~Ben Arous]{G\'erard Ben Arous}
\address{G\'erard Ben Arous, Courant Institute of Mathematical Sciences, 251 Mercer Street, New York University,
New York, 12012-1185, U.S.A.} \email{benarous@cims.nyu.edu}
\author[M.~Cabezas]{Manuel Cabezas}
\address{Manuel Cabezas\\ Pontificia Universidad Cat\'{o}lica de Chile, Facultad de Matem\'{a}ticas, Campus San Joaqu\'{i}n, Avenida Vicu\~{n}a Mackenna 4860, Santiago, Chile.} \email{mncabeza@mat.puc.cl}

\author[A.~Fribergh]{Alexander Fribergh}
\address{Alexander Fribergh\\Universit\'e de Montr\'eal, DMS\\
Pavillon Andr\'e-Aisenstadt\\     2920, chemin de la Tour Montréal (Qu\'ebec),  H3T 1J4} \email{fribergh@dms.umontreal.ca}


\begin{document}

\begin{abstract}
We prove that, after suitable rescaling, the simple random walk on the trace of a large critical branching random walk converges to the Brownian motion on the integrated super-Brownian excursion.
\end{abstract}

\maketitle

\section{Introduction}

The last decade has witnessed a resurgence of mathematical interest in
multi-dimensional random walk in random environment. We refer the reader to~\cite{zeitouni2004part}, \cite{SZ1}, \cite{SZ2}, \cite{kumagai} and~\cite{benarous_fribergh}, for several surveys of the field. 

Random walks in random environments (RWRE) is a subject deeply rooted in physics which became popular because several models exhibit anomalous behaviors. As an illustrative model, de Gennes in~\cite{deGennes1976} proposed to study the random walk on critical percolation clusters. This model is very difficult to study and actually a proper definition of the infinite critical clusters is only available for $d=2$ (see~\cite{kesten1986subdiffusive}) and for large $d$ (see~\cite{vdhjarai} and~\cite{HHH-revisited}). 

In high dimensions (currently meaning $d\geq 11$, see~\cite{fvdh1} and~\cite{fvdh2}) infinite critical clusters were defined using lace expansion (for a recent survey see~\cite{PIMS}). This technique allowed to obtained results (see~\cite{Hara}) which opened the door to one of the first significant result concerning the ant in the labyrinth, \cite{KN}, where the authors proved the Alexander-Orbach conjecture (see~\cite{AO}) stating that the spectral dimension in this model is $4/3$. This results was previously known for critical trees~\cite{Barlow_Kumagai} and for critical oriented percolation~\cite{BJKS}. A detailed discussion of the Alexander-Orbach conjecture can be found in~\cite{kumagai}.

The technique of lace expansion was developed to study random graphs in high dimensions that are critical. It allowed to prove that  critical branching random walks, oriented percolation, percolation and lattice trees have, in some sense, similar universal large scale behavior. This is explained in more details in Section 6 of~\cite{van2006infinite}. Because of this, proving results for the simple random walk on any of those models should provide a blueprint for results on other models in the same universality class.

Our goal in this paper is to prove that the properly normalized random walk on the range of critical branching random walks converges to an object known as the Brownian motion in the ISE (introduced in~\cite{Croydon_arc}). This provides a sophisticated result on the dynamics on the simplest model in the universality class of the critical percolation cluster. Our result is obtained by applying a general abstract convergence theorem (Theorem~\ref{thm_abstract}) which is proved in~\cite{BCFa}.

\subsection{The model}\label{sect_model}

We begin by introducing the critical branching random walk. For this, we will consider a critical Galton-Watson tree, which is a branching process with i.i.d.~offspring that are copies of a random variable $Z$ defined under a law ${\bf P}$ verifying ${\bf E}[Z]=1$ (excluding the case where $Z=1$ a.s.)\hfff{P}\hfff{Z}. Given a random realization of a Galton-Watson tree $\T^{\text{GW}}\hfff{gw}$, we can consider the simple random walk indexed by $\T^{\text{GW}}$, which means we will assign a spatial location $\phi_{\T^{\text{GW}}}(u)\in \Z^d$ for every $u\in \T^{\text{GW}}$. First, the spatial location $\phi_{\T^{\text{GW}}}(\text{root})$ of the root is the origin of $\Z^d$. Then, each edge $e\in \T^{\text{GW}}$ gets assigned, in an i.i.d.~manner, a random variable $X_e$ which is distributed according to the jump distribution of a simple random walk. The spatial location $\phi_{\T^{\text{GW}}}(u)$ of a vertex $u$ is the sum of the quantities $X_e$ over all edges $e$ belonging to the simple path from the root to $u$ in the tree. The couple $(\T^{\text{GW}},\phi_{\T^{\text{GW}}})$ is a random spatial tree under a measure that we will still denote ${\bf P}$. We will call this object the critical branching random walk. This spatial tree can be viewed as a subgraph of $\Z^d$, by considering the graph $\omega$ with vertices given by $\{x\in \Z^d,\ x=\phi_{\T^{\text{GW}}}(u) \text{ with } u\in \T^{\text{GW}}\}$ and edges given by $\{[x_1,x_2]\in E(\Z)^d,\ x_i=\phi_{\T^{\text{GW}}}(\T^{\text{GW}}(u_i)) \text{ for $i\in \{1,2\}$ with } [u_1,u_2] \in E(\T^{\text{GW}})\}$. Obviously, this embedded subgraph is not necessarily a tree.

In this paper, we will be mainly interested in large critical branching random walks, obtained by our previous construction under the measure ${\bf P}_n:={\bf P}[\,\cdot \mid \abs{\T^{\text{GW}}}=n]$\hfff{Pn}, where $\abs{\T^{\text{GW}}}$ denotes the vertex cardinality of $\T^{\text{GW}}$. In this case we will denoted the spatial tree by $(\T^{\text{GW}}_n,\phi_{\T^{\text{GW}}_n})$\hfff{phin} and the random embedded graph (which is $\phi_{\T^{\text{GW}}_n}(\T^{\text{GW}}_n)$) by $\omega_n$\hfff{omega}. 

In this paper we are going to study the simple random walk $(X_m^{\omega_n})_{m\in \N}$, started from $0$, on the range $\omega_n$ of the large critical branching random walks.

\subsection{Main result}

It is well known that $\omega_n$ properly rescaled has a scaling limit called the integrated super-Brownian excursion (ISE) (introduced in~\cite{Al4}) and there is a natural dynamic on this object called the Brownian motion on the ISE (introduced in~\cite{Croydon_arc}). We delay a detailed description of these objects to Section~\ref{section_bmsbm}.  Our main result is that  
\begin{theorem}\label{thm:mainannealed} 
Let us take $d>14$, assume \footnote{The probability measure ${\bf P}_n$ can asymptotically only be defined if the distribution of $Z$ is not supported on a lattice. We make this assumption in this theorem for simplicity. However, if this was not verified one  could get a similar theorem by restricting the convergence to a sub-sequence of the form $(dn)_{n\geq 0}$ where $d$ is the largest integer such that $Z/d$ is an integer almost surely} that $E[Z]=1$, $\sigma^2_Z=\text{Var}(Z)\in (0,\infty)$\hfff{sigZ} and ${\bf E}[\exp(\lambda Z)]<\infty$ for some $\lambda>0$. There exists constants $\nu>0$, $\rho_1>0$ and $\sigma>0$ such that, under ${\bf P}_n$ we have
  \[
  (n^{-1/4}X^{\omega_n}_{tn^{3/2}})_{t\geq0}\to (\sqrt{\sigma} B^{ISE}_{\nu^{-1}\sigma^{-1}\rho_1^{-1}t})_{t\geq0},
  \]
  where $B^{ISE}$ is the {\emph Brownian motion on the ISE} and the convergence is annealed and occurs in the topology of uniform convergence over compact sets. 
  
  We have $\sigma=2/\sigma_Z$ and the definitions of  $\nu$ and $\rho_1$ can be found at~\eqref{def_nu} and \eqref{def_rho1}.
\end{theorem}

This theorem is proved at the end of Section~\ref{sect_abstract_thm}.

\begin{remark} 
In~\cite{Croydon_arc}, Croydon proved that the Brownian motion on the ISE appeared in the scaling limit of a sequence of processes, $(n^{-1/4}\phi_{n}(Y^{\T^{\text{GW}}_n}_{tn^{3/2}}))_{t\geq 0}$ where $(Y_m^{\T^{\text{GW}}_n})_{m\in \N}$ is the simple random walk on $\T^{\text{GW}}_n$.  

It is important to notice that, in nature, the process considered by Croydon is very different from the one we consider in this paper. In essence $(\phi_{\T^{\text{GW}}_n}(Y_m))_{m\geq 0}$ is a random walk on a tree, which is then embedded. In our case $X_n$ is random walk on a graph that contains loops. This adds significant complexity to the analysis of the problem. 

The models have similar scaling limits because asymptotically the loops of $\omega_n$ are microscopic and their only effect, in the limit, is to speed up the random walk down by a factor $(\rho \nu)^{-1}>1$.
\end{remark}

 Our result on large critical branching random walks, Theorem~\ref{thm:mainannealed}, is obtained as a application of an abstract result available in~\cite{BCFa}. We decided to postpone the statement of this theorem because it requires a large amount of notations (see  Theorem~\ref{thm_abstract}). 

Finally, let us emphasize that in order to obtain Theorem~\ref{thm:mainannealed} we need to prove that the resistance metric and the graph metric are proportional in high dimensions (this result does not hold for $d<6$, see~\cite{JN}). This statement is of interest on its own right and one way it  can be formulated precisely is as in Theorem~\ref{theorem_res}. This answers a question raised in Section 1.8 (iii) of~\cite{BJKS}.

\subsection{Notations}

Given a graph $G$, we will denote $V(G)$ the set of its vertices and $E(G)$ the set of its edges. For $x\in G$ and $k\in \R$, we will write $B_G(x,k)$ for the ball of $k$ centered at $x$ in the natural metric induced by $G$.

Given a rooted tree $T$, we denote $H(T)$  the associated height of the underlying tree, i.e.~the largest non-empty generation, and for any $i\leq H(T)$ we denote $Z_i(T)$ (or simply $Z_i$) the number of offspring at level $i$ in $T$. 

The constants in this paper will typically be denoted $c$ (for lower bounds) and $C$ (for upper bounds) and implicitly assumed to be positive and finite. Their value may change from line to line.

This paper contains a significant amount of notation, so we decided to include a glossary of notation at the end of the paper to help the reader.

\subsection{Organization of the paper}

We start by introducing some notations related to the Brownian motion on the ISE in Section~\ref{section_bmsbm}. This will allow us to state the precise conditions that imply convergence to the Brownian motion on the ISE in Section~\ref{sect_abstract_cvg_thm}. We will then proceed to prove those conditions in the final sections.

The first condition is a resistance estimate, which is the most difficult of this paper. It is done in two parts: we start by introducing objects and techniques related to infinite and bi-infinite versions of critical branching random walks see Section~\ref{section_bi_inf}. Those objects are subsequently used in Section~\ref{sect_resistance} to prove that the resistance and the intrinsic metric are proportional for branching random walks in high dimensions.

The second condition relates to the spatial geometry of critical branching random walks. This problem is addressed in Section~\ref{sect:condG}. The third and final condition is an estimate on the asymptotic volume distribution which is tackled in Section~\ref{sect:condV}.

The paper ends with an appendix with basic estimates on branching random walks and a glossary of notations.

\section{The Brownian motion on the ISE}\label{section_bmsbm}

\subsection{Real trees and spatial trees}

Before defining the Brownian motion on the  ISE it is necessary to actually define the ISE. For this, we choose to introduce the formalism of real trees and spatial trees of which the ISE is the canonical random example. For this we follow, almost to the word, notes from Le Gall (see~\cite{LG}).

\subsubsection{Real trees}\label{sect_real_tree}

\begin{definition}
A metric space $(T,d_{T})$ is a real tree ($\R$-tree) if the following two
properties hold for every $\sigma_1,\sigma_2\in T$.

\begin{enumerate}
\item There is a unique
isometric map
$f_{\sigma_1,\sigma_2}$ from $[0,d_{T}(\sigma_1,\sigma_2)]$ into $T$ such
that $f_{\sigma_1,\sigma_2}(0)=\sigma_1$ and $f_{\sigma_1,\sigma_2}(
d_{T}(\sigma_1,\sigma_2))=\sigma_2$.
\item If $q$ is a continuous injective map from $[0,1]$ into
$T$, such that $q(0)=\sigma_1$ and $q(1)=\sigma_2$, we have
$q([0,1])=f_{\sigma_1,\sigma_2}([0,d_{T}(\sigma_1,\sigma_2)]).$
\end{enumerate}

A rooted real tree is a real tree $(T,d_{T})$
with a distinguished vertex called the root.
\end{definition}

Let us consider a rooted real tree $(T,d)$.
The range of the mapping $f_{\sigma_1,\sigma_2}$ in (1) is denoted by
$[ \sigma_1,\sigma_2]$ (this is the line segment between $\sigma_1$
and $\sigma_2$ in the tree). 
In particular, for every $\sigma\in T$, $[ \text{root},\sigma]$ is the path 
going from the root to $\sigma$, which we will interpret as the ancestral
line of the point $\sigma$. More precisely we can define a partial order on the
tree by setting $\sigma\preccurlyeq \sigma'$
($\sigma$ is an ancestor of $\sigma'$) if and only if $\sigma\in [ \text{root},\sigma']$, and, $\sigma \prec \sigma'$ if $\sigma\preccurlyeq \sigma'$ and $\sigma \neq \sigma'$.

If $\sigma,\sigma'\in T$, there is a unique $\eta\in T$ such that
$[ \text{root},\sigma ] \cap [\text{root} ,\sigma']=[\text{root},\eta ]$. We write $\eta=\sigma\wedge \sigma'$ and call $\eta$ the most recent
common ancestor to $\sigma$ and $\sigma'$.

Finally, let us observe that for any three points $\sigma_1,\sigma_2,\sigma_3$ of a real tree $T$ there exists a unique branching point $b^{T}(\sigma_1,\sigma_2,\sigma_3)\in T$ that satisfies $b^{T}(\sigma_1,\sigma_2,\sigma_3)\in T=[\sigma_1,\sigma_2]\cap [\sigma_2,\sigma_3]\cap [\sigma_3,\sigma_1].$

There are collections of real trees that cannot be distinguished as metric
spaces. For compact rooted real
trees (which are the only type of real trees we consider in this paper) two rooted real trees are equivalent if and only if
there exists a root preserving isometry between them. For our purposes, this subtlety will not be relevant and we will not make any distinction between a tree and its equivalence class. See~\cite{LG} for more details.

\vspace{0.5cm}

{\it A way to construct real trees }

\vspace{0.5cm}

There is a simple way of constructing compact real trees. We consider a	 (deterministic) continuous function
$g:[0,\infty)\longrightarrow[0,\infty)$ with compact support
and such that $g(0)=0$ and $g(x)=0$ for $x$ larger than some $x_0>0$ but $g$ is not identically zero.

For every $s,t\geq 0$, we set
\[
m_g(s,t)=\inf_{r\in[s\wedge t,s\vee t]}g(r),
\]
and
\begin{equation}\label{eq:distancefromexcursion}
d_g(s,t)=g(s)+g(t)-2m_g(s,t).
\end{equation}


We then introduce the equivalence relation
$s\sim t$ iff $d_g(s,t)=0$ (or equivalently iff $g(s)=g(t)=m_g(s,t)$). Let
$T_g$ be the quotient space
\begin{equation}\label{def_equiv_tree}
T_g=[0,\infty)/ \sim.
\end{equation}

Obviously the function $d_g$ induces a distance on $T_g$, and we keep the
notation $d_g$ for this distance. Viewing the equivalence class of $0$ as the root, this means we have the following (see~\cite{DuLG})

\begin{theorem}
\label{tree-deterministic}
The metric space $(T_g,d_g)$ is a rooted real tree.
\end{theorem}

We will call $T_g$  the real tree coded by $g$.

\subsubsection{Spatial trees}\label{sect_spatial_tree}

\begin{definition} A ($d$-dimensional) spatial tree is a pair $(T,\phi_{T})$ where 
$T$ is a real tree and $\phi_T$ is a continuous mapping from 
$\phi_T$ into $\R^d$. 
\end{definition}

\begin{remark}
Two spatial trees $(T,\phi_T)$ and $(T', \phi_{T'})$ are said to be equivalent if and only if there exists a root preserving isometry $\pi$ from $T$ to $T'$ such that $\phi_{T}=\phi_{T'}\circ \pi$. In Section~\ref{sect_condG}, we will define a topology on spatial trees, which will actually be a topology on the equivalence classes of spatial trees with respect to the previous relation. Nevertheless, for our purposes it does not pose a problem to identify a tree with its equivalence class.
\end{remark}

Let $T$ be a compact rooted real tree  with a metric $d$. We may consider the $\R^d$-valued Gaussian process
$(\phi_{T}(\sigma),\sigma\in T)$ whose distribution is characterized by
\begin{align*}
&E[\phi_{T}(\sigma)]=0\;,\\
&{\rm cov}(\phi_{T}(\sigma),\phi_{T}(\sigma'))=d(\text{root},\sigma\wedge \sigma')\,{\rm Id}\;,
\end{align*}
where ${\rm Id}$ denotes the $d$-dimensional identity matrix. 

This corresponds to a Brownian embedding of $T$ into $\R^d$. The formula for the covariance is easy to understand if we recall that $\sigma\wedge \sigma'$ is the most recent common ancestor to $\sigma$
and $\sigma'$, and so the ancestors of $\sigma$ and $\sigma'$ are the same up to
level $d(\text{root},\sigma\wedge \sigma')$. 

Under certain assumptions, that will be verified in our context (see (8) in~\cite{LG} for details) the process $(\phi_{T}(\sigma),\sigma\in T)$ has a continuous modification. We keep the notation $\phi_T$ for this modification. 

Given a real tree $T$, we denote by $Q_T$ the law of the spatial tree $(T,(\phi_{T}(\sigma),\sigma\in T))$ (provided it exists).

\subsubsection{Graph spatial trees}\label{sect_graph_tree}

Let us now present a notion introduced by Croydon in~\cite{Croydon_arc}.

\begin{definition}
If a spatial tree $(T,(\phi_T(\sigma),\sigma\in T))$ is such that $T$ is a finite tree with finite edge length, we say that $(T,(\phi_T(\sigma),\sigma\in T))$ is a graph spatial tree.
\end{definition}

 Given a graph spatial tree $(T,(\phi_T(\sigma),\sigma\in T))$, we can assign a probability measure $\lambda_{T}$  defined as the renormalized Lebesgue measure \hfff{lambdaT}(so that the $\lambda_{T}$-measure of a line segment in $T$ is proportional to its length).

\vspace{0.5cm}

{\it A simple way to construct graph spatial  trees }

\vspace{0.5cm}

There is a simple way to construct a rooted graph spatial tree from a rooted spatial tree $(T,d_{T},\phi_T)$. For this we consider a sequence $(\sigma_i)_{i\in \N}$ of elements of a real tree $T$. Fix $K\in \N$. We define the reduced subtree $T(\sigma_1,\ldots,\sigma_K)$ to be the graph tree with vertex set 
\[
V(T(\sigma_1,\ldots, \sigma_K)):=\{b^{T}(\sigma,\sigma',\sigma''):\sigma,\sigma',\sigma'' \in \{\text{root},\sigma_1,\ldots,\sigma_K\}\},
\]
and graph tree structure induced by the arcs of $T$, so that two elements $\sigma$ and $\sigma'$ of $V(T(\sigma_1,\ldots,\sigma_K))$ are connected by an edge if and only if $\sigma\neq \sigma'$ and also $[\sigma,\sigma']\cap V(T(\sigma_1,\ldots,\sigma_K))=\{\sigma,\sigma'\}$. We set the length of an edge $\{\sigma,\sigma'\}$ to be equal to $d_{T}(\sigma,\sigma')$ and we extend the distance linearly on that edge. This allows us to view $T(\sigma_1,\ldots,\sigma_K)$ as a graph spatial tree.

This spatial graph tree will be denoted $(T^{K,(\sigma_i)},d_{T^{K,(\sigma_i)}}, \phi_{T^{K,(\sigma_i)}})$. The associated normalized probability measure is denoted $\lambda_{\phi_{T^{K,(\sigma_i)}}(T^{K,(\sigma_i)})}$. The dependence on $(\sigma_i)$ will often be dropped in the notation when the context is clear.

\subsection{Definition of the $CRT$, the $ISE$ and the $B^{ISE}$}

In this section our goal is to introduce the \lq\lq canonical random object\rq\rq\ associated to real trees, spatial trees, graphs spatial trees and dynamics in these objects.

\subsubsection{The continuum random tree (CRT)}\label{sect:crt}

Denote by $({\bf e}_t)_{0\leq t\leq 1}$ a normalized Brownian excursion. Informally, $({\bf e}_t)_{0\leq
t\leq 1}$ is just a Brownian path started at the origin and conditioned to stay positive
over the time interval $(0,1)$, and to come back to $0$ at time $1$ (see e.g.~Sections 2.9 
and 2.12 of It\^o and McKean \cite{IM} for a discussion of the normalized excursion). We
extend 
the definition of ${\bf e}_t$ by setting ${\bf e}_t=0$ if $t>1$. Then the (random) function
${\bf e}$ satisfies the assumptions of Section~\ref{sect_real_tree} and we can thus consider the
real tree $T_{\bf e}$, which is a random variable with values real trees.

\begin{definition}
\label{CRTdef}
The random real tree $T_{\bf e}$ is called the Continuum Random Tree (CRT) and will be often denoted $(\mathfrak{T},d_{\frak{T}})$ \hfff{crt}. We write $\Xi$ to denote its law.
\end{definition}

The CRT was initially defined by Aldous \cite{Al1} with a different formalism,
but the preceding definition corresponds to Corollary 22 in \cite{Al3}, up to an unimportant
scaling factor $2$.

We can define a natural volume measure on $\mathfrak{T}$ by projecting the Lebesgue measure on $[0,1]$, i.e.~for any open $A\subseteq \mathfrak{T}$, we set
\begin{equation}\label{eq:lambdaT}
\lambda^{\mathfrak{T}}(A)=\text{Leb}\{t\in [0,1], [t]\in A\},
\end{equation}
where $[t]$ denotes the equivalence class of $t$ with respect to the relation defined at~\eqref{def_equiv_tree}.

One major motivation for studying the CRT is the fact that it occurs as the
scaling limit of critical ${\bf P}_n$-Galton-Watson trees. In particular, recalling the notations of the introduction, we have the following (see Theorem~3.1 in~\cite{LG} which is a simple consequence of Theorem 23 in~\cite{Al3})
\begin{theorem} 
Assume that the offspring distribution $Z$ is critical, with finite variance $\sigma^2>0$ and is aperiodic. Then the  rescaled real tree $(\text{GW}_n,\frac{\sigma}{2\sqrt n} d_{\text{GW}_n})$ chosen under ${\bf P}_n$ converges to the CRT, where  the convergence occurs in distribution with the Gromov-Hausdorff topology. 
\end{theorem}

\subsubsection{The integrated super-Brownian excursion (ISE)}

We will combine the CRT with $d$-dimensional Brownian motions started from $x=0$, in the
way explained in Section~\ref{sect_spatial_tree}. Precisely this means that we are
considering the probability measure on spatial trees defined by
\[
M=\int \Xi(d\mathfrak{T})\,Q_{\mathfrak{T}}.
\]

Recall the notation $\lambda_{\mathfrak{T}}$ for the uniform measure on $\mathfrak{T}$
(this makes sense $\Xi(d\mathfrak{T})$ a.s.).

\begin{definition}
The random probability measure $\lambda^{\phi_{\mathfrak{T}}(\mathfrak{T})}$ on $\R^d$ defined under $M$ by $\lambda^{\phi_{\mathfrak{T}}(\mathfrak{T})}:=\lambda^{\mathfrak{T}} \circ \phi_{\mathfrak{T}}^{-1}$ is called $d$-dimensional ISE (for Integrated Super-Brownian Excursion).
\end{definition}

Note that the topological support of ISE is the range 
of the spatial tree, and that ISE should be interpreted as the
uniform measure on this set. We will often abuse the terminology and write ISE to mean its topological support. We will write $\phi_{\mathfrak{T}}(\mathfrak{T})$ \hfff{ise} to designated this set.

The random measure
ISE was first discussed by Aldous~\cite{Al4}. It occurs in various 
asymptotics for models of statistical mechanics (see in particular~\cite{Sl1} and~\cite{Sl2}).

\subsubsection{The Brownian motion on the ISE: $B^{ISE}$}

We are now going to define a canonical dynamic on the ISE. For this we will start by discussing the Brownian motion on the CRT.

Let $(T,d_T)$ be any real tree and $\nu$ a (Borelian) probability measure on $T$. It was suggested by Aldous~\cite{Al2} that a Brownian motion on $(T,d_{T},\nu)$ should be a strong Markov process with continuous sample paths that is reversible with respect to its invariant measure $\nu$ and satisfies the following properties,
\begin{enumerate}
\item For $\sigma_1,\sigma_2 \in T$ with $\sigma_1\neq \sigma_2$, we have 
\[
P_{\sigma}^{T,\nu}(T_{\sigma_1}<T_{\sigma_2})=\frac{d_{T}(b^{T}(\sigma,\sigma_1,\sigma_2),\sigma_2)}{d_{T}(\sigma_1,\sigma_2)}, \qquad \forall \sigma \in T,
\]
where $T_{\sigma}:=\inf\{t>0,X_t^{T}=\sigma\}$ is the hitting time of $\sigma \in T$.
\item For $\sigma_1,\sigma_2 \in T$, the mean occupation measure for the process started at $\sigma_1$ and killed on hitting $\sigma_2$ has density
\[
2d_{T}(b^{T}(\sigma,\sigma_1,\sigma_2),\sigma_2)\nu(d\sigma) \qquad \forall \sigma \in T.
\]
\end{enumerate}

These properties guarantee the uniqueness of the Brownian motion on $(T,d_{T},\nu)$. 

The existence of such a process follows from techniques of resistance forms (see~\cite{Kigami_Harm} for an introduction on resistance forms). More specifically, it was proved in Section 6 of~\cite{Croydon_arc} that
\begin{proposition}\label{prop_def_process}
Let $(T,d_{T})$ be a compact real tree, $\nu$ be a finite Borel measure on $T$ that satisfies $\nu(A)>0$ for every non-empty open set $A\subseteq T$ and $(\mathcal{E}_{T},\mathcal{F}_{T})$ be the resistance form associated with $(T,d_{T})$. Then $(\frac 12 \mathcal{E}_{T},\mathcal{F}_{T})$ is a local, regular Dirichlet form on $L^2(T, \nu)$, and the corresponding Markov process $B^{T,\nu}$ is the Brownian motion on $(T,d_{T},\nu)$.
\end{proposition}

For $d\geq 8$, it can be proved that $\phi_{\mathfrak{T}}$ is injective from $\mathfrak{T}$ (the CRT)  to $\phi_{\mathfrak{T}}(\mathfrak{T})$ (the range of the ISE), see Proposition 3.5.~in~\cite{Croydon_arc}. This means $\phi_{\mathfrak{T}}$ is actually and isometry between $(\mathfrak{T},d_{\mathfrak{T}})$ and $(\phi_{\mathfrak{T}}(\mathfrak{T}),d_{\phi_{\mathfrak{T}}(\mathfrak{T})})$ which sends  $\lambda_{\mathfrak{T}}$ to $\lambda_{\phi_{\mathfrak{T}}(\mathfrak{T})}$.  Hence, $(\phi_{\mathfrak{T}}(\mathfrak{T}),d_{\phi_{\mathfrak{T}}(\mathfrak{T})})$ is a real tree and this allows us to define easily a process, which in the sense defined by Aldous, is the Brownian motion on the ISE in $\Z^d$ for $d\geq 8$.

\begin{proposition}\hfff{bise}\label{propdef_BISE} For $\Xi$-a.e.~$\mathfrak{T}$, the Brownian motion $B^{CRT}$ on $(\mathfrak{T},d_{\mathfrak{T}},\lambda^{\mathfrak{T}})$ exists. Furthermore if $d\geq 8$, for $M$-a.e.~ $(\mathfrak{T},\phi)$, the Brownian motion $B^{ISE}$ on $(\phi_{\mathfrak{T}}(\mathfrak{T}),d_{\phi_{\mathfrak{T}}(\mathfrak{T})},\lambda^{\phi(\mathfrak{T})})$ exists and, moreover, $B^{ISE}=\phi_{\mathfrak{T}}(B^{CRT})$.
\end{proposition}

Since the map $\phi_{\mathfrak{T}}:\mathfrak{T}\to \R^d$ is continuous for $M$-a.e.~spatial trees $(\mathfrak{T},\phi_{\mathfrak{T}})$ the law of $B^{ISE}$ is, $M$-a.s., a well-defined probability measure on $C(\R_+,\R^d)$.

\subsubsection{Approximating the ISE and the $B^{ISE}$ using graph spatial trees}\label{sect_KISE}

It will be useful for us to approximate the ISE (and the $B^{ISE}$) by a graph spatial tree (and a process on this graph spatial tree). Indeed, the topological structure of the graph spatial tree is much simpler which makes it easier to study. This is an idea that was already used by Croydon in Section 8 of~\cite{Croydon_arc} (building on ideas he developed in~\cite{Croydon_crt}).

Consider $\mathfrak{T}$ a realization of the CRT and $(U_i)_{i\in \N}$ chosen according to $(\lambda^{\mathfrak{T}})^{\otimes \N}$.
Fix $K\in \N$ . We can use the construction described in Section~\ref{sect_graph_tree} to define a graph spatial tree, which we call $K$-ISE and denote $(\mathfrak{T}^{(K)},d_{\mathfrak{T}^{(K)}},\phi_{\mathfrak{T}^{(K)}})$, where $\mathfrak{T}^{(K)}$ is called $K$-CRT \hfff{kcrt}. We recall that this object comes with a probability measure $\lambda_{\phi_{\mathfrak{T}^{(K)}}(\mathfrak{T}^{(K)})}$, which is the normalized Lebesgue measure on $\mathfrak{T}^{(K)}$. For the sake of simplicity we will denote $\lambda_{\phi_{\mathfrak{T}^{(K)}}(\mathfrak{T}^{(K)})}$ as $\lambda_{\frak{T}}^{(K)}$.

 It is also interesting to note that $\mathfrak{T}^{(K)}$ has no point of degree more than 3, indeed, by Theorem 4.6 in~\cite{DuLG}, it is known that $\Xi$-a.s.~for any $x\in \mathfrak{T}$ the set $\mathfrak{T}\setminus \{x\}$ has at most three connected components.

Once again, Proposition~\ref{prop_def_process} allows us to define a Brownian motion \hfff{bcrtk}
$B^{(K)}$ in $(\frak{T}^{(K)},d_{\frak{T}},\lambda_{\frak{T}}^{(K)})$, where $\lambda_{\frak{T}^{(K)}}$ is the Lebesgue measure in $\frak{T}^{(K)}$ normalized to be a probability measure. Also, we define the Brownian motion $B^{K-ISE}$ \hfff{bisek} on the $K$-ISE $(\phi^{(K)}(\mathfrak{T}^{(K)}),d_{\phi^{(K)}(\mathfrak{T}^{(K)})},\lambda_{\phi^{(K)}(\mathfrak{T}^{(K)})})$. It can be shown (in essence equation (8.3) of~\cite{Croydon_arc}) that

\begin{proposition} \label{prop_approx_bisek}
We have that $B^{K-ISE}$ converges to $B^{ISE}$ as $K\to \infty$, in distribution  in the topology of uniform convergence (over compact sets) in $C(\R_+,\R^d)$ for $M \otimes (\lambda^{\mathfrak{T}})^{\otimes \N}$-a.e.~realization of $(\mathfrak{T},d_{\mathfrak{T}},\phi_{\mathfrak{T}},(U_i)_{i\in \N})$.
\end{proposition}

\section{Abstract convergence theorem}\label{sect_abstract_cvg_thm}

\subsection{Construction of the skeleton of a graph}\label{sect_constr_tnk}

\subsubsection{Decomposing the graph along cut-points}

Let $G$ be a rooted finite graph that is connected. 

\begin{definition}\label{def_cut_point}
We call cut-bond any edge $e\in E(G)$ whose removal disconnects G. By definition only one of the endpoints of a cut-bond is connected to the root and any such point is called a cut-point. \end{definition}
We denote $V_{\text{cut}}(G)$ the set of cut-points of $G$, which we assume  to be non-empty.

Let us now consider a sequence  $(x_i)_{i\in \N}$ of points $V_{\text{cut}}(G)$. Fix $K\in \N$, we construct the graph $G(K)$ in the following manner
\begin{enumerate}
\item the vertices of $G(K)$ are the set of all cut-points that lie on a path the root to an $x_i$ for $i\leq K$,
\item  two vertices of $G(K)$ are adjacent if there exists a path connecting them which does not use any cut-point.
\end{enumerate}

The new graph $G(K)$ will be rooted at $\text{root}^*$ \hfff{rootstar} which is the first cut-point on the path from the root to $x_0$.

It is elementary to notice that this graph is composed of complete graphs glued together, indeed the removal of all cut-bonds in $G$ results in a graph several connected components. Those connected components with the cut-bond that link then are called bubbles. All cut-points corresponding to cut-bonds with at least one end in the same bubble are inter-connected.

\begin{figure}
  \includegraphics[width=\linewidth]{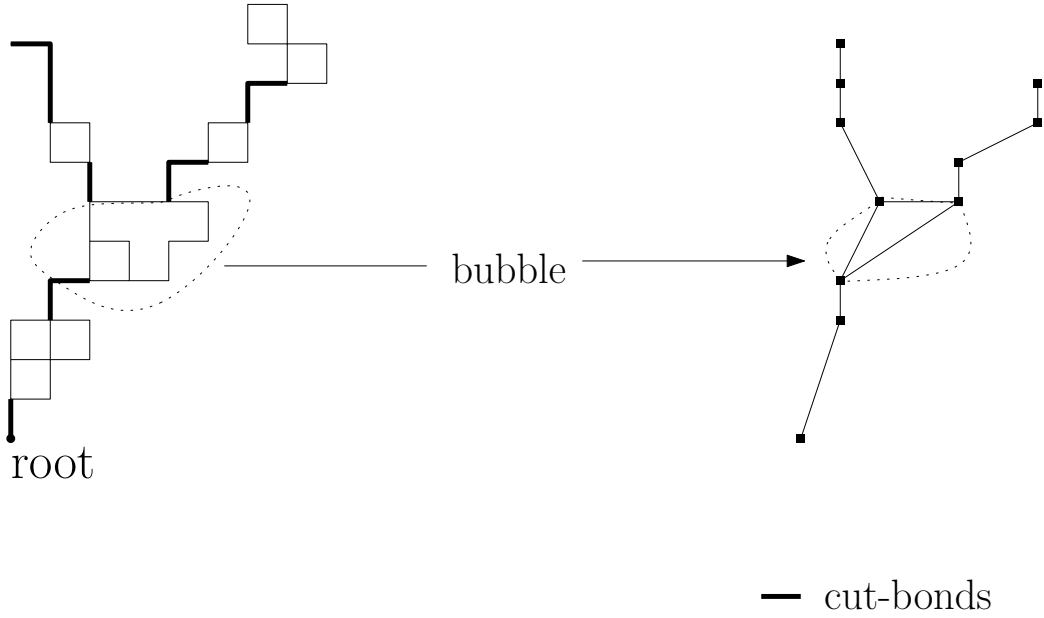}
  \caption{Construction of $G(K)$, obtained by a sequence $(x_i)_{i\in \N}$ covering the whole graph}
\end{figure}

\begin{definition}\label{def_thin}
We will say that a graph $G(K)$ is thin if it does not contain any subgraph that is a complete graph apart from segments and triangles. 
\end{definition}

We would typically expect $G$ to be thin if $G$ has many cut-bonds and the random variables $V_i$ are uniformly distributed. 

\subsubsection{Approximating a thin graph by a graph spatial tree}\label{sect:skeleton}

Let us assume that $G(K)$ is thin. We are now going to perform a technical operation, that will be helpful to complete our proofs. In essence we are trying to build a graph spatial tree that will approximate $G(K)$ well.

We want to turn the triangles  present in $G(K)$ into stars in order to turn out thin graph into a tree, this procedure will add one point for every triangle present in the graph.
 
\vspace{0.5cm}

{\it Step 1: Turning $G(K)$ into a tree $\T^{(G,K)}$}
 
\vspace{0.5cm}

 For every triangle $(x,y),(y,z),(z,x)\in E(G(K))$, we remove the edges $(x,y),(y,z),(z,x)$ and we introduce a new vertex $v_{x,y,z}$ and new edges $(x,v_{x,y,z})$, $(y,v_{x,y,z})$, $(z,v_{x,y,z})$.   We denote $\T^{(G,K)}$, the tree obtained by this construction. 
 
 We denote $V(\T^{(G,K)})$ the vertices of $\T^{(G,K)}$ and $V^*(\T^{(G,K)})$ the vertices which are not of the form $v_{x,y,z}$ (which are actually the vertices of $G(K)$). 
 
 Similarly, we denote $E(\T^{(G,K)})$ the edges of $\T^{(G,K)}$ and $E^*(\T^{(G,K)})$ \hfff{estar} the edges which are not of the form $(x,v_{x,y,z}),(y,v_{x,y,z}),(z,v_{x,y,z})$.

Finally, for $x,y\in V^\ast(\T^{(G,K)})$, we write $x\sim^*y$ if there exists no $z\in  V^*(\T^{(G,K)})$ which lies on the path from $x$ to $y$. This means that $x$ and $y$ were neighbours before the star-triangle transformation, or equivalently that they are connected by a bubble (see for example $x$, $y$ and $z$ in Figure 2).

\begin{figure}
  \includegraphics[width=\linewidth]{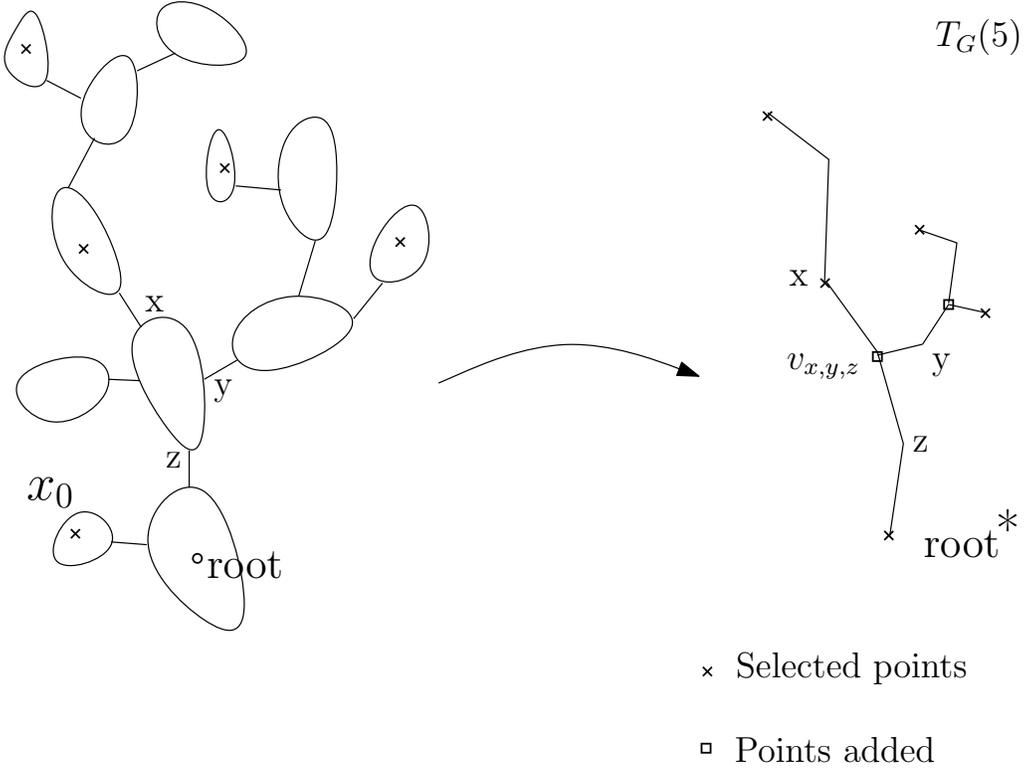}
  \caption{Construction of $T^{(G,5)}$ from a graph $G$. Even though the bubble of $x$, $y$ and $z$ has four neighbouring bubbles, the graph is still thin because the bubble on the left does not lead to a selected point $x_i$}
\end{figure}

 Since the tree $\T^{(G,K)}$ is rooted (at $\text{root}^*$) it comes with a natural notion of ancestry. For $x\in \T^{(G,K)}$, we denote $\overrightarrow{\T^{(G,K)}_{x}}$, the set of points of $\T^{(G,K)}$ which are descendants of $x$.

\vspace{0.5cm}

{\it Step 2: Turning $\T^{(G,K)}$ into a real tree by adding a metric}

\vspace{0.5cm}

The tree $\T^{(G,K)}$ comes with a natural metric by setting
\begin{enumerate}
\item for $(x,y)\in E^*(\T^{(G,K)})$, we set $d_{\T^{(G,K)}}(x,y)=d_{G}(x,y)$,
\item for any triple of edges $(x,v_{x,y,z}),(y,v_{x,y,z}),(z,v_{x,y,z})$, where $x$ is the ancestor of $y$ and $z$, we set \[d_{\T^{(G,K)}}(x,v_{x,y,z})=\frac{d_G(x,y)+d_G(x,z)-d_G(z,y)}{2},\] \[d_{\T^{(G,K)}}(y,v_{x,y,z})=\frac{d_G(x,y)+d_G(y,z)-d_G(x,z)}{2}\] and \[d_{\T^{(G,K)}}(z,v_{x,y,z})=\frac{d_G(x,z)+d_G(y,z)-d_G(x,y)}{2}.\]
Note that this assignment of distances keeps consistency in the sense that \[d_G(x,y)=d_{\T^{(G,K)}}(x,v_{x,y,z})+d_{\T^{(G,K)}}(y,v_{x,y,z}),\] \[d_G(x,z)=d_{\T^{(G,K)}}(x,v_{x,y,z})+d_{\T^{(G,K)}}(z,v_{x,y,z})\] and \[d_G(y,z)=d_{\T^{(G,K)}}(y,v_{x,y,z})+d_{\T^{(G,K)}}(z,v_{x,y,z}).\]
\item the distance grows linearly along an edge. 
\end{enumerate}

Our choice for the distances in the second part is arbitrary but it will not have an significant impact on our proof. It can be noted that this distance conserves the distance from $\text{root}^*$ to any point in $V^*(\T^{(G,K)})$.

\vspace{0.5cm}

{\it Step 3: Assigning a spatial location to the points in $\T^{(G,K)}$}

\vspace{0.5cm}

Finally we want to view our tree as a spatial tree embedded in $\R^d$, i.e.~we want to find an embedding of the edges into $\R^d$.

 Any vertex of $V^*(\T^{(G,K)})$ is assigned its original location in $G$. Moreover the vertices $v_{x,y,z}$ are mapped to the barycenter of  $x$, $y$ and $z$. We write $\phi_{G(K)}$ this map.
 
  If $(x,y)\in E(\T^{(G,K)})$, then the point $z\in[x,y]$ which is at a $d_{\T^{(G,K)}}$-distance $\alpha d_{\T^{(G,K)}}(x,y)$ along the edge $(x,y)$ is mapped to the point which is at distance $\alpha d_{\Z^d}(\phi_{G(K)}(x),\phi_{G(K)}(y))$ along the $\R^d$-geodesic between $\phi_{G(K)}(x)$ and $\phi_{G(K)}(y)$. This extends  $\phi_{G(K)}$ to a map from $\T^{(G,K)}$ to $\R^d$. 
  
  In particular the notation $\phi_{G(K)}(e)$, for $e\in E(\T^{(G,K)})$, corresponds to a segment of $\R^d$.

\vspace{0.5cm}

\subsubsection{A  natural resistance metric on the skeleton}

Let us now endow $\T^{(G,K)}$ with a resistance metric. We refer the reader to~\cite{lyons2005probability} for a background on resistances, time reversibility and electrical network theory which are central notions for the remainder of the paper.

First, for all $(x,y)\in E^*(\T^{(G,K)})$, we set $\reff^{\T^{(G,K)}}(x,y)$ as the effective resistance between $x$ and $y$ in the graph $G$ (where edges in $G$ have resistance $1$).

Let us consider a triangle $(x,y),(y,z),(z,x)\in E(G(K))$. We denote $P_x^G$ the law of a simple random walk on the graph $G$ started at $x$. For any set $A\in V(G)$, let 
\begin{equation}\label{eq:defofTT+}
T_A=\inf\{l\geq0: X_l\in A\}\quad \text{ and } \quad T^+_A=\inf\{l>0: X_l\in A \}.
\end{equation}
If $A=\{x\}$ we write $T_x,T^+_x$ instead of $T_{\{x\}}$, $T^+_{\{x\}}$.
 We set
\[\reff^{G(K)}(x,y)^{-1}:=\pi(x)P_{x}^G[T_y<T_z\wedge T^+_x]\]
where $\pi$ is number of neighbors of $x$ in $G$, which is the invariant measure associated to unit resistances on $G$ (note that this procedure defines the resistances of all three edges corresponding to a triangle).
\[\reff^{G(K)}(x,z)^{-1}:=\pi(x)P_{x}[T_z<T_y\wedge T^+_x]\]
and
\[\reff^{G(K)}(y,z)^{-1}:=\pi(y)P_{y}[T_z<T_x\wedge T^+_y].\]

We could have defined the corresponding quantities $\reff^{G(K)}(y,x)$, $\reff^{G(K)}(z,x)$ and $\reff^{G(K)}(z,x)$ in an analogous way and, by time reversibility we would have obtained that $\reff^{G(K)}(x,y)=\reff^{G(K)}(y,x),\reff^{\omega_n}(x,z)=\reff^{G(K)}(z,x)$ and $\reff^{G(K)}(y,z)=\reff^{G(K)}(z,y)$.

Next, it only remains to define the values of $\reff$ for the edges of $\T^{(G,K)}$ containing the artificial vertex $v_{x,y,z}$ using the star-triangle transformation. That is
\[\reff^{\T^{(G,K)}}(x,v_{x,y,z})=\frac{\reff^{G(K)}(x,y)\reff^{G(K)}(x,z)}{\reff^{G(K)}(x,y)+\reff^{G(K)}(y,z)+\reff^{G(K)}(z,x)},\]
\[\reff^{\T^{(G,K)}}(y,v_{x,y,z})=\frac{\reff^{G(K)}(x,y)\reff^{G(K)}(y,z)}{\reff^{G(K)}(x,y)+\reff^{G(K)}(y,z)+\reff^{G(K)}(z,x)}\]
and
\[\reff^{\T^{(G,K)}}(z,v_{x,y,z})=\frac{\reff^{G(K)}(x,z)\reff^{G(K)}(y,z)}{\reff^{G(K)}(x,y)+\reff^{G(K)}(y,z)+\reff^{G(K)}(z,x)}.\]

Taking $x,y \in V(\T^{(G,K)})$, we know that there is a unique simple path $x_0,\ldots, x_l$ in $V(\T^{(G,K)})$ from $x$ to $y$ and this allows us to  set $\reff^{\T^{(G,K)}}(x,y)=\sum_{i=0}^{l-1} \reff^{\T^{(G,K)}}(x_i,x_{i+1})$. This is the natural definition in view of the law of resistance in series and the fact that points in $V^*(\T^{(G,K)})$ are cut-points.

The resistance defines a metric on $V(\T^{(G,K)})$ that we denote $d^{\text{res}}_{\T^{(G,K)}}$. 

\begin{remark}Those definitions were chosen so as to have the following property: for all $x,y \in V^*(\T^{(G,K)})$, we have
\[
\reff^{G}(x, y)=\reff^{\T^{(G,K)}}(x, y),
\]
a fact which is proved in~\cite{BCFa}.
\end{remark}


\subsubsection{Adding a measure associated to the volume of the graph}\label{sect_mu}

We are going to add a measure to our graph $\T^{(G,K)}$.

For any $x\in G$, let $\pi^{(G,K)}(x)$ be the unique $v\in V^*(\T^{(G,K)})$ separating $x$ from the origin and such that for any $v'\in V^*(\T^{(G,K)})$ with $v'$ separating $x$ from the origin and $v'\neq v$ we have that $v' \prec v$. That is, when going from $\text{root}^*$ to $x$, the point $\pi^{(G,K)}(x)$ is the last cut-point crossed before reaching $x$. In the case where $x$ is not separated from the origin by a cut-bond, i.e.~$x$ is in the bubble of the origin, then we set $\pi^{(G,K)}(x)=\text{root}^*$ by convention.
.

Now for $x\in V^*(\T^{(G,K)})$, let $v_{\T^{(G,K)}}(x):=\#\{(y,z)\in E(G): \pi^{(G,K)}(y)=x \text{ and } y\neq x\}$. We use this to define a measure on $V(\T^{(G,K)})$.
\[
\mu^{(G,K)}:=\sum_{x\in V^*(\T^{(G,K)})} v_{\T^{(G,K)}}(x)\delta_x.
\]

\subsubsection{Another way of viewing $\T^{(G,K)}$ as graph spatial tree}\label{sect_mathfrak}

For our future purpose it will be convenient to be able to introduce a reduced version of $\T^{(G,K)}$ where we view it as a graph spatial tree with a number of vertices between $K+1$ and $2K+1$ (whereas $\T^{(G,K)}$ typically has a high number of vertices if $G$ is large). This distinction will be important for the Definition~\ref{def_condG}.

It will be a graph spatial tree  denoted $(\mathfrak{T}^{(G,K)},d_{\mathfrak{T}^{(G,K)}},\phi_{\mathfrak{T}^{(G,K)}})$ which is obtained by a  procedure similar to that of Section~\ref{sect_graph_tree}. In the notations of that section this spatial graph is $((\T^{(G,K)})^{K,(x_i)},d_{(\T^{(G,K)})^{K,(x_i)}}, \phi_{(\T^{(G,K)})^{K,(x_i)}})$.

In words this simply means we restrict the tree structure $\T^{(G,K)}$ to the subgraph obtained from the points $\text{root}^*$, $x_0,\ldots,x_K$ and the branching points that these points created. Hence, the vertices of $\mathfrak{T}^{(G,K)}$ are 
\[
V(\mathfrak{T}^{(G,K)}):=\{\text{root}^*,x_0,\ldots,x_K\} \cup (V(\T^{(G,K)})\setminus V^*(\T^{(G,K)})),
\]
 (the set on the right-hand side of the union being the branching points) and the edges $E(\mathfrak{T}^{(G,K)})$ are the ones induced by the tree structure of $\T^{(G,K)}$.

\begin{remark}\label{rem_not_depend} It is important to note that the distance, the resistance distance and the embedding we assign to $\mathfrak{T}^{(G,K)}$ coincide with those assigned to $\T^{(G,K)}$. This will allow us to use, e.g., $d_{\T^{(G,K)}}$ to signify $d_{\mathfrak{T}^{(G,K)}}$.\end{remark}

\subsection{Setting for the abstract theorem} \label{sect_model_abstract}


In this section, we will consider a sequence of random graphs $(G_n)_{n\in \N}$ chosen under a measure ${\bf P}_n$, which in practice will be  large critical structures. The $n$ will quantify the order of the volume of $G_n$, which in turn means (in the universality class we are interested in) that the intrinsic distances between points in $G_n$ is of the order of $n^{1/2}$ and the extrinsic distances are of the order $n^{1/4}$. Our eventual goal is to study $(X^{G_n}_m)_{m\in \N}$ which is the simple random walk on $G_n$.

Our construction will rely on a sequence of i.i.d.~random variables $(V_i^n)_{i\in \N}$ supported on cut-points. In practice this sequence should be asymptotically close to uniform random variables.

\begin{definition}\label{def_augmented} \hfff{gv} For $n\in \N$, we say that $(G_n,(V_i^n)_{i\in \N})_{n\in \N}$ is a sequence of random augmented graphs. \end{definition}

Fix $K\in \N$. If the graph $G_n(K)$ constructed from $(G_n,(V_i^n)_{i\in \N})_{n\in N}$ is thin, then the construction of the skeleton of the previous section can be carried out. In order to lighten the notations, we will write \hfff{mass} $G^{(n,K)}$, $\T^{(n,K)}$, $\mathfrak{T}^{(n,K)}$, $V^*(\T^{(n,K)})$, $\phi^{(n,K)}$, $v^{(n,K)}$, $\reff^{(n,K)}$ and $\pi^{(n,K)}$ for $G_n(K)$, $\T_{G_n(K)}$, $\mathfrak{T}^{(G_n,K)}$, $V^*(\T^{(G_n,K)})$, $\phi_{G_n(K)}$,  $v_{\T_{G_n(K)}}$, $\reff^{\T^{(G,K)}}$ and $\pi^{(G_n,K)}$ and we also introduce the rescaled quantities $d^{(n,K)}(\cdot,\cdot)$, $d^{(n,K)}_{\text{res}}(\cdot,\cdot)$ and $\mu^{(n,K)}$ for $n^{-1/2}d_{\T_{G_n(K)}}(\cdot,\cdot)$, $n^{-1/2}d_{\T_{G_n(K)}}^{\text{res}}(\cdot,\cdot)$ and $n^{-1}\mu^{(G_n,K)}$. All those quantities were defined in Section~\ref{sect_constr_tnk}.

We recall that $d^{(n,K)}(\cdot,\cdot)$ (resp.~$\phi^{(n,K)}$) is a distance on (resp.~embedding of) $\mathfrak{T}^{(n,K)}$ because of Remark~\ref{rem_not_depend}.


\subsubsection{Asymptotically thin graphs}

For any $x\in V^*(T_n(K))$, we call $K$-sausage of $x$ the set $\{y\in G_n,\ \pi^{(n,K)}(y)=x\}$.

  \begin{figure}
  \includegraphics[width=\linewidth]{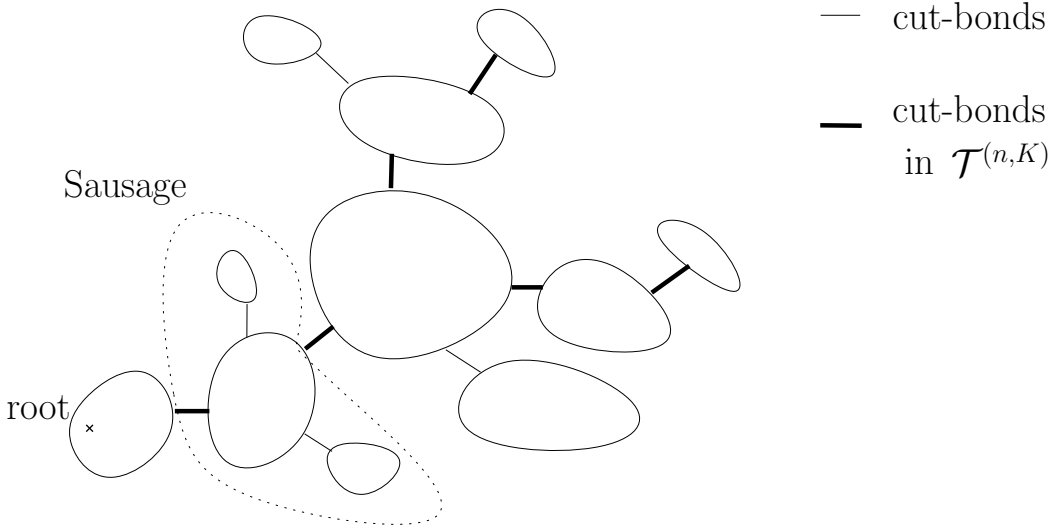}
  \caption{The sausage corresponding to a specific cut-bond.}
\end{figure}

 Note that a sausage is typically much large than the corresponding bubble because it also contains bubbles of $G_n$ which are not in $T_n(K)$. We introduce\hfff{delta1}
\begin{equation}\label{eq:defofdelta}	
\Delta^{(n,K)}_{\Z^d}:=\max_{x\in V^*(T_n(K))} \text{Diam}_{\Z^d}(\{y\in G_n,\ \pi^{(n,K)}(y)=x\}),
\end{equation}
where $\text{Diam}_{\Z^d}(A):=\max\{d_{\Z^d}(x,y):x,y\in A\}$, for any $A\subset \Z^d$.
We also introduce\hfff{delta2}
\begin{equation}\label{eq:defofdeltaintr} 
\Delta^{(n,K)}_{G_n}:=\max_{x\in V^*(\T_n^{(n,K)})} \text{Diam}_{G_n}(\{y\in G_n,\ \pi^{(n,K)}(y)=x\}),
\end{equation}
where $\text{Diam}_{G_n}(A):=\min\{d_{G_n}(x,y):x,y\in A\}$ for any $A\subset \Z^d$, and $d_{G_n}$ is the graph distance in $G_n$. 
In our context, we want to extend the definition of thin graph (see Definition~\ref{def_thin}).
\begin{definition}\label{def_athin}
We say that a sequence of random augmented graphs $(G_n,(V_i^n)_{i\in \N})_{n\in N}$ is asymptotically thin if  
\begin{enumerate}
\item for all $K\in \N$, we have
\[
\lim_{n\to \infty} {\bf P}_n[G^{(n,K)}\text{ is thin}]=1.
\]
\item for all $\epsilon>0$, we have
\[ 
\lim_{K\to \infty} \sup_{n\in \N} {\bf P}_n\Bigl[n^{-1/4}\Delta^{(n,K)}_{\Z^d} >\epsilon\Bigr]=0
\]
and
\[ 
\lim_{K\to \infty} \sup_{n\in \N} {\bf P}_n\Bigl[n^{-1/2}\Delta^{(n,K)}_{G_n} >\epsilon\Bigr]=0.
\]

\end{enumerate}
\end{definition}

A sequence of random graphs will typically be asymptotically thin if they contain, in the limit, a dense set of cut-points. This hypothesis should be verified in all models that are in the universality class of the ISE. 

The first part of the definition states that it is unlikely for four larges branches to emanate from the same bubble. It should be noted that in the ISE this property is verified (see Section~\ref{sect_KISE}). The second part of the condition is related to the fact that there are no parts of $G_n$ that have macroscopic ($\Z^d$ and intrinsic) length but where there are no cut-points.

A key property of asymptotically thin graphs is that the $K$-skeleton is, for large $K$, a good approximation of the whole graph as proved in the companion paper~\cite{BCFa}.

\begin{remark}\label{rem_abuse_tnk} If a sequence $(G_n,(V_i^n)_{i\in \N})_{n\in \N}$ is asymptotically thin, then the notation $\T^{(n,K)}$ and $\mathfrak{T}^{(n,K)}$ make sense with probability going to $1$ since these objects can be constructed with the methods of Section~\ref{sect_constr_tnk}. The conditions which will involve $\T^{(n,K)}$ and $\mathfrak{T}^{(n,K)}$ (condition $(G)$ of definition~\ref{def_condG} and condition $(V)$ of definition~\ref{def_condV}) are all asymptotical in $n$. Hence, they are not affected by the fact that $\T^{(n,K)}$ is not defined on an event of small probability. We will thus allow ourselves a slight abuse of notation in the statement of these conditions. \end{remark}

\subsection{Conditions for convergence towards $B^{ISE}$}

We are now going to introduce three conditions that are the central hypotheses for our abstract theorem.

\subsubsection{Condition (G): asymptotic shape of the graph}\label{sect_condG}

Let us define a distance $D$ on graph spatial trees (defined in Section~\ref{sect_graph_tree}). Here, we follow Section 7 of~\cite{Croydon_arc}.

For $(T,d,\phi)$ a graph spatial tree, write $T^*$ for the shape of the tree ($T$ with a prescribed ordering) and $\abs{e_1},\ldots,\abs{e_l}$ for the lengths of the edges (according to the lexicographical order of $T^*$).

Take two graph spatial trees $(T,d,\psi)$ and $(T',d',\psi ')$. If $T^*\neq T'^*$, then  we set $d_1(T,T')=\infty$ and otherwise we set
\begin{equation}\label{eq:defofd1}
d_1(T,T'):=\sup_{i} \abs{\abs{e_i}-\abs{e_i'}}.
\end{equation}

Now if $T^*=T'^*$, we have a homeomorphism $\Upsilon_{T,T'}:\psi(T)\to \psi(T')$ such that  if $x\in \psi(T)$ is at a distance $\alpha \abs{e}$ along the edge $e$ is mapped to the point $x'\in \psi'(T')$ which is at distance $\alpha \abs{e'}$ along the corresponding edge $d'$. We then set 
\[
d_2(T,T'):=\sup_{x\in \psi(T)} d_{\R^d}(\psi(x), \psi'(\Upsilon_{T,T'}(x))).
\]
This yields a metric 
\begin{equation}\label{eq:defofD}
D((T,d,\psi),(T',d',\psi ')):=(d_1(T,T')+d_2(T,T'))\wedge 1
\end{equation}
 on graph spatial trees. The importance of this metric stems from the following result, see Lemma 7.1.~in~\cite{Croydon_arc} (recall that $\lambda_T$ is the normalized Lebesgue measure on the tree).

\begin{proposition}\label{prop_metrictoBrownian}
Suppose $((T_n^{'},d_n^{'},\psi_n^{'}))_{n\in \N}$ is a sequence of graph spatial trees that converge with respect to the metric $D$ to a graph spatial tree $(T^{'},d^{'},\psi^{'})$. For each $n$, let $B^n$ be the Brownian motion on the spatial tree $(T_n^{'},d_n^{'},\lambda_{T_n^{'}})$ and let $B$ be the Brownian motion on the spatial tree $(T^{'},d^{'},\lambda_{T^{'}})$ (defined with Proposition~\ref{prop_def_process}). Then $(\psi_n(B^n))_{n\in \N}$ converges to $(\psi(B))$ in distribution for the topology of uniform convergence on compact sets.
\end{proposition}

This distance allows us to define our first condition (relevant definitions can be found at Definition~\ref{def_augmented}, Definition~\ref{def_athin} and Remark~\ref{rem_abuse_tnk}).

\begin{definition}\label{def_condG} Condition (G): We say that a sequence of asymptotically thin random augmented graphs $(G_n,(V_i^n)_{i\in \N})_{n\in N}$ satisfies condition $(G)_{\sigma_d,\sigma_{\phi}}$ if there exists $\sigma_d, \sigma_{\phi}>0$ such that for all $K\in \N$, the sequence of graph spatial trees $((\mathfrak{T}^{(n,K)},d^{(n,K)}(\cdot,\cdot),n^{-1/4}\phi^{(n,K)}))_{n\in \N}$ converges weakly to a rescaled $K$-ISE $(\mathfrak{T}^{(K)},\sigma_{d} d_{\mathfrak{T}^{(K)}},\sigma_{\phi} \sqrt{\sigma_{d} } \phi_{\mathfrak{T}^{(K)}})$ in the topology induced by $D$.
\end{definition}

This condition states that for all $K\in \N$, the $K$-skeleton of our sequence of random graphs resembles the $K$-skeleton of the ISE.
The constant $\sigma_d$ is the ratio between the rescaled graph distance in $G_n$ and the canonical distance in the CRT. The constant $\sigma_{\phi}$ is the diffusivity of the embedding per unit of graph-distance. 

\begin{remark} It is very important to stress that in condition (G), the topology induced by $D$ imposes a condition on the convergence of the distances of only a finite number of points (between $1+K$ and $1+2K$). This is where the distinction between $\mathfrak{T}^{(n,K)}$ and $\T^{(n,K)}$ makes a big difference. \end{remark}

\begin{remark} The requirement that $T$ and $T'$ have the same shape seems to be very strong, but we believe that in practice this condition will always be verified. We recall that we only require that the skeleton trees, as in Figure 2, (which a finite number of points between $1+K$ and $1+2K$) have the same shape when $n$ is large. \end{remark}

\subsubsection{Condition (V): Uniform distribution of the volume}

We introduce $\lambda^{(n,K)}$  \hfff{lambda} the Lebesgue measure on $(\T^{(n,K)},d^{(n,K)})$, normalized to have total mass $1$. This measure is well defined because $\T^{(n,K)}$ is  a graph spatial tree which can always be equipped with a Lebesgue probability measure (see Section~\ref{sect_graph_tree}). This naturally induces a Lebesgue probability measure on $\mathfrak{T}^{(n,K)}$ (see Section~\ref{sect_mathfrak} and Remark~\ref{rem_not_depend}).

 Recall that $\overrightarrow{\T^{(n,K)}_{x}}$ are the descendants of $x$ (including $x$ itself) in $\T^{(n,K)}$ \hfff{arrowt} and the definition of $\mu^{(n,K)}$ in Section~\ref{sect_mu}.  Our second condition states that

\begin{definition}\label{def_condV} Condition (V): We say that a sequence of asymptotically thin random augmented graphs  $(G_n,(V_i^n)_{i\in \N})_{n\in N}$ satisfies condition $(V)_{\nu}$ if  there exists $\nu>0$ such that for $\epsilon>0$
\[
\lim_{K\to \infty}\limsup_{n\to \infty} {\bf P}_n \Bigl[\sup_{x\in V(\T^{(n,K)})} \abs{\nu\lambda^{(n,K)}\Bigl(\overrightarrow{\T^{(n,K)}_{x}}\Bigr)-\mu^{(n,K)}\Bigl(\overrightarrow{\T^{(n,K)}_{x}}\Bigr)}>\epsilon\Bigr] =0.
\]
\end{definition}

Intuitively this condition says that the volume of the graph is asymptotically uniformly distributed over the graph and that the total volume is of order $\nu n$. 

\subsubsection{Condition (R): the linearity of the resistance}

\begin{definition} Condition (R): We say that a sequence of asymptotically thin random augmented graphs $(G_n,(V_i^n)_{i\in \N})_{n\in N}$ satisfies condition $(R)_{\rho}$ if there exists $\rho>0$ such that for all $\epsilon>0$ and for all $i\in \N$
\[
\lim_{n\to \infty} {\bf P}_n\Bigl[ \abs{\frac{R^{G_n}(0,V_i^n)}{d^{G_n}(0,V_i^n)}-\rho}>\epsilon\Bigr]=0.
\]
\end{definition}

Intuitively this condition says that the resistance distance is asymptotically almost proportional to the graph distance.

\subsection{Statement of the abstract convergence theorem and application}\label{sect_abstract_thm}

We are finally able to state our abstract convergence theorem
\begin{theorem}\label{thm_abstract}
Consider a sequence of asymptotically thin random graphs $(G_n,(V_i^n)_{i\in \N})_{n\in N}$ chosen under ${\bf P}_n$ which verifies conditions $(G)_{\sigma_d,\sigma_{\phi}}$, $(V)_{\nu}$ and $(R)_{\rho}$. Denoting $(X^{G_n}_m)_{m\in \N}$ the simple random walk on $G_n$ started at $0$, we have that 
\[
( n^{-1/4}X_{tn^{3/2}}^{G_n})_{t\geq 0} \to (\sigma_{\phi}\sqrt{\sigma_d} B^{ISE}_{(\rho \nu \sigma_d)^{-1}t})_{t\geq 0}.
\]
The convergence is annealed and occurs in the topology of uniform convergence over compact sets.
\end{theorem}

Let us explain how to use this theorem in the case of simple random walks on large critical branching random walks (defined in Section~\ref{sect_model}). 

\begin{definition}\label{defpi}
 We define a function $\pi_n(x)$ for  $x\in \phi_{\T^{\text{GW}}_n}(\T^{\text{GW}}_n)$ and which is such that $\pi_n(x)$ is the first cut-point encountered on any path from $x$ to $0$.  \hfff{pin}
 
 This function is not well defined for vertices $x$ which can be reached from $0$ without crossing any cut-point. For such vertices, we extend the notation $\pi_n(x)$ by arbitrarily setting $\pi_n(x)=z$ where $z$ is the cut-point which is closest to the origin in the metric induced by $\phi_{\T^{\text{GW}}_n}(\T^{\text{GW}}_n)$ (in case there are several such points chose one according to some predetermined order on $\Z^d$). \end{definition}

\vspace{0.2cm}

{\it Our choice of augmented random graphs}

\vspace{0.2cm}

 The sequence of augmented random graphs we will use to apply Theorem \ref{thm_abstract} is $(\omega_n,(\pi_{n}(\phi_{\T^{\text{GW}}_n}(U^n_i)))_{i\in\N})_{n\in \N}$ where $(U_i^n)_{i\in \N}$ is a sequence i.i.d.~random variables uniformly distributed on $\T^{\text{GW}}_n$. This leaves us with four conditions to verify.
 
{\it Asymptotic thinness}

\vspace{0.2cm}

The asymptotic thinness follows from estimates on the CRT and tail estimates for the distance between cut-points (see Section~\ref{sect_be_brw} in the Appendix).  Our precise statement can be found in Lemma~\ref{lem:asymthin}.

\vspace{0.2cm}

{\it Condition (G)}

\vspace{0.2cm}

The proof of Condition $(G)$ essentially follows from estimates in~\cite{JM}. Our precise statement is available at Lemma~\ref{lem:condG} in which we prove $(G)_{\sigma\sigma_G,\sigma_G^{-1/2}}$ where $\sigma$ and $\sigma_G$ are respectively defined in Theorem~\ref{thm:mainannealed} and in~\eqref{def_sigmaG}.

\vspace{0.2cm}

{\it Condition (V)}

\vspace{0.2cm}

The proof of Condition $(V)$ is very closely related to a recent result in~\cite{LGL}.
\begin{theorem}\label{teo:legalllin}
Assume ${\bf E}[Z^2]<\infty$ and $d\geq 5$. Let $v_1,\dots v_n$ be the vertices of $\T^{\text{GW}}_n$ labeled in increasing lexicographical order. Let $\text{vol}_n(j)=\{\phi_{\T^{\text{GW}}_n}(v_i),\ i=1,\dots, j\}$. Then, there exists $\nu'>0$ such that for all $a\in [0,a]$, we have that
\[
\frac 1n \abs{\text{vol}_n(\lfloor an\rfloor)} \to_{n\to \infty} a \nu',
\]
in probability, where $\omega_n$ is chosen under ${\bf P}_n$.
\end{theorem}

Our precise statement is Proposition~\ref{lem:condV}, which proves $(V)_{\nu}$ where $\nu$ is defined at~\eqref{def_nu}.

\vspace{0.2cm}

{\it Condition (R)}

\vspace{0.2cm}

Finally condition $(R)$ is proved in this paper and constitutes an important part of the proof (see Section~\ref{sect_resistance}). The method used for proving condition $(R)$ in this context has some chance of being adapted to more complex models although this would require significant work. 

Our precise statement for condition $(R)$ resembles the following 
\begin{theorem}\label{theorem_res}
Assume ${\bf E}[Z^2]<\infty$ and $d>14$. Then there exists $\rho>0$ such that for all $\epsilon>0$
\[
{\bf P}_n\Bigl[\abs{\frac{R^{\omega_n}(0,\phi_{\T^{\text{GW}}_n}(U_1^n))}{d^{\omega_n}(0,\phi_{\T^{\text{GW}}_n}(U_1^n))}-\rho}>\epsilon\Bigr]\to 0,
\]
where $U_1^n$ is a uniformly chosen point in $\T^{\text{GW}}_n$ where $\rho$ is defined at~\eqref{def_rho}.
\end{theorem}

The precise statement of $(R)_{\rho}$ is Proposition~\ref{prop_res0}.

\vspace{0.2cm}

\begin{proof}[Proof of Theorem~\ref{thm:mainannealed}]

We refer to the definitions of $\rho_1$, $\rho$, $\sigma_G$, $\nu$ and  $\sigma$ at~\eqref{def_rho1}, \eqref{def_rho}, \eqref{def_sigmaG}, \eqref{def_nu} and Theorem~\ref{thm:mainannealed} respectively. We can see that $\rho=\rho_1\sigma_G^{-1}$.

The results stated in this section imply that our model verifies asymptotic thinness, conditions $(G)_{\sigma,\sigma_G}$, $(V)_{\nu}$ and $(R)_{\rho}$. 

Hence~Theorem~\ref{thm:mainannealed} follows from the abstract convergence result (Theorem~\ref{thm_abstract}). 
\end{proof}

\section{An infinite and a bi-infinite version of critical branching random walks} \label{section_bi_inf}

In this section, we will assume that ${\bf E}[Z^2]<\infty$.

We are going to construct an infinite and a bi-infinite version of critical branching random walks. This construction will prove useful because
\begin{enumerate}
\item the bi-infinite version of critical branching random walks has good ergodic properties (see Section~\ref{sect_ergo}),
\item certain results for infinite critical branching random walks can be transferred to the case of the large critical branching walks (see Section~\ref{sect_transfer}).
\end{enumerate}

In short, these properties will allow us to obtain, by ergodicity, results similar to the law of large numbers on large critical branching random walks. This will be critical to show that the resistance distance and the usual distance are proportional.

\subsection{Definition}

In this section, it will be convenient to work under a different conditioning than ${\bf P}_n$ and transfer our results to ${\bf P}_n$ (see Section~\ref{proof_res}). For this we introduce ${\bf Q}_n:={\bf P}[ \cdot \mid H(\T^{\text{GW}})\geq n]$. The tree chosen under that measure will be denoted $\T^{\text{GW}}_n$ as it was under ${\bf P}_n$, this poses no issue since in the index $n$ is simply a notational reminder of the fact that we work with a tree $\T$ conditioned to be large.

We will also need infinite and bi-infinite critical branching random walks. In order to do this, we will simply need to define infinite (resp.~bi-infinite) critical trees. Indeed, given an infinite (countable) tree $T$, we can use the standard simple random walk embedding $\phi_T$ described in the introduction (see Section~\ref{sect_model}) to obtain spatial tree which corresponds to the branching random walk  associated to $T$.

Consider our critical Galton-Watson tree obtained from $Z$. We denote $p_k={\bf P}[Z=k]$  and we define $\tilde{Z}$ the size-biased version of $Z$ by setting ${\bf P}[\tilde{Z}=k]=kp_k$ for all $k$ (which is a probability distribution since ${\bf E}[Z]=1$).  Our assumption that ${\bf E}[Z^2]<\infty$ guarantees that ${\bf E}[\tilde{Z}]<\infty$.

We start by introducing the incipient  infinite critical Galton-Watson tree (or IICGW), denoted $\T^{\text{GW}}_{\infty}$ and the  corresponding measure will be denoted ${\bf P}_{\infty}$. This object was constructed in~\cite{kesten1986subdiffusive} as a weak limit for large critical Galton-Watson trees.  For our purpose it is important to recall Lemma 2.2 from~\cite{kesten1986subdiffusive} which gives a very useful way to construct $\T^{\text{GW}}_{\infty}$. This construction will serve as our definition.

\begin{definition}
Let us construct a two-type Galton-Watson tree with the following offspring rules
\begin{enumerate}
\item the origin being of type $A$,
\item all points of type $A$ have $1$ offspring of type $A$ and a random number of offspring of type $B$ distributed as $\tilde{Z}-1$,
\item all points of type $B$ have random number of offspring of type $B$ distributed as $Z$.
\end{enumerate}

The tree $\T^{\text{GW}}_{\infty}$ obtained is this manner is called incipient infinite critical Galton-Watson tree. It has one unique infinite line of descent (corresponding to the points of type $A$) which we call backbone. 
\end{definition}

Next, we want to introduce a bi-infinite version of the IICGW.

\begin{definition} Consider two independent IICGWs $\T^{\text{GW},(1)}_{\infty}$ and $\T^{\text{GW},(2)}_{\infty}$. Connect their roots by one edge and define the root of the resulting tree to be the root of $\T^{\text{GW},(1)}_{\infty}$. The resulting tree is called incipient incipient bi-infinite critical Galton-Watson tree (IBICGW), is denoted $\T^{\text{GW}}_{-\infty,\infty}$ with the corresponding measure being ${\bf P}^{-\infty,\infty}$. 

Similarly to the case of the IICGW, we obtain a unique bi-infinite line of descent, which we will again refer to as a backbone.
\end{definition}

\begin{figure}
  \includegraphics[width=\linewidth]{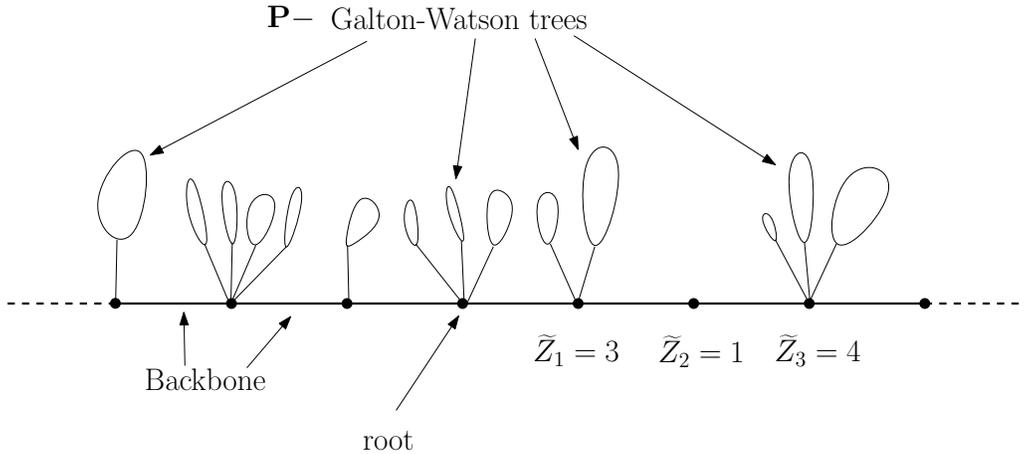}
  \caption{The IBICGW is invariant by translation by construction. By keeping only the vertices and edges on the right of the root we would have an IICGW}
\end{figure}

As we mentioned, once we constructed critical branching random walks we can obtain the corresponding branching random walks. We will use the same notation for the probability spaces of the random tree and the associated random spatial tree.

\begin{definition}
We write  ${\bf Q}_n$ (resp.~${\bf P}_{\infty}$, resp.~${\bf P}^{-\infty,\infty}$) for the probability measure associated to  the branching random walk $\omega_n$ on the Galton-Watson conditioned to have height $n$ (resp.~$\omega_{\infty}$ on the IICGW, resp.~$\omega_{-\infty,\infty}$ on the IBICGW) which we will be abbreviated LCBRW (resp.~IICBRW, resp.~IBICBRW).
\end{definition}

\begin{remark}
We will commit a small abuse of notation by using LCBRW for any conditioning of the form, $\abs{\T}\geq n$, $\abs{\T}=n$, $H(\T)\geq n$ and $H(\T)=n$.
\end{remark}

The distributions of  $\T^{\text{GW}}_{\infty}$ and $\T^{\text{GW}}_{-\infty,\infty}$ are explicit by their definition. Furthermore, the distribution of $\T^{\text{GW}}_n$ is also explicit (see Lemma 2.1 in~\cite{geiger1999elementary} for the description of the law of a tree of height $n$). This can be used to verify the following stochastic domination result
\begin{equation}\label{stoch_dom}
\T^{\text{GW}}_n \prec \T^{\text{GW}}_{\infty} \prec \T^{\text{GW}}_{-\infty,\infty} \text{ and } \omega_n\prec \omega_{\infty} \prec \omega_{-\infty,\infty}.
\end{equation}

Those stochastic dominations will be used to create a coupling between those trees and graphs. For the coupling between $\T^{\text{GW}}_{\infty}$ and $\T^{\text{GW}}_{-\infty,\infty} $ the backbones will also be coupled.

\subsection{Representation of the IICBRW and the IBICBRW} \label{sect_res_est_not}

Let us denote $\mathcal{B}(x)$ a random spatial tree obtained by
\begin{enumerate}
\item choosing modified Galton-Watson tree $\T^{\text{GW}}_{*}(x)$ where the root has $\tilde{Z}-1$ offspring, all of which reproduce as a standard critical ${\bf P}$-Galton-Watson tree, 
\item the root is mapped to $x$ and for the rest of the tree we use  the standard simple random walk embedding described in the introduction (see Section~\ref{sect_model}).
\end{enumerate} 

In the sequel $\mathcal{B}(x)$ and $\mathcal{B}(y)$ will denote, unless stated otherwise, two independent copies this process started at $x$ and $y$ respectively.

 This notation is useful because in the case of the IICBRW, the graph $\omega_{\infty}$ is composed by two parts
\begin{enumerate}
\item  the backbone which is just the trajectory of a simple random walk and we will typically denote it $(\alpha(n))_{n\geq 0}$.
\item the parts discovered by the critical Galton-Watson trees emanating from a point $\alpha(n)$ ($n\in \N$) of the backbone which has the same law as $\mathcal{B}(\alpha(n))$.
\end{enumerate}

For the case of the IBICBRW, we will use a similar decomposition with the notation $(\alpha(n))_{n\in \Z}$.

Fix $A\subset \Z$. Given, a nearest-neighbour trajectory $(x_n)_{n\in \Z}$, we denote $x_{A}$ the sequence $(x_n)_{n\in A}$. Furthermore, we define the random variable $\mathcal{B}(x_A)=\cup_{n\in A} \mathcal{B}^{(n)}(x_n)$, where $\mathcal{B}^{(k)}(x)$ are random variables with law $\mathcal{B}(x)$ that are independent in $k$.

Those notations will give us a standard representation for the IICBRW and the IBICBRW described in the following proposition (whose proof we omit).
\begin{proposition} If $(\alpha(n))_{n\in \Z}$ is a simple random walk then the set $\mathcal{B}(\alpha([0,\infty)))$ has the law of the range of a IICBRW. Here, we implicitly assume that the simple random walk is independent of the random variables $\mathcal{B}(x)$ used to defined $\mathcal{B}(\alpha([0,\infty)))$.

 Similarly, the set $\mathcal{B}(\alpha((-\infty,\infty))))$ has the law of the random of a IBICBRW.
\end{proposition}

\subsection{Ergodicity properties of the IBICBRW}\label{sect_ergo}

Let us consider $\omega_{-\infty,\infty}$ a IBICBRW, arising from a sequence $((\alpha(n))_{n\in \Z}, (\mathcal{B}^{(n)}(x))_{n\in \Z,x\in \Z^d})$. We define, for $k\in \Z$, the shift $\theta_k$ acting on $\omega_{-\infty,\infty}$ by defining
\[
\omega_{-\infty,\infty}\circ \theta_k=((\alpha(n+k)-\alpha(k))_{n\in \Z}, (\mathcal{B}^{(n+k)}(x))_{n\in \Z,x\in \Z^d}),
\]
which intuitively means that we are witnessing $\omega_{-\infty,\infty}$ after a re-centering at $\alpha(n)$. We typically write $\theta$ instead of $\theta_1$. 

Since in $((\alpha(n))_{n\in \Z}, (\mathcal{B}^{(n)}(x))_{n\in \Z,x\in \Z^d})$, $(\alpha_n)_{n\in \N}$ is a simple random walk and all $(\mathcal{B}^{(n)}(x))_{n\in \Z,x\in \Z^d})$ are independent and identically distributed (up to a spatial translation), we can easily see the following.

\begin{lemma}\label{ergo_shift0} The shift $\theta$ on $\omega_{-\infty,\infty}$ is ergodic with respect to ${\bf P}^{-\infty,\infty}$. \end{lemma}

Given a IBICBRW, we say that $\alpha(i)$ is a pivotal point if $\{\mathcal{B}(\alpha((-\infty, i]))\cap \mathcal{B}(\alpha([i+1,\infty))) \neq \emptyset\}$. This means that if we remove a pivotal point from the graph $\omega_{-\infty,\infty}$, it is split into two infinite connected components.

\begin{remark}\label{pivotal_cut} It is obvious that any pivotal point is a cut-point, since they disconnect the graph, but the converse is not true in general. \end{remark}

\begin{figure}
  \includegraphics[width=0.75\linewidth]{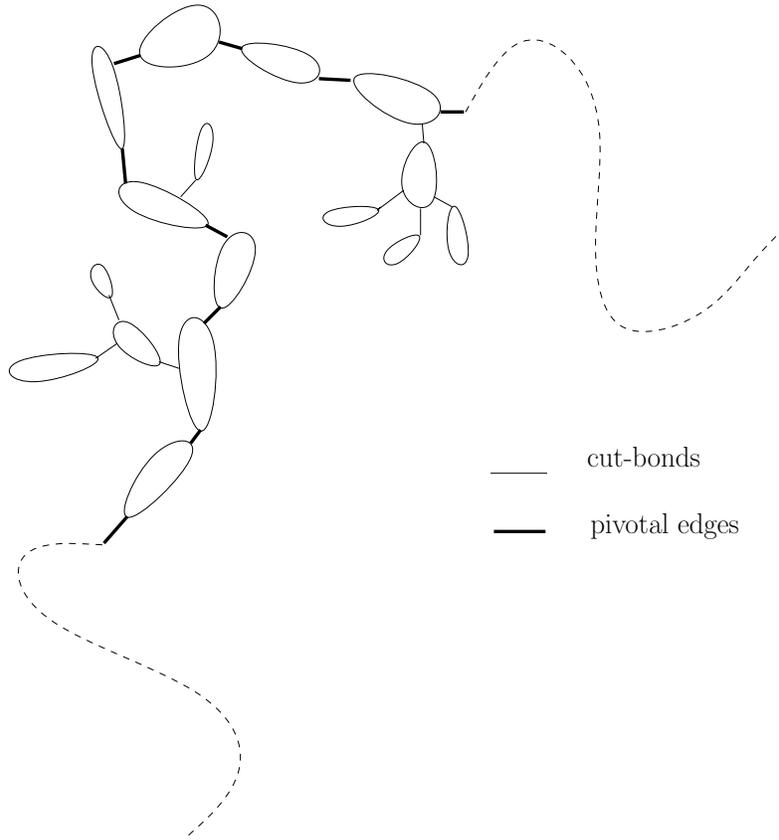}
  \caption{A representation of the IBICBRW in high dimensions.}
\end{figure}

Pivotal points exists by the following lemma whose proof is deferred to the appendix (see Section~\ref{proof_pivotal}).

\begin{lemma}\label{cutp_2}
For $d>10$, we have
\[
{\bf P}^{-\infty,\infty}[0\text{ is a pivotal point}]>0.
\]
\end{lemma}

The existence of such pivotal points is natural in light of the following estimate (which is also proved in the appendix, see Section~\ref{proof_intersect})
\begin{lemma}\label{cutp_1}
Fix $d\geq 3$. For $(\alpha(n))_{n\in \Z}$ a simple random walk started at $0$, we have,
\[
{\bf P}^{-\infty,\infty}[\mathcal{B}(\alpha((-\infty, 0)))\cap \mathcal{B}(\alpha((R,\infty))) \neq \emptyset] \leq C R^{-(d-8)/2}.
\]
\end{lemma}

Lemma~\ref{ergo_shift0} and Lemma~\ref{cutp_2} imply that there are infinitely many pivotal points with ${\bf P}^{-\infty,\infty}$-probability $1$. Hence, there is a well defined increasing sequence times of pivotal points $(\gamma_i)_{i\in \Z}$ where  $\gamma_1$ is the smallest positive (in the sense of time for $\alpha(n)$). We also define $\Theta=\theta_{\gamma_1}$, then the ergodicity of the shift $\theta$ implies the following.

\begin{lemma}\label{ergo_shift} The shift $\Theta$ on $\T^{\text{GW}}_{-\infty,\infty}$ is ergodic with respect to ${\bf P}^{-\infty,\infty}[\;\cdot \mid 0\text{ is a pivotal point}]$. \end{lemma}

We can turn previous estimates into a bound on the tail of the distance between pivotal points. This will be done using a strategy from~\cite{bolthausen2003cut}
\begin{lemma}\label{tail_gamma}
Fix $d\geq 3$. For any $i\in \Z$, we have
\[
{\bf P}^{-\infty,\infty}[\gamma_{i+1}-\gamma_i\geq n]<C (\log n)^{(d-6)/2}n^{-(d-8)/2}.
\]
\end{lemma}
\begin{proof}
Fix an integer $R\geq 1$. We write
\[
k_j=1+Rj, \text{for any $j\geq 0$}.
\]

We can see that for any $J\geq 1$, we have
\begin{align}\label{sdfsdf}
& {\bf P}^{-\infty,\infty}[\gamma_1\geq k_{2J}]\\\nonumber  \leq & \sum_{0\leq j \leq 2J+1}{\bf P}^{-\infty,\infty}\bigl[\mathcal{B}(\alpha((-\infty,k_j]))\cap \mathcal{B}(\alpha([k_{j+1}+1,\infty)))= \emptyset\bigr]+ \\\nonumber
     &  {\bf P}^{-\infty,\infty}\bigl[\text{for all $0\leq j<2J+1$, } \mathcal{B}(\alpha((-\infty,k_j]))\cap \mathcal{B}(\alpha([k_{j+1}+1,\infty)))\neq \emptyset,\ \gamma_1\geq k_{2J}\bigr]   \\\nonumber
     \leq & C (2J+1) R^{-(d-8)/2} + \\ \nonumber & {\bf P}^{-\infty,\infty}\bigl[\text{for all $0\leq j<2J+1$, } \mathcal{B}(\alpha((-\infty,k_j]))\cap \mathcal{B}(\alpha([k_{j+1}+1,\infty)))= \emptyset,\ \gamma_1\geq k_{2J}\bigr] ,
     \end{align}
     where we used Lemma~\ref{cutp_1} in the last line. 
     
     Note that whenever   $\gamma_1\geq k_{2J}$ and  $\mathcal{B}(\alpha((-\infty,k_j]))\cap \mathcal{B}(\alpha([k_{j+1}+1,\infty)))= \emptyset$ for all $0\leq j<2J+1$, then for any $j\in [1,2J]$ we have
     \[
     \emptyset \neq \mathcal{B}(\alpha((-\infty,k_j]))\cap \mathcal{B}(\alpha([k_j+1,\infty)))=\mathcal{B}(\alpha([k_{j-1}+1,k_j]))\cap \mathcal{B}(\alpha([k_j+1,k_{j+1}])).
     \]

     By independence of the previous events for odd $j$s (and using translation invariance along the backbone)  this means that
     \begin{align*}
 &   {\bf P}^{-\infty,\infty}\bigl[\text{for all $0\leq j<2J+1$, } \mathcal{B}(\alpha((-\infty,k_j]))\cap \mathcal{B}(\alpha([k_{j+1}+1,\infty)))\neq \emptyset,\ \gamma_1\geq k_{2J}\bigr] \\
 \leq & {\bf P}^{-\infty,\infty}\bigl[  \mathcal{B}(\alpha((-R+1,0]))\cap \mathcal{B}(\alpha([1,R)))\neq \emptyset\bigr]^J\leq {\bf P}^{-\infty,\infty}[0\text{ is a pivotal point}]^J.
   \end{align*}

Recall that ${\bf P}^{-\infty,\infty}[0\text{ is a pivotal point}]>0$ by Lemma~\ref{cutp_2}, so taking $J=\lfloor \frac d{-\ln P[0\text{ is a pivotal point}]} \ln n \rfloor$ and $R=\lfloor \frac{n}{3J} \rfloor$, the previous equation, along with~\eqref{sdfsdf}, implies that
   \[
   {\bf P}^{-\infty,\infty}[\gamma_1\geq k_{2J}]\leq C (\ln n)^{(d-6)/2} n^{-(d-8)/2} + n^{-d}.
   \]

The proof follows by the fact that $\gamma_{i+1}-\gamma_i$ is distributed as $\gamma_1$ under ${\bf P}^{-\infty,\infty}[\cdot \mid 0\text{ is a pivotal point}]$ and that ${\bf P}^{-\infty,\infty}[0\text{ is a pivotal point}]>0$ by Lemma~\ref{cutp_2}.
\end{proof}

Using this lemma, we can show that pivotal points cannot be far, this implies that the distance and the resistance between pivotal points cannot be macroscopic. In the next section, we will use this result to show that $K$-sausages in large critical random trees cannot be long.
\begin{lemma}\label{neglect_bubble}
For any $d>10$, we have
\[
{\bf P}^{-\infty,\infty}\Bigl[\max_{i\in [-n,n]} (\gamma_{i+1}-\gamma_i) \geq n^{1-\epsilon}\Bigr]=o(n^{-\frac{d-10}2-\epsilon}),
\]
for some $\epsilon>0$.

As a consequence, for $d>10$
\[
{\bf P}^{-\infty,\infty}\Bigl[\max_{i\in [-n,n]} R_{\text{eff}}^{-\infty,\infty}(\alpha(\gamma_{i+1}),\alpha(\gamma_i)) \geq n^{1-\epsilon}\Bigr]=o(n^{-\frac{d-10}2-\epsilon}),
\]
where $R_{\text{eff}}^{-\infty,\infty}(\cdot,\cdot)$ denotes the resistance in the graph $\omega_{-\infty,\infty}$. Furthermore, we have
\[
{\bf P}^{-\infty,\infty}\Bigl[\max_{i\in [-n,n]} d_{\omega_{-\infty,\infty}}(\alpha(\gamma_{i+1}),\alpha(\gamma_i)) \geq n^{1-\epsilon}\Bigr]=o(n^{-\frac{d-10}2-\epsilon}),
\]
for some $\epsilon>0$.
\end{lemma}
\begin{proof}
From Lemma~\ref{tail_gamma}, we can deduce that there exists $\epsilon>0$ such that
\begin{equation}\label{moment_gamma}
{\bf E}^{-\infty,\infty}[(\gamma_{i+1}-\gamma_i)^{\frac 1{1-\epsilon}\frac{d-8}2}]<\infty.
\end{equation}

Then, we can use Markov's inequality and a simple union bound to obtain the first part of the lemma.

 For the second part, it is enough to see that by Rayleigh's monotonicity principle (see~\cite{lyons2005probability}) we have
\[
 R_{\text{eff}}^{-\infty,\infty}(\alpha(\gamma_{i+1}),\alpha(\gamma_i))\leq \gamma_{i+1}-\gamma_i,
 \]
 and for the last part we simply notice that the embedded distance is smaller than the distance on the tree
 so $d_{\omega_{-\infty,\infty}}(\alpha(\gamma_{i+1}),\alpha(\gamma_i))\leq \gamma_{i+1}-\gamma_i$.
  \end{proof}

\subsection{Transferring results from the infinite branching random walks to large finite ones}\label{sect_transfer}

In this part, we are going to focus on showing how to turn results obtained for the IICBRW (or a very similar random object) into results for the LCBRW.
\begin{enumerate}
\item The first result allows to transfer estimates on any events which are measurable  with respect to the restriction of the tree (and its embedding) to its first $(1-\epsilon)n$ levels. The drawback of this method is that is does not allow us to express any results about cut-points, which can only identified by looking at the whole tree.
\item The second result, which is more intricate, allows us to transfer results about certain types of increasing events.
\end{enumerate}

\subsubsection{A first transfer result}

Let us discuss the first result from the IICBRW to the LCBRW conditioned on height.

 Fix some $\epsilon>0$. If $A$ is an event measurable with respect to the restriction of the tree (and its embedding) to its first $\lfloor (1-\epsilon)n\rfloor $ levels, then we know that (see (2.12) and (2.15) in~\cite{van2006infinite})
\[
{\bf P}_{\infty}[A]={\bf E}[\1{A} Z_{(1-\epsilon)n}],
\]
where $Z_{(1-\epsilon)n}$ is the cardinal of the generation $\lfloor (1-\epsilon) n \rfloor$ of $\T^{\text{GW}}$ and
\[
{\bf Q}_n[A]={\bf E}\Bigl[\1{A}\frac{1-(1-\eta_{\epsilon n})^{Z_{(1-\epsilon)n}}}{\eta_{n}}\Bigr],
\]
where $\eta_k={\bf P}[H(\T^{\text{GW}})\geq k]$ and ${\bf Q}_n={\bf P}[\cdot \mid H(\T^{\text{GW}})\geq n ]$. Since $\eta_k \sim \frac 2{{\bf E}[Z_1^2]} k^{-1}$ (see~\cite{kesten1966galton}) we can see that for $n$ large enough
\[
\frac{1-(1-\eta_{ \epsilon n})^{Z_{(1-\epsilon)n}}}{\eta_{n}}\leq \frac{\eta_{\epsilon n}}{\eta_n}Z_{(1-\epsilon)n}\leq \frac 2{\epsilon} Z_{(1-\epsilon)n}.
\]

Hence, we immediately obtain the following result
\begin{lemma}\label{transfer_cvg}
If $A_{(1-\epsilon)n}$ is an event measurable with respect to the restriction of the tree (and its embedding) to its first $(1-\epsilon)n$ levels for some $\epsilon>0$, then
\[
\text{ if }\lim_{n\to \infty} {\bold P}_{\infty}[A_{(1-\epsilon)n}]= 0\text{, then } \lim_{n\to \infty} {\bold Q}_{n}[A_{(1-\epsilon)n}]= 0.
\]
\end{lemma}

\subsubsection{Defining the monotone representation of a marked tree}\label{sect:monotone}

Before presenting the second transfer result we will require some additional notations. These definitions are inspired from a paper  of Aldous~\cite{Aldous_fringe} (although we use a slightly different terminology)

We start by considering a finite rooted tree $T$ and $x\in T$, we will call this a marked tree. Denote $x_0,x_1,\ldots, x_{k}$ the unique simple path from $x$ to the root ($x_0=x$ and $x_k=\text{root}$). In practice, this will be viewed as a backbone.

Furthermore consider $\tilde{T}_x(x_i)$ the tree obtained by considering the subtree rooted at $x_i$ composed of all vertices disconnected from $x$ when removing the edge $[x_i,x_{i-1}]$ with the particular case that $\tilde{T}_x(x)$ is the entire tree rooted at $x$. With those definitions, we can see that $\tilde{T}_x(x_{i+1})$ is a subtree of $\tilde{T}_x(x_{i})$.
See Figure 6.

\begin{figure}
  \includegraphics[width=\linewidth]{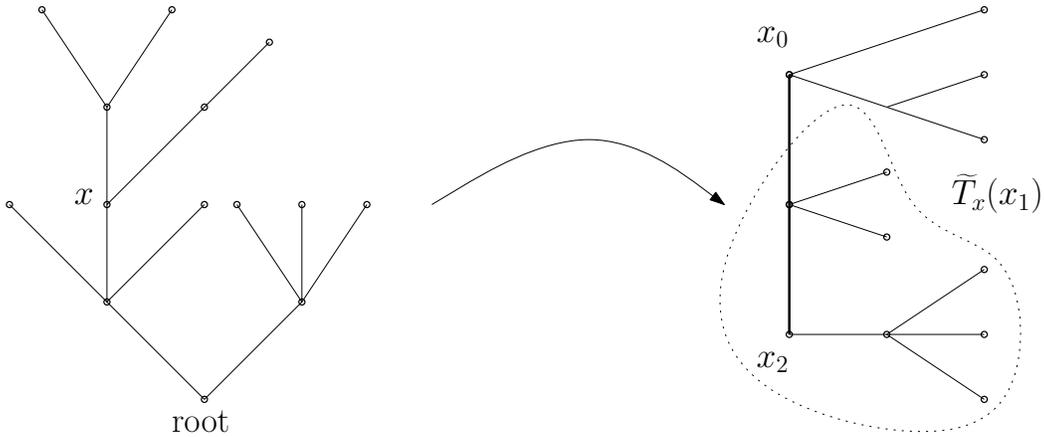}
  \caption{A picture showing the construction of $BB(T,x)$}
\end{figure}

Finally, we introduce $BB(T,x)=(\tilde{T}_x(x_k), \ldots, \tilde{T}_x(x))$, which is an increasing sequence of trees. We will call this the monotone representation of a marked tree. 

Now, if we are given a rooted infinite tree $T_{\infty}$ with a single infinite simple path, such as an IICBRW (but not the IBICBRW), we can introduce a similar definition: We denote $\text{root}=x_0,x_1,\ldots,$ the infinite path (or backbone) started from the root. Consider $T(x_i)$ the tree rooted at $x_i$ obtained by keeping only the vertices that can be reached from $x_i$ without using the edge $[x_i,x_{i+1}]$.  Then, we can view $T_{\infty}$ as a marked tree by writing is as $(T(x_0),T(x_1),\ldots)$. Let us emphasize that contrary to the finite marked tree case, the root does not change during this transformation.

\subsubsection{A key formula linking LCGW and the IICGW}

We will now present a result from Section 4 of~\cite{Aldous_fringe} that will be central for us.

Define a kernel on finite rooted trees $Q(T,S)$ as follows. In a tree $S$, we denote the children of the root by $v_1,\ldots,v_d$ and the tree obtained by considering the descendants of those vertices by $f(S,v_1),\ldots,f(S,v_d)$. Then we define
\[
Q(S,T)=\sum_{i} \1{f(s,v_i)=T},
\]
which counts the number of subtrees of $S$ rooted at distance $1$ of the root of $S$ that are isomorphic to $T$.
 This extends in a natural way to a kernel $Q^{\infty}$ to infinite marked trees if we set
 \[
 Q^{\infty}((T_0,T_1,\ldots),(T_{-1},T_0,T_1,\ldots))=Q(T_0,T_{-1}).
 \]
 
Let us define a measure ${\bf P}_{\infty,+}$ on infinite marked trees yielding a tree  $\T^{\text{GW}}_{\infty,+}$  obtained by taking a ${\bf P}$-Galton-Watson tree, adding an edge to the root at the end of which an independent IICBRW is added. By equation (35) of~\cite{Aldous_fringe}, we know that for any possible marked tree $(T_0,\ldots,T_k)$
\begin{equation}\label{eq_aldous}
{\bf P}_{\infty,+}[\T^{\text{GW}}_{\infty,+,(k)}=(T_0,\ldots,T_k)]={\bf P}[\T^{\text{GW}}=T_k] \prod_{i=1}^k Q(T_i,T_{i-1}),
\end{equation}
where $\T^{\text{GW}}_{\infty,+,(k)}$ denotes the $k$ first components of the representation of the infinite marked tree $\T^{\text{GW}}_{\infty,+}$. 

For the reader aiming to read~\cite{Aldous_fringe} in order to understand equation (35) (which is stated in a general setting) we will explain, without detailing all the terms, why this formula holds.  This explanation may not make much sense on its own but will help the reader navigate~\cite{Aldous_fringe} quickly. The key observation is that ${\bf P}_{\infty,+}$  is the unique measure invariant under the kernel $Q^{\infty}$ associated, in the sense of equation (34) in~\cite{Aldous_fringe}, to the asymptotic fringe distribution (see Section 2 of~\cite{Aldous_fringe}) of ${\bf Q}_n$ which for Galton-Watson trees simply turns out to be ${\bf P}$ (see Lemma 9 in~\cite{Aldous_fringe}).

\subsubsection{Second transfer result}

Let us now present the proposition allowing us to transfer estimates on increasing events obtained on the IICBRW to the case of LCBRW conditioned on size. 

We consider a family of events of the form $\mathfrak{A}_n(\overline{T},k)$ which is  measurable with respect to a marked tree $\overline{T}$ and integers $n,k$ (in pratice $\overline{T}$ and $k$ will be random and $n$ will simply be an index). We say that an event $\mathfrak{A}_n(\overline{T},k)$ is 
\begin{enumerate}
\item increasing in the $k$-component, if for any  $\overline{T}$ and $n$, we have $\mathfrak{A}_n(\overline{T},k) \subset \mathfrak{A}_n(\overline{T},k')$ when $k\leq k'$,
\item increasing in the marked tree component, if for any $n$ and $k$, if  $\mathfrak{A}_n(\overline{T},k) \subset \mathfrak{A}_n(\overline{T}',k)$ when $\overline{T}\leq\overline{T}'$,
\end{enumerate}
where we say that $\overline{T}=(T_0,\ldots,T_l)$ is a smaller than  $\overline{T}'=(T_0',\ldots,T_{l'}')$ if $l\leq l'$ and for every $j\leq l$ we have $T_j\subset T_j'$.

The next lemma essentially says that increasing events that rarely occur on the first $n$ vertices of the backbone of $\T^{\text{GW}}_{\infty,+}$ will only rarely occur on the marked tree obtained by taking $\T^{\text{GW}}_n$ and marking a uniform point $U_1^n$ of $\T^{\text{GW}}_n$ (in spirit turning the path from the root of $\T^{\text{GW}}_n$ to $U_1^n$ into a backbone). More precisely,
\begin{lemma}\label{transfer_monotone}
We denote $\mathfrak{A}_n(\overline{T},k)$ an event which is  measurable with respect to a marked tree $\overline{T}$ and integers $n,k$. Let us assume that, for fixed $n\in \N$, the event is increasing in the $\overline{T}$ and the $k$-components. Furthermore, assume that there exists a function $f(n)\to \infty$ such that
\[
{\bf P}_{\infty,+}[ \mathfrak{A}_n(\T^{\text{GW}}_{\infty,+},n^{1/2}f(n))] =o(n^{-1}),
\]
then we have
\[
\lim_{n\to \infty} {\bf P}_n[ \mathfrak{A}_n(BB(\T^{\text{GW}}_n,U_1^n) ,\abs{U_1^n})] =0,
\]
where $U_1^n$ is a uniformly chosen point on $\T^{\text{GW}}_n$.
\end{lemma}

\begin{proof}
We start  by noticing that by our hypothesis, we know that 
\[
{\bf P}_{\infty,+}[ \mathfrak{A}_n(\T^{\text{GW}}_{\infty,+}, n^{1/2}f(n))]\leq n^{-1}g(n)^{-1},
\]
 for functions $f(n)$ and $g(n)$ going to infinity.
 \[
 {\bf P}_{\infty,+}[ \mathfrak{A}_n(\T^{\text{GW}}_{\infty,+},n^{1/2}(g(n)^{1/2}\wedge f(n)))]\leq n^{-1}g(n)^{-1},
 \]
 so we can bootstrap our hypothesis to see that we can assume that there exists a function $f(n)$ such that $f(n)\to \infty$ and
\begin{equation}\label{bootstrap}
\lim_{n\to \infty} {\bf P}_{\infty,+}[ \mathfrak{A}_n(\T^{\text{GW}}_{\infty,+},n^{1/2}f(n))] \leq n^{-1}f(n)^{-2}.
\end{equation}

 We introduce $A(n)=\{H(\T^{\text{GW}}_n)\leq n^{1/2}f(n)\}$. We have that 
\begin{equation}\label{height_width}
{\bf P}_n[A(n)^c]=o(1),
\end{equation}
which can be deduced from a standard computation based on the fact that since ${\bf E}[Z^2]<\infty$ we have ${\bf P}[H(\T^{\text{GW}})\geq h \mid \abs{\T^{\text{GW}}}=n] \leq C \exp(-ch^2/n)$ for all $n,h\geq 1$ (see~\cite{adj}).

Fix a monotone representation $(T_0,\ldots T_k)$. If we have $BB(\T^{\text{GW}}_n,U_1^n)=(T_0,\ldots, T_k)$, then necessarily $\T^{\text{GW}}_n=T_k$, and $U_1^n$ was chosen as one point among $\prod_{i=1}^k Q(T_i,T_{i-1})$. This means that 
\begin{align*}
{\bf P}_n[BB(\T^{\text{GW}}_n,U_1^n)=(T_0,\ldots, T_k)]&=\frac 1{\abs{T_k}} {\bf P}_n[\T^{\text{GW}}_n=T_k] \prod_{i=1}^k Q(T_i,T_{i-1})\\
            &\leq \frac 1{n} \frac 1{{\bf P}[\abs{\T^{\text{GW}}}= n]} {\bf P}[\T^{\text{GW}}=T_k] \prod_{i=1}^k Q(T_i,T_{i-1}) \\
            &= \frac 1{n} \frac 1{{\bf P}[\abs{\T^{\text{GW}}}=n]} {\bf P}_{\infty,+}[\T^{\text{GW}}_{\infty,+,(k)}=(T_0,\ldots,T_k)],
\end{align*}
where we used~\eqref{eq_aldous}.

This implies, using that  ${\bf P}[\abs{\T^{\text{GW}}}= n]\sim C n^{-3/2}$ (see Lemma 2.1.4.~in~\cite{Kolchin}), the following inequality
\[
{\bf P}_n[BB(\T^{\text{GW}}_n,U_1^n)=(T_0,\ldots, T_k),A(n)]\leq C n^{1/2}{\bf P}_{\infty,+}[\T^{\text{GW}}_{\infty,+,(k)}=(T_0,\ldots,T_k)].
\]

Notice that on $A(n)$, we have $H(\T^{\text{GW}}_n)\leq n^{1/2}f(n)$, so that $BB(\T^{\text{GW}}_n,U_1^n)$ has to be a subset of $(T_0,\ldots, T_{n^{1/2}f(n)})$ for some monotone representation of length $n^{1/2}f(n)$. We can use this and the previous equation to compute 
\begin{align*}
 &{\bf P}_n[\mathfrak{A}_n(BB(\T^{\text{GW}}_n,U_1^n) ,\abs{U_1^n}),A(n)] \\
 =&\sum_{k=1}^{n^{1/2}f(n)} \sum_{(T_0,\ldots, T_k)} {\bf  P}_n[BB(\T^{\text{GW}}_n,U_1^n)=(T_1,\ldots,T_k),A(n)] \1{ \mathfrak{A}_n((T_0,\ldots,T_k),k)} \\
 \leq & C n^{1/2}\sum_{k=1}^{n^{1/2}f(n)} \sum_{(T_0,\ldots, T_k)} {\bf P}_{\infty,+}[\T^{\text{GW}}_{\infty,+,(k)}=(T_0,\ldots,T_k)] \1{ \mathfrak{A}_n((T_0,\ldots,T_k),k)} \\
 = &C n^{1/2}\sum_{k=1}^{n^{1/2}f(n)}  {\bf P}_{\infty,+}[\mathfrak{A}_n(\T^{\text{GW}}_{\infty,+,(k)},k)],
 \end{align*}
 and using the monotonicity properties of $\mathfrak{A}_n$, we see that for $k\leq n^{1/2}f(n)$ we have 
\[
\mathfrak{A}_n(\T^{\text{GW}}_{\infty,+,(k)},k) \leq \mathfrak{A}_n(\T^{\text{GW}}_{\infty,+}, n^{1/2}f(n)) ,
\]
hence the two previous equations imply that 
\[
{\bf P}_n[\mathfrak{A}_n(BB(\T^{\text{GW}}_n,U_1^n) ,\abs{U_1^n}),A(n)] \leq Cn f(n){\bf P}_{\infty,+}[  \mathfrak{A}_n(\T^{\text{GW}}_{\infty,+}, nf(n))] \leq C  f(n)^{-1},
\]
where we used~\eqref{bootstrap}. This implies the result because of~\eqref{height_width} and the fact that $f(n)\to \infty$.
\end{proof}

Using a similar, but simpler, proof, we can also prove that
\begin{lemma}\label{mono2}
If
\[
\lim_{n\to \infty}{\bf P}_{\infty,+}[ \mathfrak{A}_n(\T^{\text{GW}}_{\infty,+},nf(n))] =0,
\]
then
\[
\lim_{n\to \infty} {\bf Q}_n[ \mathfrak{A}_n(BB(\T^{\text{GW}}_n,U_1^n) ,\abs{U_1^n})] =0,
\]
where $U_1^n$ is a uniformly chosen point on $\T^{\text{GW}}_n$.
\end{lemma}

\subsection{An important application obtained by transferring results from the IBICBRW to the LCBRW}

\begin{definition}\label{def_strongcut}
We say that a point $x\in \T^{\text{GW}}_n$ has a loopless image if $\{\phi_{\T^{\text{GW}}_n}(y),\ y \preceq x \} \cap \{ \phi_{\T^{\text{GW}}_n}(y),\ y \nprec x \text{ and } x\neq y\} =\emptyset$, where $\prec$ stands for the genealogical order on $\T^{\text{GW}}_n$.
\end{definition}

\begin{remark}\label{rem_strongcut} Assume that $x$ has a loopless image, then $\phi_{\T^{\text{GW}}_n}(x)$ is a cut-point.\end{remark}

 Let us consider $\T^{\text{GW}}_n$ and $x\in \T^{\text{GW}}_n$, we denote $\overline{\pi}_{\T^{\text{GW}}_n}(x)$ the first ancestor of $x$ in $\T^{\text{GW}}_n$ which  has a loopless image (or the root if no such point exists). We can prove that the distance between $x$ and $\overline{\pi}_{\T^{\text{GW}}_n}(x)$ cannot be macroscopic, actually we can even obtain the following result
\begin{lemma}\label{close_cut_point}
Fix $d>14$. There exists $\epsilon' >0$, such that 
\[
\lim_{n\to \infty} {\bf P}_n\bigl[ \max_{x\prec U_1^n} (\abs{x}-\abs{\overline{\pi}_{\T^{\text{GW}}_n}(x)}) \geq n^{(1-\epsilon')/2}\bigr] =0,
\]
where $U_1^n$ is a uniformly chosen point on $\T^{\text{GW}}_n$.
\end{lemma}

\begin{proof}
Let us consider a marked tree $(T_0,\ldots,T_k)$. We can introduce, for $i\leq k$, the quantity
\[
\psi( (T_0,\ldots,T_k), i)=\max_{j \leq i}\{\abs{x_j}-\abs{\overline{\pi}_{\T^{\text{GW}}_n}(x_j)}\},
\]
where $x_j$ denotes the root of $T_j$. This corresponds to the maximal distance to the first ancestor which is a point with loopless image on the $i$ first levels of the backbone in $(T_0,\ldots,T_k)$. We want to use Lemma~\ref{transfer_monotone} applied to the event $\mathfrak{A}_n(\overline{T},k)=\{\psi(\overline{T},k) \geq n^{(1-\epsilon')/2}\}$ for some $\epsilon'>0$, this would prove that
\[
\lim_{n\to \infty} {\bf P}_n[ \psi(BB(\T^{\text{GW}}_n,U_1^n),\abs{U_1^n}) \geq n^{(1-\epsilon')/2}] =0,
\]
which implies the result.

Let us notice that the event $\mathfrak{A}_n(\overline{T},k)$ is obviously increasing in $k$. Moreover, if $T\subset T'$ and $x\in T$ is a cut-point (resp.~point with loopless image) for $T'$, then it is also a cut-point (resp.~point with loopless image) for $T$. This implies that the event is also increasing in the marked tree component.

We can check the last condition of Lemma~\ref{transfer_monotone} by proving, for example, that for some $\epsilon'>0$
\[
\lim_{n\to \infty} {\bf P}^{-\infty,\infty}[ \psi(\T^{\text{GW}}_{\infty,+},n^{(1+\epsilon')/2}) \geq n^{(1-\epsilon')/2}] =0,
\]
this equation follows from the fact that
\begin{enumerate}
\item By stochastic domination we can create a coupling such that $\T^{\text{GW}}_{\infty,+} \subset \T^{\text{GW}}_{-\infty,\infty}$,
\item Since $\psi$ is increasing in the marked tree component, we have that $\psi(\T^{\text{GW}}_{-\infty,\infty},n^{(1+\epsilon')/2}) \geq \psi(\T^{\text{GW}}_{\infty,+},n^{(1+\epsilon')/2})$,
\item In the coupling, pivotal points for $\T^{\text{GW}}_{-\infty,\infty}$ are always points with loopless images for $\T^{\text{GW}}_{\infty,+} $ and that the estimate of Lemma~\ref{neglect_bubble} states that on the backbone of the IIC, pivotal points before level $n^{1/2}$ cannot be at a distance larger than $n^{(1-\epsilon')/2}$ except with probability $o(n^{-1})$ when $d>14$.
\end{enumerate}

This finishes the proof.
\end{proof}

Using Remark~\ref{mono2} we could also obtain 
\begin{lemma}\label{close_cut_point2}
For $d>10$, there exists $\epsilon' >0$, such that 
\[
\lim_{n\to \infty} {\bf Q}_n\bigl[ \max_{x\prec U_1^n} (\abs{x}-\abs{\overline{\pi}_{\T^{\text{GW}}_n}(x)}) \geq n^{(1-\epsilon')}\bigr] =0,
\]
where $U_1^n$ is a uniformly chosen point on $\T^{\text{GW}}_n$.
\end{lemma}

\section{Resistance estimate, verifying condition $(R)$}\label{sect_resistance}

The goal of this section is to show that the resistance distance  is asymptotically proportional to the graph distance. The proof goes as follows

\begin{enumerate}
\item We are going to apply an ergodic theorem on the IBICBRW with respect to the shift along pivotal points (using the tools of Section~\ref{sect_ergo}). This will allow us to obtain estimates on the resistance between the origin and distant pivotal points.
\item Since  two successive pivotal points are at microscopic distance, we know that understanding the resistance to pivotal points is enough to approximate well the resistance to any point.
\item The same estimate is true on the IICBRW because the discrepancy between the resistance in the IICBRW and in the IBICBRW is controlled by the resistance between the pivotal points surrounding the origin in the IBICBRW (which is small).
\item Resistance estimates on the IICBRW can be transferred to the LCBRW, using results of Section~\ref{sect_transfer}.
\end{enumerate}

The last step is the most complex part of the proof. 

It is intuitively natural to believe that, in the bulk (meaning for points of $\omega_n$ emanating from points in $\T^{\text{GW}}_n$ which are far from level $n$), the IICBRW and the LCBRW should look similar. This intuition is validated by Lemma~\ref{transfer_cvg}. However, one should not forget that the resistance could, in principle, be significantly influenced by points located outside the bulk.  In order to tackle this problem, one can see that it is sufficient to say that bubbles (parts of the graph lying between cut-points) are small and this was proved by transferring results from the IBICBRW to the LCBRW (see Lemma~\ref{close_cut_point}). It should be noted that our transfer of results to the LCBRW is done in several steps: first we transfer to trees conditioned on $\{H(\T)\geq n\}$, then on $\{\abs{\T}\geq n\}$ and finally on $\{\abs{\T}=n\}$.

\subsection{Resistance estimate on IBICBRW}\label{sect_res_est}






Recalling the definition of  $\alpha(n)$ in Section~\ref{sect_res_est_not}, we introduce
\begin{equation}\label{def_rho1}
\rho_1=\frac{{\bf E}^{-\infty,\infty}[R_{\text{eff}}^{-\infty,\infty}(0, \alpha(\gamma_1))\mid 0\text{ is a pivotal point}] }{{\bf E}^{-\infty,\infty}[\gamma_1\mid 0\text{ is a pivotal point}] },
\end{equation}
and
\begin{equation}\label{def_rho2}
\rho_2=\frac{{\bf E}^{-\infty,\infty}[d_{\omega_{{-\infty,\infty}}}(0, \alpha(\gamma_1))\mid 0\text{ is a pivotal point}] }{{\bf E}^{-\infty,\infty}[\gamma_1\mid 0\text{ is a pivotal point}] },
\end{equation}

Then we have
\begin{lemma}\label{simple_LLN1}
Take $d>10$. Then we have
\[
\lim_{n\to \infty}\frac{R_{\text{eff}}^{-\infty,\infty}(0, \alpha(n))}n= \rho_1 \qquad {\bf P}^{-\infty,\infty}\text{-a.s.}
\]
and
\[
\lim_{n\to \infty}\frac{d_{\omega_{-\infty,\infty}}(0, \alpha(n))}n= \rho_2 \qquad {\bf P}^{-\infty,\infty}\text{-a.s.}
\]
\end{lemma}
\begin{proof}
We will only prove the first point, the second one is similar but simpler. By the law of resistances in series, we have that
\[
 R_{\text{eff}}^{-\infty,\infty}(\alpha(\gamma_{i+2}),\alpha(\gamma_i))= R_{\text{eff}}^{-\infty,\infty}(\alpha(\gamma_{i+1}),\alpha(\gamma_i))+ R_{\text{eff}}^{-\infty,\infty}(\alpha(\gamma_{i+2}),\alpha(\gamma_{i+1})),
 \]
 for any $i\in \Z$ since, by the definition of pivotal points, any path going from $\alpha(\gamma_i)$ to $\alpha(\gamma_{i+2})$ has to go through $\alpha(\gamma_{i+1})$. Hence, by Lemma~\ref{ergo_shift} and Birkhoff's ergodic theorem (p.~340 of \cite{durrett2010probability}) we have that $ {\bf P}^{-\infty,\infty}$-$\text{a.s.}$
 \[
 \lim_{n\to \infty} \frac{ R_{\text{eff}}^{-\infty,\infty}(0, \alpha(\gamma_{n}))}n={\bf E}^{-\infty,\infty}[R_{\text{eff}}^{-\infty,\infty}(0, \alpha(\gamma_1))\mid 0\text{ is a pivotal point}].
 \]
 since the previous expectation is finite by~\eqref{moment_gamma} and Rayleigh's monotonicity principle.

Let us denote $\tilde{\gamma}_n$ the last pivotal point before level $n$ i.e.~$\tilde{\gamma}_n=\max \{\gamma_i,\ \gamma_i\leq n\}$. Obviously, by the definition of $\tilde{\gamma}_n$ we have
\[
\tilde{\gamma}_n\leq n < \tilde{\gamma}_{n+1},
\]
and now using Lemma~\ref{neglect_bubble}, we can see that for any $\epsilon>0$
\begin{equation}\label{beta_bound}
{\bf P}^{-\infty,\infty}\Bigl[\abs{\frac{\tilde{\gamma}_n}n-1}>\epsilon\Bigr]=o(1).
\end{equation}

Furthermore let us denote $i(n)$ the largest pivotal point index before $n$, i.e., such that $\gamma_{i(n)}=\tilde{\gamma}_n =\max\{\gamma_i,\ \gamma_i\leq n\}$. By Lemma~\ref{ergo_shift} and Birkhoff's ergodic theorem (p.~340 of \cite{durrett2010probability}) and a standard inversion argument (see for example the proof of Proposition 2.1.~in~\cite{SZ}) we see that
\[
\lim_{n\to \infty} \frac{i(n)}n =\frac 1{{\bf E}^{-\infty,\infty}[\gamma_1\mid 0\text{ is a pivotal point}]}  \qquad {\bf P}^{-\infty,\infty}-\text{a.s.}
\]

 Hence by the previous points, we have
\begin{align*}
\lim_{n\to \infty} \frac{ R_{\text{eff}}^{-\infty,\infty}(0, \alpha(\tilde{\gamma}_n))}n & = \lim_{n\to \infty} \frac{ R_{\text{eff}}^{-\infty,\infty}(0, \alpha(\gamma_{i(n)}))}{ i(n)}\frac{i(n)}n \\
 &=\rho  \qquad {\bf P}^{-\infty,\infty}-\text{a.s.}
\end{align*}

This proves the first point. Furthermore, we can see that
\[
R_{\text{eff}}^{-\infty,\infty}(0, \alpha(\tilde{\gamma}_n))\leq R_{\text{eff}}^{-\infty,\infty}(0, \alpha(n)) \leq R_{\text{eff}}^{-\infty,\infty}(0, \alpha(\tilde{\gamma}_{n+1}))  \qquad {\bf P}^{-\infty,\infty}-\text{a.s.}
\]
which implies the expected result by~\eqref{beta_bound}.
\end{proof}

Fix a spatial tree $(T,\phi_T)$ with $\phi_T(\text{root})=0$ and $k\leq H(T)$ an integer. For $x\in T$, denote $\overline{R}^{\phi_T}_{\text{eff}}(0 , \phi_T(x)\mid k)$ (resp.~$\overline{d}_{\phi_T}(0 , x \mid k)$) the resistance (resp.~graph distance) between $0$ and $\phi_T(x)$ in the graph given by the image through $\phi$ of the tree $T$ restricted to the levels between $0$ and $k$. In the case where $k=d_{T}(0,x)$ we will simply use $\overline{R}_{\text{eff}}^{\phi_T}(0 , \phi_T(x))$ (resp.~$\overline{d}_{\phi_T}(0 , \phi_T(x))$).  Those definitions will be useful for applying Lemma~\ref{transfer_cvg}. We have

\begin{lemma}\label{simple_LLN}
Take $d>10$. Then we have
\[
\lim_{n\to \infty}\frac{\overline{R}_{\text{eff}}^{-\infty,\infty}(0, \alpha(n))}n= \rho_1 \qquad {\bf P}^{-\infty,\infty}\text{-a.s.}
\]
and
\[
\lim_{n\to \infty}\frac{\overline{d}_{\omega_{-\infty,\infty}}(0, \alpha(n))}n= \rho_2 \qquad {\bf P}^{-\infty,\infty}\text{-a.s.}
\]
\end{lemma}

The proof is based on comparing $\overline{R}$ with $R$ and $\overline{d}$ with $d$. The main argument of the proof will have to be repeated three times with slight variations. This proof is the most  delicate one of the three.

We start by proving a technical result.  Fix $\delta>0$, write $y_0^{(n,\delta)},\ldots, y_{l(n,\delta)}^{(n,\delta)}$ for the vertices at generation $n-\frac{\delta}2 n$ that have descendants at generation $n$.  
\begin{enumerate} 
\item there is exactly one of those vertices (say $y_0^{(n,\delta)}$) that belongs to the backbone. 
\item the other vertices $y_i^{(n,\delta)}$, $1\leq i\leq l(n,\delta)$ have ancestors at level $n-\frac{3\delta}4 n$ that we call $z_i^{(n,\delta)}$. When $z_i^{(n,\delta)}$ is not on the backbone then  the tree formed by the descendants of $z_i^{(n,\delta)}$ is denoted by $\T^{i,\delta n}$ which is a critical Galton-Watson tree conditioned to have height at least $\frac{3\delta}4 n$. Up to reordering we can assume that the indices $i$ such that $z_i^{(n,\delta)}$ is not on the backbone are between $1$ and some number $l'(n,\delta)$.
\end{enumerate}

We will show the following
  \begin{lemma} \label{FUCKKKKKKKKKKKinfty}
Fix $\delta>0$,
\[
\lim_{R\to \infty} \limsup_{n\to \infty}{\bf P}^{-\infty,\infty}[ l(n,\delta) >R]=0,
\]
and 
\[
\limsup_{\epsilon' \to 0}\limsup_{n\to \infty} {\bf P}^{-\infty,\infty}\Bigl[ \min_{i=1,\ldots, l'(n,\delta)} \abs{\overrightarrow{\T^{i,\delta n}}_{y_i^{(n,\delta)}}}/\abs{\T^{i,\delta n}} <\epsilon' n\Bigr] =0,
\]
where 
\end{lemma}
\begin{proof}
For the sake of simplicity we will prove the Lemma for the infinite tree $\T^{\text{GW}}_\infty$ instead of the bi-infinite $\T^{\text{GW}}_{-\infty,\infty}$. The proofs can be easily adapted to $\T^{\text{GW}}_{-\infty,\infty}$. 

Let $v_0=\text{root}< v_1<\dots$ be the vertices of the infinite tree $\T^{\text{GW}}_{\infty}$ indexed in increasing lexicographical order. Let 
\[h_{\T^{\text{GW}}_{\infty}}(i)=d^{\T^{\text{GW}}_{\infty}}(\text{root},v_i)\]
be the \emph{height process} associated to $\T^{\text{GW}}_\infty$. We extend the domain of definition of $h_{\T^{\text{GW}}_{\infty}}$ to $[0,\infty)$ through linear interpolation.
By part (ii) of Theorem 1.6 in \cite{angel2013scaling} we get that
\begin{equation}\label{eq:goodman}
(n^{-1}h_{\T^{\text{GW}}_{\infty}}(n^{1/2}t))_{t\geq0}\stackrel{n\to\infty}{\to}(\gamma^{-1/2}(2B_t-3\underline{B}_t))_{t\geq0},
\end{equation}
in distribution in the topology of uniform convergence in compact sets, where $\gamma$ is a positive constant, $B$ is a standard Brownian motion and $\underline{B}_t=\min_{s\leq t}B_t$. 

On the other hand, it is clear that $l(n,\delta)$ equals the number of up-crossings of $[n-\frac{\delta}2 n,n]$ by $h_{\T_{\infty}^{\text{GW}}}$.
We will argue that $l(n,\epsilon)$ converges in distribution, as $n\to\infty$, to the number of up-crossings of $[1-\frac{\delta}{2},1]$ by $\gamma^{-1/2}(2B_t-3\underline{B})$.
 

For each $i\in\N$, let $\mathcal{B}_i$ be the branch emerging from the $i$-th vertex of the backbone of $\T^{\text{GW}}_{\infty}$. Let $U_i$ be the Lukaciewicz  path\footnote{ For the definition of Lukaciewicz path, we refer to Section 1.1 of \cite{le2005random}} of $\mathcal{B}_i$ and $U_{\T^{\text{GW}}_\infty}:\N\to\R$ be the Lukaciewicz  path corresponding to the sequence of trees $(\mathcal{B}_i)_{i\in\N}$, i.e., $U_{\T^{\text{GW}}_\infty}$ is obtained as the concatenation of the $(U_i)_{i\in\N}$. We extend the domain of definition of $U_{\T^{\text{GW}}_\infty}$ to $[0,\infty)$ by linear interpolation. It is clear that the index $j$ of the branch $\mathcal{B}_j$ containing $v_i$ (the $i$-th vertex of $\T^{\text{GW}}_\infty$) is $-\min_{j\leq i} U_{\T^{\text{GW}}_\infty}(j)$. As in the proof of Corollary 4.5 in \cite{angel2013scaling} we have that
\begin{equation}
(n^{-1}\min_{s\leq t}U_{_{\T^{\text{GW}}_\infty}}(n^{1/2}s))_{t\geq0}\stackrel{n\to\infty}{\to}(\frac{1}{m}\underline{B}_t)_{t\geq0},
\end{equation}
in distribution in the topology of uniform convergence over compacts, where $m$ is a positive constant. 
From the display above it can be deduced that
\begin{equation}
\lim_{M\to\infty}\limsup_{n\in\N}\bold{P}[-\min_{i\leq M  n^2}U_{\T^{\text{GW}}_\infty}(i) \leq n]=0.
\end{equation}

Therefore, for all $\eta>0$ there exists $M$ such that 
\[\limsup_{n\in\N}\bold{P}[-\min_{i\leq M  n^2}U_{\T^{\text{GW}}_\infty}(i) \leq n]\leq \eta .\]
It is clear that in the event $\{-\min_{i\leq M  n^2}U_{\T^{\text{GW}}_\infty}(i) > n\}$, all the vertices with indices greater than $Mn^{1/2}$ are on a tree whose index is greater than $n$. In particular, the distance to the root of the IIC of all those vertices is greater than $n$.
 Therefore, letting $l(n,\delta\vert \alpha)$ be the number of up-crossings of $[n-\frac{\delta}2 n,n]$ by $h_{\T_{\infty}^{\text{GW}}}$ in the interval $[\alpha,\infty]$, we get that
  \begin{equation}\label{eq:dasni}
\limsup_{n\to\infty}\bold{P}[l(n,\delta\vert Mn^2)\neq0 ]\leq \eta.
 \end{equation}
 
 On the other hand, it follows from \eqref{eq:goodman} that the number of up-crossings of $[n(1-\frac{\delta}{2}),n]$ by $h_{\T^{\text{GW}}_{\infty}}$ in $[0,Mn^2]$ converges to the number of up-crossings of $[1-\frac{\delta}{2},1]$ by $\gamma^{-1/2}(2B_t-3\underline{B}_t)$ in $[0,M]$. Since $\gamma^{-1/2}(2B_t-3\underline{B})$ is continuous, the latter quantity is finite. Therefore, it follows that
 \[
\lim_{R\to\infty}\limsup_{n\to\infty}\bold{P}[l(n,\delta)-l(n,\delta\vert Mn^2)\geq R]=0.
 \]
 
 This, together with \eqref{eq:dasni} implies that
\[
\lim_{R\to\infty}\limsup_{n\to\infty}\bold{P}[l(n,\delta)\geq R]=0.
 \]
 
That shows the first claim of the Lemma. 
 
For the second claim, we reason as follows:
Using, as above, that each $y_i^{(n,\epsilon)}$ is associated to an up-crossing of the interval $[n(1-\frac{\epsilon}{2}),n]$ by $h_{\T^{\text{GW}}_{\infty}}$, we see that $\abs{\overrightarrow{\T^{i,\delta n}}_{y_i^{(n,\delta)}}}$ equals the duration of the excursion above level $n$ corresponding to the said up-crossing.
Again, by the convergence of $h_{\T^{\text{GW}}_\infty}$ in display \eqref{eq:goodman}, we get that $n^{-1}\abs{\overrightarrow{\T^{i,\delta n}}_{y_i^{(n,\delta)}}}$ converges in distribution of the duration of the excursion of $\gamma^{-1/2}(2B-3\underline{B})$ associated to an up-crossing of $[1-\delta/2,1]$. Analogously $n^{-1}\abs{\T^{i,\delta n}}$ converges in distribution of the duration of the excursion of $\gamma^{-1/2}(2B-3\underline{B})$ associated to an up-crossing of $[1-3\delta/4,1]$. Since $\gamma^{-1/2}(2B-3\underline{B})$ is continuous, the minimum duration of those excursions is positive. The result follows. 
\end{proof}

Let us go back to the proof of Lemma~\ref{simple_LLN}
\begin{proof}
Using Lemma~\ref{simple_LLN1}, it will be sufficient to show that for any $\epsilon>0$
\begin{equation}\label{august1}
\limsup_{n\to \infty} {\bf P}^{-\infty,\infty}\Bigl[\frac{\abs{\overline{R}_{\text{eff}}^{-\infty,\infty}(0, \alpha(n))-R_{\text{eff}}^{-\infty,\infty}(0, \alpha(n))}}n>\epsilon\Bigr]=0.
\end{equation}

On the one hand, by Rayleigh's monotonicity principle, we can see that 
\[
R_{\text{eff}}^{-\infty,\infty}(0, \alpha(n))\leq \overline{R}_{\text{eff}}^{-\infty,\infty}(0, \alpha(n)).
\]

\vspace{0.5cm}

{\it Step 1: construction of a subgraph which is sufficient to evaluation the resistance }

\vspace{0.5cm}

For the other inequality we start by fixing $\delta\in(0,\frac 12)$. By Lemma~\ref{neglect_bubble}, we know that with probability going to $1$ we have that $\max_{i\in [-n,n]} (\gamma_{i+1}-\gamma_i) \leq n^{(1-\epsilon')}$ for $\epsilon'>0$, which implies that  there exists a pivotal point $\alpha(y_{\delta})$ with  $n-\delta n \leq y_{\delta} \leq n -\frac{\delta}2 n$. Obviously, by Rayleigh's monotonicity principle and the fact that $\alpha(y_{\delta})$ is a  pivotal point lying between $0$ and $\alpha(n)$, we have that
\begin{equation}\label{olamexicanoinfty}
R^{-\infty,\infty}(0, \alpha(y_{\delta}))\leq  R^{-\infty,\infty}(0,  \alpha(n)) \leq  R^{-\infty,\infty}(0, \alpha(y_{\delta})) +\delta n.
 \end{equation}

  We will now turn our gaze to $R^{-\infty,\infty}(0, \alpha(y_{\delta}))$. For this, we define  the graph 
 \begin{align*}
  \mathcal{B}_{\text{ubble}}(\delta)  :=  & \{z\in \omega_{-\infty,\infty},\ z\text{ is connected without using cut-points to } \alpha(k)  \\
& \text{ for some $0\leq k\leq n -\frac{\delta}2 n$}\},
 \end{align*}
 that is the graph formed by the bubbles containing the first $n -\frac{\delta}2 n$ point of the backbone starting from $0$.
 
 Take a cut-bond $e$ that is not in $ \mathcal{B}_{\text{ubble}}(\delta)$. It doesn't lie on any simple path from $0$ to $\alpha(y_{\delta})$, so when a current flows from $0$ to $y_{\delta}$  we know that the current travelling through $e$ is $0$ (this follows from Proposition 2.2.~in~\cite{lyons2005probability}). In particular this implies that 
\begin{equation}\label{chowchow}
\overline{R}^{-\infty,\infty}(0, \alpha(y_{\delta})\mid n)=
\overline{R}_{\mathcal{B}_{\text{ubble}}(\delta)}(0, \alpha(y_{\delta})\mid n),
\end{equation}
where we used a natural extension of the notation $\overline{R}^{-\infty,\infty}(\cdot,\cdot \mid n)$ to $\overline{R}_{\mathcal{B}_{\text{ubble}}(\delta)}(\cdot,\cdot \mid n)$ by looking at the resistance in the graph $\mathcal{B}_{\text{ubble}}$ intersected with the sub-graph of $\omega_{-\infty,\infty}$ generated by the images of points of $\T_{-\infty,\infty}^{\text{GW}}$ which are at a level lower than $n$. This also implies
\[
R^{-\infty,\infty}(0, \alpha(y_{\delta}))=
R_{\mathcal{B}_{\text{ubble}}(\delta)}(0, \alpha(y_{\delta})).
\]

The idea developed in this step is illustrated in Figure 7

  \begin{figure}
  \includegraphics[width=\linewidth]{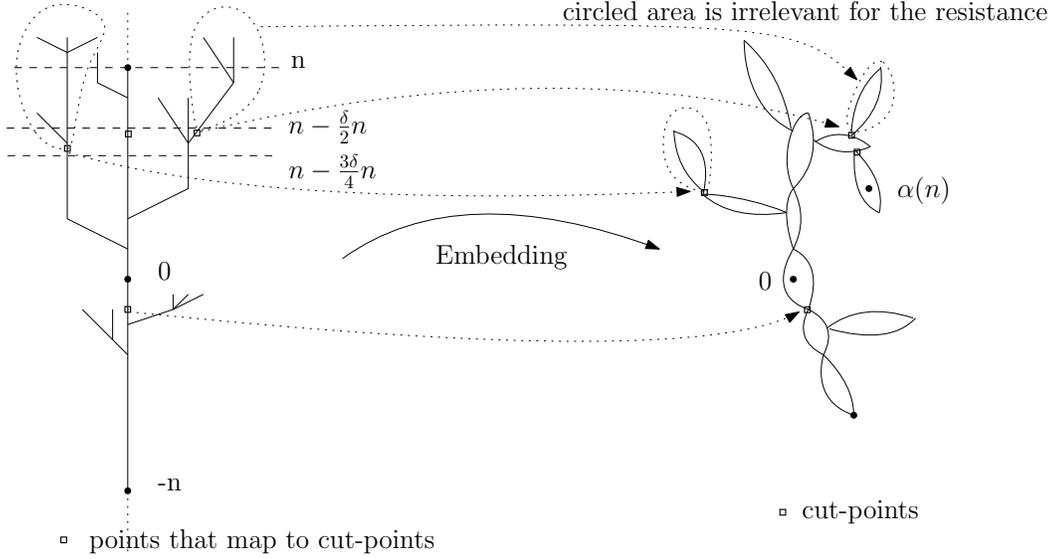}
  \caption{We do not need to know the whole tree in order to compute resistance between $0$ and $\alpha(n)$. If branches reaching level $n$ contain cut-points earlier on, we can estimate the resistance accurately without knowing what happens after level $n$.}
\end{figure}

\vspace{0.5cm}

{\it Step 2: Event on which we can apply the first computation from Step 1}

\vspace{0.5cm}

Introduce
\[
A_{\delta}(n):=\{\text{for any } z\in \T^{\text{GW}}_{-\infty,\infty} \text{ with } \abs{z}\geq n , \text{ we have } \phi_{\T^{\text{GW}}_{-\infty,\infty}}(z) \notin  \mathcal{B}_{\text{ubble}}(\delta)\},
\]
where $\phi_{\T^{\text{GW}}_{-\infty,\infty}}$ is the embedding of $ \T^{\text{GW}}_{-\infty,\infty} $ that yields $\omega_{-\infty,\infty}$. The previous equation means that $ \mathcal{B}_{\text{ubble}}(\delta)$  are only made of images of points in $\T^{\text{GW}}_{-\infty,\infty}$ whose generation is lower than $n$. The key observation is that on $A_{\delta}(n)$ we have
 \[
 R_{\mathcal{B}_{\text{ubble}}(y_{\delta})}(0, \alpha(y_{\delta}))=\overline{R}_{\mathcal{B}_{\text{ubble}}(y_{\delta})}(0, \alpha(y_{\delta}) \mid n),
 \]
 which implies by the two preceding equations that, on $A_{\delta}(n)$
 \begin{equation}\label{olamexicano2infty}
 \overline{R}^{-\infty,\infty}(0, \alpha(y_{\delta}) \mid n)=R^{-\infty,\infty}(0, \alpha(y_{\delta})).
 \end{equation}

Finally, we can notice that $y_{\delta}$ is a cut-point for the restricted subgraph generated by embedding the $n$ first levels of $\T^{\text{GW}}_{-\infty,\infty}$. This and Rayleigh's monotonicity principle imply that
\begin{equation}\label{olamexicano3infty}
 \overline{R}^{-\infty,\infty}(0, \alpha(y_{\delta})\mid n)\leq  \overline{R}^{-\infty,\infty}(0, \alpha(n))\leq  \overline{R}^{-\infty,\infty}(0, \alpha(y_{\delta})\mid n) +\delta n.
 \end{equation}

The combination of~\eqref{olamexicanoinfty}, \eqref{olamexicano2infty} and~\eqref{olamexicano3infty} imply that on $A_{\delta}(n)$, we see that
\[
\abs{\overline{R}^{-\infty,\infty}(0, \alpha(n))-R^{-\infty,\infty}(0, \alpha(n))} \leq 2\delta n.
\]


Hence, in order to prove the result, it will be enough to show that for any $\delta\in(0,\frac 12)$ we have that 
\[
 \limsup_{n\to\infty}{\bf P}^{-\infty,\infty}[A_{\delta}(n)]=0,
 \]
 i.e.~showing that the ancestors of $\alpha(n)$ which are at a generation lower than $n-\frac{\delta}2n$ are not in bubbles that contain images of points at a generation greater than $n$, except with vanishing probability.
 
Fix $\delta \in(0,\frac 12)$, write $y_0^{(n,\delta)},\dots, y_{l(n,\delta)}^{(n,\delta)}$ for the vertices at generation $n-\frac{\delta}2n$ that have descendants at generation $n$ such as introduced before Lemma~\ref{FUCKKKKKKKKKKKinfty}. For any $\epsilon>0$, we can choose, by Lemma~\ref{FUCKKKKKKKKKKKinfty}, $M<\infty$ and $\epsilon_1>0$ such that the event
\[
B(n)=\{l(n,\delta) \leq M\} \cap \Bigl\{ \min_{i=1,\ldots, l'(n,\delta)} \abs{\overrightarrow{\T^{i,\delta n}}_{y_i^{(n,\delta)}}}/\abs{\T^{i,\delta n}} <\epsilon' n \Bigr\},
\]
verifies
 \begin{equation}\label{ave2a}
\limsup_{n\to \infty} {\bf P}^{-\infty,\infty}[B(n)^c ] <\epsilon.
\end{equation}

On $B(n)$, by virtue of~Lemma~\ref{neglect_bubble}, we have  that with probability going to $1$ there is a loopless cut-point of $y_0^{(n,\delta)}$ at distance less than $n^{(1-\epsilon')}$. For the other (at most $M$) vertices $y_i^{(n,\delta)}$, there are two cases.
\begin{enumerate} 
\item If the corresponding point $z_i^{(n,\delta)}$ is on the backbone then by~Lemma~\ref{neglect_bubble} with probability going to $1$ there is a loopless cut-point of $y_i^{(n,\delta)}$ at distance less than $\frac{3\delta}4n+n^{(1-\epsilon')}$ from $y_i^{(n,\delta)}$.
\item  If the corresponding point $z_i^{(n,\delta)}$ is not on the backbone (i.e.~$i\leq l'(n,\delta)$), we can successively choose uniform points $(U_j^{\delta n,i})_{j\geq 0}$ in the corresponding tree $\T^i_{\delta n}$ (which is a Galton-Watson tree conditioned to have height larger than $\frac{3\delta}4 n$). On $B(n)$, we know that at each step we have probability at least $\epsilon_1$ to pick a uniform point which has $y_i^{(n,\delta)}$ as an ancestor. Since there are at most $M$ such $y_i^{(n,\delta)}$, we know there exists $J<\infty$ such that
\begin{equation}\label{ave2b}
\limsup_{n\to \infty} {\bf P}^{-\infty,\infty}[B(n), \text{for any $i \leq l'(n,\delta)$, } y_i^{(n,\delta)}\prec U_j^{\delta n,i} \text{ for some $j\leq J$}]>1-\epsilon,
\end{equation}
and denote $C(n)^c$ the event $\{\text{for any $i \leq l(n,\delta)$, } y_i^{(n,\delta)}\prec U_j^{\delta n , i} \text{ for some $j\leq J$}\}$.
\end{enumerate}

Now, by applying Lemma~\ref{close_cut_point2} to the points $U_j^{\delta n,i }$ for $j\leq J$, we know that for some $\epsilon'>0$ we have
\begin{equation}\label{ave2c}
\lim_{n\to \infty} {\bf P}^{-\infty,\infty}\Bigl[ \max_{x\prec U_j^{\delta n,i}} (\abs{x}-\abs{\overline{\pi}_{\T^{i,\delta n}}(x)}) \geq \bigl(\frac 34 \delta n\bigr)^{(1-\epsilon')} \text{ for some $j\leq J$}\Bigr] =0,
\end{equation}
and call $D(n)^c$ the event in the previous line.

We can notice that on $B(n)\cap C(n) \cap D(n)$,
\begin{itemize}
\item the only paths leading to generation $n$ go through the points $ y_i^{(n,\delta)}$ for $ i \leq l(n,\delta)$,
\item $y_0^{(n,\delta)}$ and all $y_i^{(n,\delta)}$ for $i>l'(n,\delta)$ have an ancestor $\chi_i$ which is a loopless point at a distance which is less than $\frac{3\delta}4n+n^{(1-\epsilon')}$.
\item any $y_i^{(n,\delta)}$ is an ancestor of some $U_j^{\delta n,i}$ for $j\leq J$ and as such $y_i^{(n,\delta)}$  has an ancestor $\chi_i$ which is a cut-point (in the image of the tree $\T^i_{\delta n}$) at distance less than $n^{(1-\epsilon')}$. Hence this $\T^i_{\delta n}$-cut-point is at a generation at least $n-\frac{3\delta}4n -(\frac 34 \delta n)^{(1-\epsilon')}$.
\end{itemize}

Let us now argue that, on $B(n)\cap C(n) \cap D(n)$, all the $\chi_i$ are, with high probability, cut-points in $\omega_{-\infty,\infty}$. This statement is obvious for $\chi_0$ and $\chi_i$ for $i>l'(n\delta)$ since all  points with loopless image are cut-points. It remains to be proved that with high probability all $\chi_i$ for $1\leq i \leq l'(n,\delta)$ are cut-points. We know that $\chi_i$ is a $\T^i_{\delta n}$-cut-point, hence we only have to see how parts of $\omega_{-\infty,\infty}$ other than $\T^i_{\delta n}$ can prevent $\chi_i$ from being a cut-point.
\begin{itemize}
\item The part of the graph that are descendants of $\alpha(n)$ cannot prevent any $\chi_i$ from being a cut-point in $\omega_{-\infty,\infty}$ since there exists a loopless point separating $\chi_i$ and any descendant of $\alpha(n)$.
\item Similarly, if $\gamma_{-1} \geq -n$, then any points that are not descendants of $\alpha(-n)$ cannot prevent $\chi_i$ from being cut-points.
\item Recalling the notations at the beginning of Section~\ref{sect_res_est_not}, we denote $\mathcal{B}^{\geq (1-\frac{3\delta}2) n}(x)$ the image of the points at height larger that $(1-\frac{3\delta}2n)$ of a modified Galton-Watson tree $\mathcal{T}_*^{\text{GW}}(x)$ rooted at $x$. Then, if for all $-n\leq i,j \leq n$ we have  that $\mathcal{B}^{\geq (1-\frac{3\delta}2) n}(\alpha(i))\cap \mathcal{B} (\alpha(j)) =\emptyset$, then we know that no descendants of $\alpha(j)$ with $-n\leq j \leq n$ can prevent $\chi_i$ from being a cut-point (since it is at generation at least $n-\frac {3\delta}4 n -(\frac 34 \delta n)^{1-\epsilon'}$.
\end{itemize}

This means that 
\begin{align*}
& \limsup_{n\to\infty}{\bf P}^{-\infty,\infty}[B(n)\cap C(n) \cap D(n),\ \text{one of the $\chi_i$ is not a cut-point}] \\
\leq & \limsup_{n\to\infty}{\bf P}^{-\infty,\infty}[\gamma_{-1} \geq -n] \\
& \qquad \qquad + \limsup_{n\to\infty}{\bf P}^{-\infty,\infty}[\mathcal{B}^{\geq (1-\frac{3\delta}2) n}(\alpha(i))\cap \mathcal{B}(\alpha(j)) \neq \emptyset, \text{ for some $-n\leq i,j\leq n$}] \\
\leq & Cn^2 \max_{n\leq i,j\leq n} \limsup_{n\to\infty}{\bf P}^{-\infty,\infty}[\mathcal{B}^{\geq (1-\frac{3\delta}2) n}(\alpha(i))\cap \mathcal{B}(\alpha(j)) \neq \emptyset],
\end{align*}
where we used Lemma~\ref{tail_gamma} to say that $\limsup_{n\to\infty}{\bf P}^{-\infty,\infty}[\gamma_{-1} \geq -n]=0$.

Notice that $\mathcal{B}^{\geq (1-\frac{3\delta}2)n}(x)\subset \mathcal{B}(x)$, which means by Corollary~\ref{bound_q} for any $x,y\in \Z^d$
\begin{align*}
{\bf P}^{-\infty,\infty}[\mathcal{B}^{\geq (1-\frac{3\delta}2)n}(x)\cap \mathcal{B} (y) \neq \emptyset] & \leq  {\bf P}^{-\infty,\infty}[\mathcal{B}(x)\cap \mathcal{B} (y) \neq \emptyset]\\
& \leq C n^{-\frac{d-4}2},
\end{align*}
where $p_k(\cdot,\cdot)$ is the heat kernel of the simple random walk that is controlled by~\eqref{gauss_bound}. Hence, for $d>8$ we have that
\[
 \limsup_{n\to\infty}{\bf P}^{-\infty,\infty}[B(n)\cap C(n) \cap D(n),\ \text{one of the $\chi_i$ is not a cut-point}]=0.
 \]

On the former event we know that any path from level $n-\delta n$ that reaches level $n$ has to cross one of the cut-points $\chi_i$ (on the path to some $y_i^{(n,\delta)}$) with $ i \leq l(n,\delta)$. This is not compatible with $A_{\delta}(n)$. This means, by the previous equation, \eqref{ave2a}, \eqref{ave2b} and~\eqref{ave2c}, that
  \begin{align*}
  &  \limsup_{n\to\infty}{\bf P}^{-\infty,\infty}[A_{\delta}(n)] \\
  \leq & \limsup_{n\to \infty} {\bf P}^{-\infty,\infty}[B(n)^c ]+ \limsup_{n\to \infty} {\bf P}^{-\infty,\infty}[B(n),C(n)^c] + \limsup_{n\to \infty} {\bf P}^{-\infty,\infty}[B(n),C(n),D(n)^c]  \\ 
  & \qquad + \limsup_{n\to\infty}{\bf P}^{-\infty,\infty}[B(n)\cap C(n) \cap D(n),\ \text{all $\chi_i$ are not cut-points}] \\
  \leq & 2\epsilon
  \end{align*}
  for any $\epsilon>0$. Since $\epsilon$ is arbitrary this means that
  \[
  \limsup_{n\to\infty}{\bf P}^{-\infty,\infty}[A_{\delta}(n)]=0,
  \]
  which is what we needed to complete the proof.
\end{proof}

\subsection{Resistance estimate on the IICBRW}

We are now transferring our results to the IICBRW.
\begin{lemma}
 \label{res_LLN}
Take $d>10$. Then we have
\[
\lim_{n\to \infty}\frac{\overline{R}^{\infty}_{\text{eff}}(0, \alpha(n))}n= \rho_1 \qquad {\bf P}_{\infty}-\text{a.s.}
\]
and
\[
\lim_{n\to \infty}\frac{\overline{d}_{\omega_{\infty}}(0, \alpha(n))}n= \rho_2 \qquad {\bf P}_{\infty}-\text{a.s.}
\]
\end{lemma}

\begin{proof}
Once again, we will only prove the first part which is the more complicated one.

Let us use the natural coupling between $\omega_{-\infty,\infty}$ (the IBICBRW)  and $\omega_{\infty}$ (the IICBRW) such that
$\omega_{\infty} \subset \omega_{-\infty,\infty}$ (see~\ref{stoch_dom}). We see that by Rayleigh's monotonicity principle that
\[
\overline{R}^{-\infty,\infty}_{\text{eff}}(0, \alpha(n))\leq \overline{R}^{\infty}_{\text{eff}}(0, \alpha(n)).
\]

We shall now look for a similar bound in the other direction. Recall that $\gamma_1$ denotes the first positive pivotal point in $\omega_{-\infty,\infty}$. This point is also a pivotal point for $\omega_{\infty}$ (although not necessarily the first one). Since the resistance is a metric on finite networks (see~\cite{lyons2005probability} exercise 2.68), we have
\begin{align*}
\overline{R}^{\infty}_{\text{eff}}(0, \alpha(n)\mid n) & \leq \overline{R}^{\infty}_{\text{eff}}(0, \alpha(\gamma_{1})\mid n) + \overline{R}^{\infty}_{\text{eff}}( \alpha(\gamma_{1}), \alpha(n)\mid n)  \\
& \leq \gamma_0 +\overline{R}^{\infty}_{\text{eff}}( \alpha(\gamma_{1}), \alpha(n)\mid n),
\end{align*}
where we used Rayleigh's monotonicity principle in the last line (indeed $0$ and $\alpha(\gamma_1)$ are connected by a path of length at most $\gamma_1$).

Furthermore, we let $i(\cdot)$ be the unit current (see~\cite{lyons2005probability} for an introduction to electrical networks) flowing from $\alpha(\gamma_1)$ to $\alpha(n)$ on $\omega_{-\infty,\infty}$. Since any current is acyclic (see Proposition 3.2.~of~\cite{lyons2005probability}), we know that $i(\cdot)$ does not flow through any edge between $0$ and $\alpha(\gamma_1)$. This implies, by Thompson's principle (see~\cite{lyons2005probability}), that
\[
\overline{R}^{\infty}_{\text{eff}}( \alpha(\gamma_{1}), \alpha(n))=\overline{R}^{-\infty,\infty}_{\text{eff}}( \alpha(\gamma_{1}), \alpha(n)).
\]

The last three equations imply that
\[
\abs{\overline{R}^{-\infty,\infty}_{\text{eff}}(0, \alpha(n))- \overline{R}^{\infty}_{\text{eff}}(0, \alpha(n))}\leq \gamma_1.
\]

We can now use Lemma~\ref{neglect_bubble} and Lemma~\ref{simple_LLN} to obtain the result.
\end{proof}

\subsection{Resistance on LCBRW conditioned on height} \label{proof_res0}

We are now going to transfer the results obtained for the IICBRW to the case of LCBRW conditioned on height. This part of the proof is slightly involved and is done in several steps.


\subsubsection{Transferring results from the IICBRW}

Let us start by a technical version of the result we are trying to prove. For $x\in \T^{\text{GW}}$ such that $d(o,x)\geq k$, we denote for $k'\leq k$ the unique ancestor of $x$ at level $k'$ by $\overleftarrow{x}^{k'}$, in particular $\overleftarrow{x}^0$ is the root.
\begin{lemma}\label{res_tech1}
For $\epsilon>0$, $\delta\in (0,1)$ and $\eta\in (0,1)$, we have for any $i\in \N$ such that $\eta \leq \eta+i\delta \leq 1-\eta$,
\[
\lim_{n\to \infty} {\bf Q}_n\Bigl[\abs{\frac{\overline{R}^{\omega_n}(0, \phi_{\T^{\text{GW}}_n}(\overleftarrow{U_1^n}^{(\eta+i\delta)n}))}{ (\eta+i\delta)n} -\rho_1} >\epsilon, \frac{d^{\omega_n}(0,U_1^n)}n\in [\eta +i\delta,\eta+(i+1)\delta)\Bigr]=0,
\]
and
\[
\lim_{n\to \infty} {\bf Q}_n\Bigl[\abs{\frac{\overline{d}_{\omega_n}(0, \phi_{\T^{\text{GW}}_n}(\overleftarrow{U_1^n}^{(\eta+i\delta)n}))}{ (\eta+i\delta)n} -\rho_2} >\epsilon, \frac{d^{\omega_n}(0,U_1^n)}n\in [\eta +i\delta,\eta+(i+1)\delta)\Bigr]=0,
\]
where $U_1^n$ is uniformly chosen among the vertices of $\T^{\text{GW}}_n$.
\end{lemma}

\begin{proof}
We will only prove the first equation, since both are proved in the same manner.

We start by noticing that conditionally on the first $(\eta+i\delta)n$ levels of $\T^{\text{GW}}_n$, the position of $\overleftarrow{U_1^n}^{(\eta+i\delta)n}$ (when $d^{\omega}(0,U_1^n)/n\in [\eta +i\delta,\eta+(i+1)\delta)$) is uniformly distributed on the vertices of $\T^{\text{GW}}_n$ at level $(\eta+i\delta)n$. 

We also recall that, by definition,  the  random variable $\overline{R}^{\omega_n}(0, \phi_{\T^{\text{GW}}_n}(\overleftarrow{U_1^n}^{(\eta+i\delta)n}))$ is measurable with respect to the first $(\eta+i\delta)n$ levels of $\T^{\text{GW}}_n$. This means that
\begin{align} \label{res_tech0}
& \limsup_{n\to \infty} {\bf Q}_n\Bigl[\abs{\frac{\overline{R}^{\omega_n}(0, \phi_{\T^{\text{GW}}_n}(\overleftarrow{U_1^n}^{(\eta+i\delta)n}))}{ (\eta+i\delta)n} -\rho_1} >\epsilon, \frac{d^{\omega_n}(0,U_1^n)}n\in [\eta +i\delta,\eta+(i+1)\delta)\Bigr] \\ \nonumber
\leq & \limsup_{n\to \infty} {\bf Q}_n\Bigl[\abs{\frac{\overline{R}^{\omega_n}(0, \phi_{\T^{\text{GW}}_n}(\overleftarrow{U_1^n}^{(\eta+i\delta)n}))}{ (\eta+i\delta)n} -\rho_1} >\epsilon\Bigr],
\end{align}
where $U^{(\eta+i\delta)n}$ is a uniform random variables on the vertices of $\T^{\text{GW}}_n$ at level $(\eta+i\delta)n$ that is independent of the structure of $\T^{\text{GW}}_n$ at other levels.

On the other hand, we know that for the IICBRW, conditionally on the first $(\eta+i\delta)n$ levels of $\T^{\text{GW}}_{\infty}$, the vertex $\alpha((\eta+i\delta)n)$ (the unique vertex of the backbone at level $(\eta+i\delta)n$) is uniform among the vertices of $\T^{\text{GW}}_{\infty}$ at level $(\eta+i\delta)n$. Hence
\[
{\bf P}_{\infty}\Bigl[\abs{\frac{\overline{R}^{\omega_{\infty}}(0, \phi_{\T^{\text{GW}}_{\infty}}(\overleftarrow{U_1^n}^{(\eta+i\delta)n}))}{ (\eta+i\delta)n} -\rho_1} >\epsilon\Bigr]={\bf P}_{\infty}\Bigl[\abs{\frac{\overline{R}^{\omega_{\infty}}(0, \alpha((\eta+i\delta)n))}{ (\eta+i\delta)n} -\rho_1} >\epsilon\Bigr],
\]
where $\overleftarrow{U_1^n}^{(\eta+i\delta)n}$ is a uniform random variables on the vertices of $\T^{\text{GW}}_{\infty}$ at level $(\eta+i\delta)n$ and independent of the rest of the tree. By Lemma~\ref{res_LLN}, the right hand side of the previous equation goes to $0$ as $n$ goes to infinity, which implies that
\[
\lim_{n\to \infty}{\bf P}_{\infty}\Bigl[\abs{\frac{\overline{R}^{\omega_{\infty}}(0, \phi_{\T^{\text{GW}}_{\infty}}(\overleftarrow{U_1^n}^{(\eta+i\delta)n}))}{ (\eta+i\delta)n} -\rho_1} >\epsilon\Bigr]=0.
\]

Now, we want to use Lemma~\ref{transfer_cvg} to transfer this result for the measure ${\bf Q}_n$. In order to use this result, we simply notice that, by the definition of $\overline{R}^{\omega_n}(0 , \cdot)$, the event $\Bigl\{\abs{\frac{\overline{R}^{\omega_n}(0, \phi_{\T^{\text{GW}}}(\overleftarrow{U_1^n}^{(\eta+i\delta)n}))}{ (\eta+i\delta)n} -\rho} >\epsilon\Bigr\}$ is measurable with respect to the first $(\eta+i\delta)n\leq(1-\eta)n$ levels of the tree. Thus we can combine the previous equation with Lemma~\ref{transfer_cvg} to see that 
\[
\lim_{n\to \infty} {\bf Q}_n\Bigl[\abs{\frac{\overline{R}^{\omega_n}(0, \phi_{\T^{\text{GW}}_n}(\overleftarrow{U_1^n}^{(\eta+i\delta)n}))}{ (\eta+i\delta)n} -\rho} >\epsilon\Bigr]=0,
\]
as well. Recalling~(\ref{res_tech0}), we obtain the result.
\end{proof}

\subsubsection{Relating $\overline{R}^{\omega_n}$ and $R^{\omega_n}$}

Let us start by proving a technical result. Fix $\delta>0$, write $y_1^{(n,\delta)},\ldots, y_{l(n,\delta)}^{(n,\delta)}$ for the vertices at generation $(\eta +(i-1/2) \delta)n $ that have descendants at generation $(\eta+i\delta)n$. If $\abs{U_1^n}<\eta$ we set $l(n,\epsilon)=0$.

  \begin{lemma} \label{FUCKKKKKKKKKKK1}
Fix $\delta>0$, we have for any $i\in \N$ such that $\eta \leq \eta+i\delta \leq 1-\eta$,
\begin{equation}\label{eq:barriewhit}
\limsup_{K\to \infty} \limsup_{n\to \infty}{\bf Q}_n[ l(n,\delta) >K]=0,
\end{equation}
and 
\begin{equation}\label{eq:lordqua}
\limsup_{\epsilon' \to 0}\limsup_{n\to \infty} {\bf Q}_n\Bigl[ \min_{i=1,\ldots, l(n,\delta)} \abs{\overrightarrow{\T^{\text{GW}}_n}_{y_i^{(n,\delta)}}} <\epsilon' n\Bigr] =0,
\end{equation}
where $\overrightarrow{\T^{\text{GW}}_n}_{y_i^{(n,\delta)}}$ is the tree composed of the descendants of $y_i^{(n,\delta)}$ in $\T^{\text{GW}}_n$.  
\end{lemma}
\begin{proof}
\emph{Proof of display \eqref{eq:barriewhit}}
Let $h_{\T^\text{GW}_n}$ be the height function associated to $\T^{\text{GW}}_n$ as in the proof of Lemma \ref{FUCKKKKKKKKKKKinfty} and extend its domain of definition to $[0,\infty)$ by setting $h_{\T^{\text{GW}}_n}(t)=0$ for $t\geq \abs{\T^{\text{GW}}_n}$.
It is not hard to notice that $l(n,\delta)$ equals the number of up-crossings of the interval $[(\eta +(i-1/2) \delta) n, (\eta +i \delta)n]$ by $h_{\T^{\text{GW}}_n}$. On the other hand, from Corollary 1.13 of \cite{le2005random} we have that
\begin{equation}\label{eq:olodu} (n^{-1}h_{\T^\text{GW}_n}(\lfloor tn^2\rfloor))_{t\geq 0} \stackrel{n\to\infty}{\to} (\frac{2}{\sigma_Z} \hat{\bold{e}}_t)_{t\geq0} 
\end{equation} in distribution, where $\hat{\bold{e}}$ is a Brownian excursion conditioned on having height larger than $\frac{\sigma_Z}{2}$ (and we recall that $\sigma_Z^2$ is the variance of the offspring distribution of $\T^{\text{GW}}$). Hence, $l(n,\delta)$ converges in distribution, as $n\to\infty$, to the number of up-crossings of $[\eta +(i-1/2) \delta, \eta +i \delta]$ by $\frac{2}{\sigma_Z}\hat{\bold{e}}$. But, since $\hat{\bold{e}}$ is continuous, the said number of up-crossings is almost surely finite. This finishes the proof.
 
\emph{Proof of display \eqref{eq:lordqua}} 
Using, as above, that each $y_j^{(n,\delta)}$ is associated to an up-crossing of the interval $[(\eta +(i-1/2) \delta) n, (\eta +i \delta)n]$ by $h_{\T^{\text{GW}}_n}$, we see that $\abs{\overrightarrow{\T^{\text{GW}}_n}_{y_j^{(n,\delta)}}}$ equals the duration of the excursion above level $(\eta +(i-1/2) \delta) n$ corresponding to the said up-crossing.
Again, by the convergence of $h_{\T^{\text{GW}}_n}$ in display \eqref{eq:olodu}, we get that $n^{-1}\abs{\overrightarrow{\T^{\text{GW}}_n}_{y_j^{(n,\delta)}}}$ converges in distribution of the duration of the excursion of $\frac{2}{\sigma_Z}\hat{\bold{e}}$ associated to an up-crossing of $[\eta +(i-1/2) \delta, \eta +i \delta]$. Since $\hat{\bold{e}}$ is uniformly continuous, the minimum duration of those excursions is positive. The result follows. 
\end{proof}

We have
  \begin{lemma} \label{ASH}
Fix $\delta>0$, we have for any $i\in \N$ such that $\eta \leq \eta+i\delta \leq 1-\eta$,
\begin{align*}
 & \limsup_{n\to\infty}{\bf Q}_n\Bigl[\abs{\overline{R}^{\omega_n}(0, \phi_{\T^{\text{GW}}_n}(\overleftarrow{U_1^n}^{(\eta+i\delta)n}))-R^{\omega_n}(0, \phi_{\T^{\text{GW}}_n}(\overleftarrow{U_1^n}^{(\eta+i\delta)n}))} \geq 2\delta n  \\
& \qquad \qquad \qquad \qquad\qquad \qquad \text{ and }\frac{d^{\omega_n}(0,U_1^n)}n\in [\eta +i\delta,\eta+(i+1)\delta)\Bigr]=0
\end{align*}
and
\begin{align*}
 & \limsup_{n\to\infty}{\bf Q}_n\Bigl[\abs{\overline{d}^{\omega_n}(0, \phi_{\T^{\text{GW}}_n}(\overleftarrow{U_1^n}^{(\eta+i\delta)n}))-d^{\omega_n}(0, \phi_{\T^{\text{GW}}_n}(\overleftarrow{U_1^n}^{(\eta+i\delta)n}))} \geq 2\delta n  \\
& \qquad \qquad \qquad \qquad\qquad \qquad \text{ and }\frac{d^{\omega_n}(0,U_1^n)}n\in [\eta +i\delta,\eta+(i+1)\delta)\Bigr]=0
\end{align*}
\end{lemma}

This argument will be a version of the argument given to prove~\eqref{august1}.
\begin{proof}
As in the previous proofs, we will only prove the first statement which is the harder one. We assume that $d^{\omega}(0,U_1^n)/n\in [\eta +i\delta,\eta+(i+1)\delta)$.

By Lemma~\ref{close_cut_point2}, we know that with probability going to $1$ we have that $\max_{x\prec U_1^n} (\abs{x}-\abs{\overline{\pi}_{\T^{\text{GW}}_n}(x)}) \leq n^{(1-\epsilon')}$ for $\epsilon'>0$, which implies that  there exists a point with loopless image $y_{\delta} \prec \overleftarrow{U_1^n}^{(\eta+i\delta)n}$ such that $(\eta+(i-1)\delta)n \leq \abs{y_{\delta}} \leq (\eta+(i-1/2)\delta)n  $. Obviously, by Rayleigh's monotonicity principle and the fact that $\phi_{\T^{\text{GW}}_n}(y_{\delta})$ is (the image of a loopless point and hence) a cut-point lying between $0$ and $U_1^n$, we have that
\begin{equation}\label{olamexicano1}
R^{\omega_n}(0, \phi_{\T^{\text{GW}}_n}(y_{\delta}))\leq  R^{\omega_n}(0,  \phi_{\T^{\text{GW}}_n}(\overleftarrow{U_1^n}^{(\eta+i\delta)n})) \leq  R^{\omega_n}(0, \phi_{\T^{\text{GW}}_n}(y_{\delta})) +2\delta n.
 \end{equation}

  We are now going to evaluate $R^{\omega_n}(0, \phi_{\T^{\text{GW}}_n}(y_{\delta}))$. For this, we introduce  the graph 
 \begin{align*}
  \mathcal{B}_{\text{ubble}}(\delta)  :=  & \{z\in \omega,\ z\text{ is connected without using cut-points to } \phi_{\T^{\text{GW}}_n}(z')  \\
& \text{ for some $z'\prec \overleftarrow{U_1^n}^{(\eta+i\delta)n}$ with $\abs{z'} \leq (\eta+(i-1/2)\delta)n$}\},
 \end{align*}
 that is the graph formed by the bubbles containing the first $(\eta+(i-1/2)\delta)n$ ancestors of $U_1^n$ starting from $0$. 
 
 Take a cut-bond $e$ that is not in $ \mathcal{B}_{\text{ubble}}(\delta)$. It doesn't lie on any simple path from $0$ to $\phi_{\T^{\text{GW}}_n}(y_{\delta})$, so when a current flows from $0$ to $y_{\delta}$  we know that the current travelling through $e$ is $0$ (Proposition 2.2.~in~\cite{lyons2005probability}). In particular this implies that
\[
\overline{R}^{\omega_n}(0, \phi_{\T^{\text{GW}}_n}(y_{\delta})\mid (\eta+i\delta)n)=
\overline{R}_{\mathcal{B}_{\text{ubble}}(\delta)}(0, \phi_{\T^{\text{GW}}_n}(y_{\delta})\mid (\eta+i\delta)n),
\]
where we used a natural extension of the notation $\overline{R}^{\omega_n}$ (as in~\eqref{chowchow}), and also
\[
R^{\omega_n}(0, \phi_{\T^{\text{GW}}_n}(y_{\delta}))=
R_{\mathcal{B}_{\text{ubble}}(\delta)}(0, \phi_{\T^{\text{GW}}_n}(y_{\delta})).
\]

Introduce 
\[
A_{\delta}(n):=\{\text{for any } z\in \T^{\text{GW}}_n \text{ with } \abs{z}\geq (\eta+i\delta)n , \text{ we have } \phi_{\T^{\text{GW}}_n}(z) \notin  \mathcal{B}_{\text{ubble}}(\delta)\},
\]
 which means that $ \mathcal{B}_{\text{ubble}}(\delta)$  are only made of images of points in $\T^{\text{GW}}_n$ whose generation is lower than $(\eta+i\delta)n$. The key observation is that on $A_{\delta}(n)$ we have
 \[
 R_{\mathcal{B}_{\text{ubble}}(y_{\delta})}(0, \phi_{\T^{\text{GW}}_n}(y_{\delta}))=\overline{R}_{\mathcal{B}_{\text{ubble}}(y_{\delta})}(0, \phi_{\T^{\text{GW}}_n}(y_{\delta})\mid (\eta+i\delta)n),
 \]
 which implies by the two previous equations that, on $A_{\delta}(n)$
 \begin{equation}\label{olamexicano12}
 \overline{R}^{\omega_n}(0, \phi_{\T^{\text{GW}}_n}(y_{\delta})\mid (\eta+i\delta)n)=R^{\omega_n}(0, \phi_{\T^{\text{GW}}_n}(y_{\delta})).
 \end{equation}

Finally, we can notice that $y_{\delta}$ is a cut-point for the restricted subgraph generated by embedding the $(\eta+i\delta)n$ first levels of $\T^{\text{GW}}_n$. This and Rayleigh's monotonicity principle imply that
\begin{equation}\label{olamexicano13}
 \overline{R}^{\omega_n}(0, \phi_{\T^{\text{GW}}_n}(y_{\delta})\mid (\eta+i\delta)n)\leq  \overline{R}^{\omega_n}(0, \phi_{\T^{\text{GW}}_n}(\overleftarrow{U_1^n}^{(\eta+i\delta)n}))\leq  \overline{R}^{\omega_n}(0, \phi_{\T^{\text{GW}}_n}(y_{\delta})\mid (\eta+i\delta)n) +\delta n.
 \end{equation}

The combination of~\eqref{olamexicano1}, \eqref{olamexicano12} and~\eqref{olamexicano13} imply that on $A_{\delta}(n)$, we see that
\[
\abs{\overline{R}^{\omega_n}(0, \phi_{\T^{\text{GW}}_n}(\overleftarrow{U_1^n}^{(\eta+i\delta)n}))-R^{\omega_n}(0, \phi_{\T^{\text{GW}}_n}(\overleftarrow{U_1^n}^{(\eta+i\delta)n}))} \leq 2\delta n.
\]


Hence, in order to prove the result, it will be enough to show that for any $\delta>0$ we have that 
\[
 \limsup_{n\to\infty}{\bf Q}_n[A_{\delta}(n)]=0,
 \]
 i.e.~showing that the ancestors of $U_1^n$ which are at a generation lower than $(\eta+(i-1/2)\delta)n$ are not in bubbles that contain images of points at a generation greater than $(\eta+i\delta)n$, except with vanishing probability.
 
Fix $\delta>0$, write $y_1^{(n,\delta)},\dots, y_{l(n,\delta)}^{(n,\delta)}$ for the vertices at generation $(\eta+(i-1/2)\delta)n$ that have descendants at generation $(\eta+i\delta)n$. For any $\epsilon>0$, we can choose, by Lemma~\ref{FUCKKKKKKKKKKK1}, $M<\infty$ and $\epsilon_1>0$ such that the event
\[
B(n)=\{l(n,\delta) \leq M\} \cap \Bigl\{\min_{i=1,\ldots, l(n,\delta)} \abs{\overrightarrow{\T^{\text{GW}}_n}_{y_i^{(n,\delta)}}} \geq \epsilon_1 n\Bigr\},
\]
verifies
 \[
\limsup_{n\to \infty} {\bf Q}_n[B(n)^c ] <\epsilon.
\]

On $B(n)$, by successively choosing uniform points $(U_j^n)_{j\geq 0}$ in the tree, we know that at each step we have probability at least $\epsilon_1$ to pick a uniform point which has $y_i^{(n,\delta)}$ as an ancestor. Since there are at most $M$ such $y_i^{(n,\delta)}$, we know there exists $J<\infty$ such that
\[
\limsup_{n\to \infty} {\bf Q}_n[B(n), \text{for any $i \leq l(n,\delta)$, } y_i^{(n,\delta)}\prec U_j^n \text{ for some $j\leq J$}]>1-\epsilon,
\]
and denote $C(n)^c$ the event $\{\text{for any $i \leq l(n,\delta)$, } y_i^{(n,\delta)}\prec U_j^n \text{ for some $j\leq J$}\}$.

Now, by applying Lemma~\ref{close_cut_point2} to the points $U_j^n$ for $j\leq J$, we know that for some $\epsilon'>0$ we have
\[
\lim_{n\to \infty} {\bf Q}_n\Bigl[ \max_{x\prec U_j^n} (\abs{x}-\abs{\overline{\pi}_{\T^{\text{GW}}_n}(x)}) \geq n^{(1-\epsilon')} \text{ for some $j\leq J$}\Bigr] =0.
\]

But we can notice that if $B(n)$ and $C(n)$ occur along with $\max_{x\prec U_j^n} (\abs{x}-\abs{\overline{\pi}_{\T^{\text{GW}}_n}(x)}) \leq n^{(1-\epsilon')} \text{ for all $j\leq J$}$, then we can see that 
\begin{itemize}
\item the only paths leading to generation $U_1^n$ go through the points $ y_i^{(n,\delta)}$ for $ i \leq l(n,\delta)$,
\item any $y_i^{(n,\delta)}$ is an ancestor of some $U_j^n$ for $j\leq J$ and as such $y_i^{(n,\delta)}$  has an ancestor which is a cut-point at distance less than $n^{(1-\epsilon')}$ (and hence this cut-point is at a generation at least $(\eta+i\delta)n -n^{(1-\epsilon')}$)
\end{itemize}

  This means that any path from level $(\eta+(i-1/2)\delta)n$ that reaches level $(\eta+i\delta)n$ has to cross one of the cut-points on the path to some $y_i^{(n,\delta)}$ with $ i \leq l(n,\delta)$. This is not compatible with $A_{\delta}(n)$. This means that
  \begin{align*}
  &  \limsup_{n\to\infty}{\bf Q}_n[A_{\delta}(n)] \\
  \leq & \limsup_{n\to \infty} {\bf Q}_n[B(n)^c ]+ \limsup_{n\to \infty} {\bf Q}_n[B(n),C(n)^c] \\ 
  & \qquad + \lim_{n\to \infty} {\bf Q}_n\Bigl[ \max_{x\prec U_j^n} (\abs{x}-\abs{\overline{\pi}_{\T^{\text{GW}}_n}(x)}) \geq n^{(1-\epsilon')} \text{ for some $j\leq J$}\Bigr] \\
  \leq & 2\epsilon
  \end{align*}
  for any $\epsilon>0$. Since $\epsilon$ is arbitrary this means that
  \[
  \limsup_{n\to\infty}{\bf Q}_n[A_{\delta}(n)]=0,
  \]
  which is what we needed to complete the proof.
\end{proof}

\subsubsection{A first result on the proportionality of the resistance and the tree distance}

We can deduce from the previous lemma
\begin{lemma}\label{res_tech}
For $\epsilon>0$ and $\eta\in (0,1)$, we have 
\[
\lim_{n\to 0} {\bf Q}_n\Bigl[\abs{\frac{R^{\omega_n}(0, \phi_{\T^{\text{GW}}_n}(U_1^n))}{ d^{\T^{\text{GW}}_n}(0,U_1^n) } -\rho_1} >\epsilon, \frac{d^{\T^{\text{GW}}_n}(0,U_1^n)}n\in [\eta,1-\eta]\Bigr]=0,
\]
and
\[
\lim_{n\to \infty} {\bf Q}_n\Bigl[\abs{\frac{d_{\omega_n}(0, \phi_{\T^{\text{GW}}_n}(U_1^n))}{ d^{\T^{\text{GW}}_n}(0,U_1^n)} -\rho_2} >\epsilon, d^{\T^{\text{GW}}_n}(0,U_1^n)\in [\eta ,1-\eta]\Bigr]=0,
\]
where $U_1^n$ is uniformly chosen among the vertices of $\T^{\text{GW}}_n$.
\end{lemma}

\begin{proof}
We will only prove the first point, the second one being similar.

Let us now fix $ i\in \N $ with $\frac{\eta}2 \leq \frac{\eta}2+i\delta \leq 1-\frac{\eta}2$.  Since the embedding diminishes the distances, we know that $\phi(\overleftarrow{U_1^n}^{(\eta+i\delta)n})$ is at a distance of at most $\delta n$ of $\phi(U_1^n)$, when $d^{\T^{\text{GW}}_n}(0,U_1^n)/n\in [\frac{\eta}2 +i\delta,\frac{\eta}2+(i+1)\delta)$ (by the definition of $\overleftarrow{U_1^n}^{(\frac{\eta}2+i\delta)n}$). Hence, since the resistance is a metric we can see that, when  $d^{\T^{\text{GW}}_n}(0,U_1^n)/n\in [\frac{\eta}2 +i\delta,\frac{\eta}2+(i+1)\delta)$, we have
\[
\abs{R^{\omega_n}(0, \phi_{\T^{\text{GW}}_n}(U_1^n))- R^{\omega_n}(0, \phi_{\T^{\text{GW}}_n}(\overleftarrow{U_1^n}^{(\frac{\eta}2+i\delta)n}))}\leq\delta n,
\]
this implies, using $d^{\T^{\text{GW}}_n}(0,U_1^n)/n\in [\frac{\eta}2 +i\delta,\frac{\eta}2+(i+1)\delta)$, that
\begin{align*}
\frac{R^{\omega_n}(0, \phi_{\T^{\text{GW}}_n}(\overleftarrow{U_1^n}^{(\frac{\eta}2+i\delta)n}))-\delta n}{(\frac{\eta}2 +(i+1)\delta)n}
 & \leq \frac{R^{\omega_n}(0, \phi_{\T^{\text{GW}}_n}(U_1^n))}{d^{\T^{\text{GW}}_n}(0,U_1^n)} \\
  & \leq \frac{R^{\omega_n}(0, \phi_{\T^{\text{GW}}_n}(\overleftarrow{U_1^n}^{(\eta+i\delta)n}))+\delta n}{(\frac{\eta}2 +i\delta)n}.
\end{align*}

Hence, on $d^{\T^{\text{GW}}_n}(0,U_1^n)/n\in [\frac{\eta}2 +i\delta,\frac{\eta}2+(i+1)\delta)$, we have
\begin{align*}
& \abs{\frac{R^{\omega_n}(0, \phi_{\T^{\text{GW}}_n}(U_1^n))}{d^{\T^{\text{GW}}_n}(0,U_1^n)}-\frac{R^{\omega_n}(0, \phi_{\T^{\text{GW}}_n}(\overleftarrow{U_1^n}^{(\frac{\eta}2+i\delta)n}))}{(\frac{\eta}2 +i\delta)n}}\\
\leq & \frac{\delta}{(\eta/2)} +\frac{\delta}{(\frac{\eta}2+(i+1)\delta)(\frac{\eta}2+i\delta)n} R^{\omega_n}(0, \phi_{\T^{\text{GW}}_n}(\overleftarrow{U_1^n}^{(\frac{\eta}2+i\delta)n})) \\
\leq & \frac{\delta}{(\eta/2)} +\frac{\delta}{(\frac{\eta}2+(i+1)\delta)(\frac{\eta}2+i\delta)},
\end{align*}
where we used that $R^{\omega_n}(0, \phi_{\T^{\text{GW}}_n}(\overleftarrow{U_1^n}^{(\frac{\eta}2+i\delta)n})) \leq n$ by Rayleigh's monotonicity principle since $\phi_{\T^{\text{GW}}_n}$ diminishes distances and $\frac{\eta}2+i\delta\leq 1$.

Now, fix $\epsilon>0$ and choose $\delta$ small enough (depending on $\eta$ and $\epsilon$) such that $\frac{\delta}{(\eta/2)} +\frac{\delta}{((\eta/2)+\delta)((\eta/2)+\delta)}\leq \frac{\epsilon}3$, which implies that 
\[
\abs{\frac{R^{\omega_n}(0, \phi_{\T^{\text{GW}}_n}(U_1^n))}{d^{\T^{\text{GW}}_n}(0,U_1^n)}-\frac{R^{\omega_n}(0, \phi_{\T^{\text{GW}}_n}(\overleftarrow{U_1^n}^{(\frac{\eta}2+i\delta)n}))}{(\frac{\eta}2 +i\delta)n}}\leq \frac{\epsilon}3.
\]

We can then use Lemma~\ref{res_tech1} (applied with $\epsilon/3$ instead of $\epsilon$) to see that
\[
\lim_{n\to \infty} {\bf Q}_n\Bigl[\abs{\frac{\overline{R}^{\omega_n}(0, \phi_{\T^{\text{GW}}_n}(\overleftarrow{U_1^n}^{(\frac{\eta}2+i\delta)n}))}{ (\frac{\eta}2+i\delta)n} -\rho_1} >\frac{\epsilon}3, \frac{d^{\omega_n}(0,U_1^n)}n\in [\frac{\eta}2 +i\delta,\frac{\eta}2+(i+1)\delta)\Bigr]=0,
\]
 and Lemma~\ref{ASH} ( provided $2\delta \leq \frac{2\epsilon}{3\eta}$) to see that 
 \begin{align*}
 & \limsup_{n\to\infty}{\bf Q}_n\Bigl[\frac{\abs{\overline{R}^{\omega_n}(0, \phi_{\T^{\text{GW}}_n}(\overleftarrow{U_1^n}^{(\frac{\eta}2+i\delta)n}))-R^{\omega_n}(0, \phi_{\T^{\text{GW}}_n}(\overleftarrow{U_1^n}^{(\frac{\eta}2+i\delta)n}))}}{(\frac{\eta}2 +i \delta)n} \geq \frac{\epsilon} 3   \\
& \qquad \qquad \qquad \qquad\qquad \qquad \text{ and }\frac{d^{\omega_n}(0,U_1^n)}n\in [\frac{\eta}2 +i\delta,\frac{\eta}2+(i+1)\delta)\Bigr]=0
\end{align*}
 
 The last three equations imply that for any $i\in \N$ such that for $\delta$ small enough and $\frac{\eta}2 \leq \frac{\eta}2+i\delta \leq 1-\frac{\eta}2$ we have
\[
\lim_{n\to \infty} {\bf Q}_n\Bigl[\abs{\frac{R^{\omega_n}(0, \phi_{\T^{\text{GW}}_n}(U_1^n))}{d^{\T^{\text{GW}}_n}(0,U_1^n)} -\rho_1} >\epsilon, d^{\omega}(0,U_1^n)\in [\frac{\eta}2 +i\delta,\frac{\eta}2+(i+1)\delta)\Bigr]=0,
\]
which implies the lemma by summing over $i$ (with $\frac{\eta}2 \leq \frac{\eta}2+i\delta \leq 1-\frac{\eta}2$) if we impose further that $\delta<\frac{\eta}2$.
\end{proof}

\subsubsection{Proving the proportionality of the distance and the resistance distance for LCBRW conditioned on height}

Write 
\begin{equation}\label{de_longA}
A(\eta):=\Bigl\{\abs{\frac{R^{\omega_n}(0, \phi_{\T^{\text{GW}}_n}(U_1^n))}{ d^{\T^{\text{GW}}_n}(0,U_1^n)} -\rho_1} >\epsilon, \frac{d^{\T^{\text{GW}}_n}(0,U_1^n)}{H(\T^{\text{GW}}_n)}\in [\eta ,1-\eta ]\Bigr\},
\end{equation}
and
\begin{equation}\label{de_longB}
B(\eta):=\Bigl\{\abs{\frac{d_{\omega_n}(0, \phi_{\T^{\text{GW}}_n}(U_1^n))}{ d^{\T^{\text{GW}}_n}(0,U_1^n)} -\rho_2} >\epsilon, \frac{d^{\T^{\text{GW}}_n}(0,U_1^n)}{H(\T^{\text{GW}}_n)}\in [\eta ,1-\eta ]\Bigr\},
\end{equation}
where $U_1^n$ is uniformly chosen among the vertices of $\T^{\text{GW}}_n$. Let us emphasize the fact that the events $A(\eta)$ and $B(\eta)$ are not defined in terms of $n$.

We can generalize the previous result to obtain 
\begin{lemma}\label{res_techh}
For $\epsilon>0$ and $\eta\in (0,1)$, we have 
\[
\lim_{n\to \infty} {\bf Q}_n[A(\eta)]=0,
\]
and
\[
\lim_{n\to \infty} {\bf Q}_n[B(\eta)]=0.
\]
\end{lemma}

\begin{proof}
For notational convenience, we will use $H$ instead of $H(\T^{\text{GW}}_n)$ during this proof.

 Write ${\bf Q}_{n,2}:={\bf P}[\cdot \mid H \in [n,2n)]$. We have 
\begin{align*}
&{\bf Q}_{n,2}\Bigl[\abs{\frac{R^{\omega_n}(0, \phi_{\T^{\text{GW}}_n}(U_1^n))}{ d^{\T^{\text{GW}}_n}(0,U_1^n) } -\rho_1} >\epsilon, \frac{d^{\T^{\text{GW}}_n}(0,U_1^n)}n\in \Bigl[\frac{\eta}2,1-\frac{\eta}2\Bigr]\Bigr] \\
= & \frac 1{{\bf P}[ H \in [n,2n)]} \\
& \times {\bf P}\Bigl[\abs{\frac{R^{\omega_n}(0, \phi_{\T^{\text{GW}}_n}(U_1^n))}{ d^{\T^{\text{GW}}_n}(0,U_1^n) } -\rho_1} >\epsilon, \frac{d^{\T^{\text{GW}}_n}(0,U_1^n)}n\in \Bigl[\frac{\eta}2,1-\frac{\eta}2\Bigr], H \in [n,2n)\Bigr] \\
\leq &\frac 1{{\bf P}[ H \in [n,2n)]} \\
&\qquad \times {\bf P}\Bigl[\abs{\frac{R^{\omega_n}(0, \phi_{\T^{\text{GW}}_n}(U_1^n))}{ d^{\T^{\text{GW}}_n}(0,U_1^n) } -\rho_1} >\epsilon, \frac{d^{\T^{\text{GW}}_n}(0,U_1^n)}n\in \Bigl[\frac{\eta}2,1-\frac{\eta}2\Bigr], H \geq n\Bigr] \\
\leq &\frac {{\bf P}[H\geq n]}{{\bf P}[ H \in [n,2n)]} {\bf Q}_n\Bigl[\abs{\frac{R^{\omega_n}(0, \phi_{\T^{\text{GW}}_n}(U_1^n))}{ d^{\T^{\text{GW}}_n}(0,U_1^n) } -\rho_1} >\epsilon, \frac{d^{\T^{\text{GW}}_n}(0,U_1^n)}n\in \Bigl[\frac{\eta}2,1-\frac{\eta}2\Bigr]\Bigr],
\end{align*}
but since ${\bf P}[H\geq n] \sim C n^{-1}$, by~\cite{KNS}, we know that $ \frac{{\bf P}[H\geq n]}{{\bf P}[n\leq H < 2n]} \leq C$, so we can use Lemma~\ref{res_tech} to see that
\[
\lim_{n\to \infty} {\bf Q}_{n,2}\Bigl[\abs{\frac{R^{\omega_n}(0, \phi_{\T^{\text{GW}}_n}(U_1^n))}{ d^{\T^{\text{GW}}_n}(0,U_1^n) } -\rho_1} >\epsilon, \frac{d^{\T^{\text{GW}}_n}(0,U_1^n)}n\in \Bigl[\frac{\eta}2,1-\frac{\eta}2\Bigr]\Bigr] =0.
\]

Now, notice that conditionally $H \in [n,2n)$, we have 
\[
\Bigl\{\frac{d^{\T^{\text{GW}}_n}(0,U_1^n)}H\in [\eta,1-\eta]\Bigr\}\subset \Bigl\{\frac{d^{\T^{\text{GW}}_n}(0,U_1^n)}n \in [\frac{\eta}2,1-\frac{\eta}2]\Bigr\},
\]
 we can use this and the previous equation to see that 
\begin{equation}\label{this_article_is_long}
\lim_{n\to \infty} {\bf Q}_{n,2}[A(\eta)] =0.
\end{equation}

Fix $N\in \N$. The quantity we want to study is 
\begin{align*}
& {\bf Q}_n[A(\eta)]  \\
= & \frac 1{{\bf P}[H\geq n]} {\bf P}[A(\eta), H \geq n] \\
\leq &\frac 1{{\bf P}[H\geq n]}  \Bigl( \sum_{k=0}^{N-1} {\bf P}[A(\eta), 2^k n \leq H < 2^{k+1}n] +{\bf P}[H\geq 2^{N}n]\Bigr) \\
\leq & \sum_{k=0}^{N-1} \frac{{\bf P}[H \geq  2^{k}n]}{{\bf P}[H\geq n]}  {\bf Q}_{2^kn,2}[A(\eta)] + \frac{{\bf P}[H \geq  2^{N}n]}{{\bf P}[H\geq n]},
\end{align*}
noticing that $\frac{{\bf P}[H \geq  2^{k}n]}{{\bf P}[H\geq n]} \leq C 2^{-k}$, we can get the upper bound
\begin{equation}\label{no_more_labels}
{\bf Q}_{n}[A(\eta)] \leq N \sup_{k\geq 0} {\bf Q}_{2^kn,2}[A(\eta)] +  C 2^{-N}.
 \end{equation}

 By~\eqref{this_article_is_long}, we know that
 \[
 \lim_{n\to \infty} \sup_{k\geq 0} {\bf Q}_{2^kn,2}[A(\eta)]  =0,
 \]
 so by choosing $N:=(\sup_{k\geq 0} {\bf Q}_{2^kn,2}[A(\eta)])^{-1/2}$, we have $N\to_{n\to \infty} \infty$ and~\eqref{no_more_labels} implies 
\[
\lim_{n\to \infty}  {\bf Q}_n[A(\eta)]=0,
\]
which is the result we were trying to prove.
\end{proof}

We have
\begin{lemma} \label{uniform_points_bulk}
For any $\epsilon>0,$ there exists $\eta\in(0,1)$, we have
\[
\limsup_{n\to \infty} {\bf Q}_n\Bigl[\frac{d^{\T^{\text{GW}}_n}(0,U_1^n)}{H(\T^{\text{GW}}_n)} \notin [\eta,1-\eta]\Bigr]<\epsilon,
\]
where $U_1^n$ is chosen uniformly on $\T^{\text{GW}}_n$.
\end{lemma}
\begin{proof}
Let $h_{\T^{\text{GW}}_n}$ be the height function associated to $\T^{\text{GW}}_n$ and $\hat{\bold{e}}$ be a Brownian excursion conditioned on having height larger than $\frac{\sigma_Z}{2}$.
 Set $R:=\sup\{t\geq0: \hat{\bold{e}}_t>0\}$. It follows from the proof of Corollary 1.13 in \cite{le2005random} that
 \begin{equation}
 ((n^{-1}h_{\T^{\text{GW}}_n}(n^2t))_{t\geq0},n^{-1}\abs{\T^{\text{GW}}_n}) \stackrel{n\to\infty}{\to} ((\hat{\bold{e}}_t)_{t\geq0}, R),
 \end{equation}  
 in distribution under $\bold{Q}_n$, where the convergence of the first coordinate is in the topology of uniform convergence.  It is not hard to see from the display above that
\begin{equation}\label{eq:conejitosQ}
 {\bf Q}_n\Bigl[\frac{d^{\T^{\text{GW}}_n}(0,U_1^n)}{H(\T^{\text{GW}}_n)} \in [\eta,1-\eta]\Bigr] \stackrel{n\to\infty}{\to} E\left[\frac{1}{R}\text{Leb}(\{t\in[0,R]: \frac{\abs{\hat{\bold{e}}_t}}{M}\in [\eta, 1-\eta ]\})\right], 
 \end{equation}
From standard facts about Brownian motion, (e.g., existence of local times) and the dominated convergence Theorem we have that
\begin{equation}\label{eq:conejitoseta}
\lim_{\eta\to0}E\left[\frac{1}{R}\text{Leb}(\{t\in[0,R]: \frac{\abs{\hat{\bold{e}}_t}}{M}\in [\eta, 1-\eta ]\})\right]=1.
\end{equation}
The result follows from \eqref{eq:conejitosQ} and \eqref{eq:conejitoseta} 
\end{proof}

Finally we obtain the main result of this section

\begin{lemma}\label{res_techhh}
For $\epsilon>0$, we have 
\[
\lim_{n\to \infty} {\bf Q}_n\Bigl[\abs{\frac{R^{\omega_n}(0, \phi_{\T^{\text{GW}}_n}(U_1^n))}{ d^{\T^{\text{GW}}_n}(0,U_1^n)} -\rho_1} >\epsilon\Bigr]=0,
\]
and
\[
\lim_{n\to \infty} {\bf Q}_n\Bigl[\abs{\frac{d^{\omega_n}(0, \phi_{\T^{\text{GW}}_n}(U_1^n))}{ d^{\T^{\text{GW}}_n}(0,U_1^n)} -\rho_2} >\epsilon\Bigr]=0.
\]
\end{lemma}

\begin{proof}
We prove only the first result. Fix $\epsilon'>0$ and chose $\eta>0$ given by Lemma~\ref{uniform_points_bulk} associated to that $\epsilon'$. We know that
\[
{\bf Q}_n\Bigl[\abs{\frac{R^{\omega_n}(0, \phi_{\T^{\text{GW}}_n}(U_1^n))}{ d^{\T^{\text{GW}}_n}(0,U_1^n)} -\rho_1} >\epsilon\Bigr] \leq {\bf Q}_n[A(\eta)] +{\bf Q}_n\Bigl[\frac{d^{\T^{\text{GW}}_n}(0,U_1^n)}{H(\T^{\text{GW}}_n)} \notin [\eta,1-\eta]\Bigr],
\]
and which implies by Lemma~\ref{uniform_points_bulk} and Lemma~\ref{res_techh}
\[
\limsup_{n\to \infty} {\bf Q}_n\Bigl[\abs{\frac{R^{\omega_n}(0, \phi_{\T}(U_1^n))}{ d^{\T^{\text{GW}}_n}(0,U_1^n)} -\rho_1} >\epsilon\Bigr] \leq  \epsilon',
\]
which is valid for all $\epsilon'>0$. This proves the lemma.
\end{proof}

\subsection{Resistance on LCBRW conditioned on size} \label{proof_res}

Our proof will proceed in two steps. We first condition on the size being larger than $n$ and then we will transfer this to the size being exactly $n$.

\subsubsection{Conditioning on the size being larger than $n$}

We will start by proving a result conditioned on $\abs{\T}\geq n$. First we need a technical Lemma
\begin{lemma}\label{arggg} There exists $C>0$ such that
for all $\epsilon'>0$, we have
\[
\limsup_{n\to \infty} {\bf P}[H(\T)\leq (\epsilon'n)^{1/2} \mid \abs{\T}\geq n]\leq C \epsilon',
\]
\end{lemma}
\begin{proof}
The proof is based on the convergence of the height function of $\T$ to a Brownian excursion. Let $h_{\T}$ denote the height function of $\T$ as in the proof of Lemma \ref{uniform_points_bulk}. 
Using the notation of the proof of Lemma \ref{uniform_points_bulk}, let
\[L':=\inf \{t>0: \beta_t=0 \text{ and }\beta_{t+s}>0 \text{ for all }s< 1 \}\]
and
\[R':=\inf\{t>L':\beta_t=0\}\]
 Let $(\bold{e}^1_t)_{t\geq0}$ be a Brownian excursion conditioned on having duration larger than $1$.  i.e. $(\bold{e}_t^1)_{t\geq0}$ is distributed as $\left(\frac{2}{\sigma} \beta_{L'+t} 1_{\{t\leq R'-L'\}}\right)_{t\geq0}$. The same arguments leading to \eqref{eq:jtcvgc} can be used to prove that  $h_\T$ conditioned on $\abs{\T}\geq n$ converge to a $(\bold{e}^1_t)_{t\geq0}$. Therefore
\begin{equation}\label{eq:traile}
 \lim_{n\to\infty}{\bf P}[H(\T)\leq (\epsilon'n)^{1/2} \mid \abs{\T}\geq n]=P\left[\max_{t\geq 0}\bold{e}^1_t \leq \frac{\sigma}{2}\epsilon'\right].
 \end{equation}
As the displays above suggest, we will need to control the maximum of a Brownian excursion. Theorem 1 of \cite{vervaat1979relation} establishes a relation between the Brownian bridge and the normalized Brownian excursion $(\bold{e}_t)_{t\in[0,1]}$. More precisely, let $(W^b(t))_{t\in[0,1]}$ be a Brownian bridge and $\tau\in[0,1]$ the point where the Brownian bridge attains its minimum. Then $(W^b(\tau+t\text{ mod }1)-W^b(\tau))_{t\in[0,1]}$ has the law of a normalized Brownian excursion. Therefore
\begin{equation}\label{eq:conejitosderaza}
P\left[\max_{t\in[0,1]} W^b(t)-\min_{t\in[0,1]} W^b(t)< \epsilon\right]= P\left[\max_{t\in[0,1]} \bold{e}_t < \epsilon\right].
\end{equation} 
 Let $W$ be a standard, one-dimensional Brownian motion. A simple analysis yields that if $W(1)\notin [2W(1/2)-2\epsilon,2W(1/2)+2\epsilon]$, then the Brownian bridge $W_t-tW_1$ evaluated at time $1/2$ is outside the ball of radius $\epsilon$ centered at the origin. That is 
\begin{equation}\label{eq:Brownianbridge}
\abs{W(1/2)-1/2W(1)}>\epsilon.
\end{equation}
 Therefore 
 \[
 \begin{aligned}
 P\left[ W^b(1/2)<\epsilon\right] \leq& P\left[W(1)\in [2W(1/2)-2\epsilon,2W(1/2)+2\epsilon]\right]\\
 \leq& 4\epsilon \times \sup_{x\in\R} p_{\frac{1}{2}}(x),
 \end{aligned}
 \]
where $p_t(x)$ is the Gaussian density of $W$ at time $t$. Since $\sup_{x\in\R}p_{\frac{1}{2}}(x)< \infty$, the display above yields that there exists $C>0$ such that 
\[P[W^{b}(1/2)<\epsilon ]\leq C\epsilon.\]
That, together with \eqref{eq:conejitosderaza} yields that 
\begin{equation}\label{eq:acho}
P\left[\max_{t\in[0,1]} \bold{e}_t < \epsilon\right]\leq C \epsilon.
\end{equation}
 We need to replace the normalized Brownian excursion $\bold{e}$ in the display above by the Brownian excursion conditioned on having duration larger than $1$, $\bold{e}^1$. But that step can be easily justified using the known fact that $(\tau_0^{-1/2} \bold{e}^1_{\tau_0t})_{t\in[0,1]}$ is distributed as $(\bold{e}_t)_{t\in[0,1]}$, where $\tau_0:=\inf\{t>0: \bold{e}^1_t=0\}$. Notice that, since $\bold{e}^1$ is conditioned to have duration larger than $1$, we have that $\tau_0>1$ almost surely and therefore
 \[
 P\left[\max_{t\geq 0} \bold{e}^{1}_t < \epsilon\right]
  =
 P\left[\max_{t\geq 0} \tau_0^{1/2} \bold{e}_t < \epsilon\right]\\
  \leq P\left[\max_{t\geq 0} \bold{e}_t < \epsilon\right].
 \]
 Hence, the lemma follows from \eqref{eq:acho} and \eqref{eq:traile}.
\end{proof}

This allows us to prove
\begin{lemma}\label{prop_res0000}
For $\epsilon>0$, we have 
\[
\lim_{n\to \infty} {\bf P}\Bigl[\abs{\frac{R^{\omega}(0, \phi_{\T}(U_1^n))}{ d^{\T}(0,U_1^n)} -\rho_1} >\epsilon\mid \abs{\T}\geq n\Bigr]=0,
\]
and
\[
\lim_{n\to \infty} {\bf P}\Bigl[\abs{\frac{d_{\omega}(0, \phi_{\T}(U_1^n))}{ d^{\T}(0,U_1^n)} -\rho_2} >\epsilon\mid \abs{\T}\geq n\Bigr]=0.
\]
\end{lemma}
\begin{proof}
Fix $\epsilon'>0$. We have 
\begin{align*}
& {\bf P}\Bigl[\abs{\frac{R^{\omega}(0, \phi_{\T}(U_1^n))}{ d^{\T}(0,U_1^n)} -\rho_1} >\epsilon\mid \abs{\T}\geq n\Bigr]\\
\leq & {\bf P}\Bigl[\abs{\frac{R^{\omega}(0, \phi_{\T}(U_1^n))}{ d^{\T}(0,U_1^n)} -\rho_1} >\epsilon, H(\T)\geq (\epsilon'n)^{1/2} \mid \abs{\T}\geq n\Bigr] \\
& \qquad \qquad  \qquad \qquad\qquad \qquad +{\bf P}[H(\T)\leq (\epsilon'n)^{1/2} \mid \abs{\T}\geq n] \\
\leq & \frac 1{{\bf P}[\abs{\T}\geq n]}{\bf P}\Bigl[\abs{\frac{R^{\omega}(0, \phi_{\T}(U_1^n))}{ d^{\T}(0,U_1^n)} -\rho_1} >\epsilon,  H(\T)\geq (\epsilon'n)^{1/2} , \abs{\T}\geq n\Bigr] \\ & \qquad\qquad   \qquad \qquad\qquad \qquad+{\bf P}[H(\T)\leq (\epsilon'n)^{1/2} \mid \abs{\T}\geq n] \\
\leq & \frac {{\bf P}[H(\T)\geq (\epsilon'n)^{1/2}]}{{\bf P}[\abs{\T}\geq n]} {\bf P}\Bigl[\abs{\frac{R^{\omega}(0, \phi_{\T}(U_1^n))}{ d^{\T}(0,U_1^n)} -\rho_1} >\epsilon\mid H(\T)\geq (\epsilon'n)^{1/2}\Bigr] \\ & \qquad \qquad  \qquad \qquad \qquad \qquad+{\bf P}[H(\T)\leq (\epsilon'n)^{1/2} \mid \abs{\T}\geq n].
\end{align*}

Using ${\bf P}[H(\T^{\text{GW}})\geq n]\sim C n^{-1}$ (see~\cite{KNS}), ${\bf P}[\abs{\T^{\text{GW}}}\geq n] \sim C n^{-1/2}$ (Lemma 2.1.4 in \cite{Kolchin}) we see that $\frac {{\bf P}[H(\T)\geq (\epsilon'n)^{1/2}]}{{\bf P}[\abs{T}\geq n]} \leq C(\epsilon')^{-1/2}$. Using this, the previous computation, Lemma~\ref{arggg} and Lemma~\ref{res_techhh} we can see that 
\[
\limsup_{n\to \infty} {\bf P}\Bigl[\abs{\frac{R^{\omega}(0, \phi_{\T}(U_1^n))}{ d^{\T}(0,U_1^n)} -\rho_1} >\epsilon\mid \abs{\T}\geq n\Bigr] \leq C\epsilon',
\]
and since this is true for all $\epsilon'>0$ the lemma is proved.
\end{proof}

\subsubsection{Conditioning on the size being $n$}\label{subsubsection_searchdepth}

 Let us define the breadth first search of a Galton-Watson tree. We consider the vertices level-by-level, and from left to right within each level. This induces an ordering on the tree which we denote $(v^{\T}_i)_{i\leq \abs{\T}}$. The breadth first search keeps a queue of vertices $Q_i$ with $Q_0=1$ and the recursion $Q_i=Q_{i-1}-1+Z_i$, with $Z_i$ which are i.i.d.~random variables with the same law as $Z$. The breadth first search stops when the tree is completely explored when $Q_j=0$. 
 
 This process is a way of encoding a tree with a random walk. We can see that a tree has size at least $n$ if, and only if, $Q_j>0$ for all $j<n$, and it has size exactly $n$ if, and only if, we also have $Q_n=0$.

We will also need another way of encoding the tree known as the depth first search process around $\T^{\text{GW}}_n$, which corresponding to exploring the tree by going around it. This process $(w_n(\frac{i}{2(n-1)}))_{i\leq 2 (n-1)}$ takes values in $V(\T^{\text{GW}}_n)$ and is defined by starting at the root then defining recursively $w_n(\frac{i+1}{2(n-1)})$ in the following manner:
\begin{itemize}
\item if $w_n(\frac{i}{2(n-1)})$ has offsprings which has not yet been visited by the process, then $w_n(\frac{i+1}{2(n-1)})$ will be the first of those offsprings according to the lexicographical order,
\item otherwise $w_n(\frac{i+1}{2(n-1)})$ is the ancestor of $w_n(\frac{i}{2(n-1)})$.
\end{itemize}
Using this we can define the search depth process by setting $\hat{w}_n(i)=d^{\T^{\text{GW}}_n}(\text{root},w_n(i))$. We extend the domain of definition of $\hat{w}_n$ to $[0,1]$ by linear interpolation.

Before moving on to our main result we need three more technical estimates. For any $\eta>0$, we write
\[
C(\eta)=\{\text{Card}(\{i\leq \abs{\T},\ d_{\T}(0,v^{\T}_i)\leq d(0,U_1^n)\}) \leq (1-\eta)n\}
\]
where $U_1^n$ is uniformly chosen over $\T$.

\begin{lemma}\label{stupid_estimate_n_plus_one}
For any $\epsilon'>0$ there exists $\eta>0$, we have that
\[
\limsup_{n\to \infty} {\bf P}_n[C(\eta)^c]\leq \epsilon'.
\]
\end{lemma}
\begin{proof}
Let
\begin{equation}
H^\ast(\eta)=\max\{h\in \N: |\{v\in V(\T^{\text{GW}}_n): d^{\T^{\text{GW}}_n}(\text{root},v)\leq h\}|\leq (1-\eta)n
\}.
\end{equation}
It is easy to see that if $d^{\T^{\text{GW}}_n}(\text{root},U^1_n)\leq H^\ast(\eta)$ then $C(\eta)$ holds. Therefore
\begin{equation}\label{eq:madliblib}
{\bf P}_n[C(\eta)^c]\leq {\bf P}_n[d^{\T^{\text{GW}}_n}(\text{root},U^1_n)> H^\ast(\eta)].
\end{equation}
Moreover, by definition of $H^\ast(\eta)$ and using that $|V(\T^{\text{GW}}_n)|=n$ it follows that 
\[|\{ v\in V(\T^{\text{GW}}_n):d^{\T^{\text{GW}}_n}(\text{root},v)> H^\ast(\eta)\}|\leq \eta n.\]
Therefore, since $U^1_n$ is uniform on the $n$ vertices of $\T^{\text{GW}}_n$, 
\[
{\bf P}_n[d^{\T^{\text{GW}}_n}(\text{root},U^1_n)> H^\ast(\eta)]\leq \eta.
\]
The display above together with \eqref{eq:madliblib} proves the lemma. \end{proof}

\begin{lemma}\label{stupid_estimate_n_plus_two}
For any $\epsilon'>0$ we have
\[
\limsup_{n\to \infty} {\bf P}_n[d^{\T^{\text{GW}}_n}(0,U_1^n) \leq \delta n^{1/2} ]=0.
\]
where $U_1^n$ is uniformly distributed on $\T^{\text{GW}}_n$.
\end{lemma}
\begin{proof}

Let $\T^{\text{GW}}_n$ be a Galton-Watson tree conditioned on $\abs{\T^\text{GW}_n}=n$.
 Let us introduce the \emph{height profile} of 
\[H^n(i)=\#\{x\in V(\T^{\text{GW}}_n) : d^{\T^{\text{GW}}_n}(\text{root}, x)=i\},\quad i\in \N.\]
We extend the domain of $H^n$ by linear interpolation.
Recall that $(\hat{w}_n(i))_{i\leq 2(n-1)}$ denotes the search depth process of $\T^{\text{GW}}_n$.
 Let $(\bold{e}_t)_{t\in [0,1]}$ be a normalized Brownian excursion and $(\frak{l}(x))_{x\in\R_+}$ be its local time at time $1$.
More precisely, $\frak{l}\in C(\R_+,\R)$ is the random function which satisfies
\[\text{Leb}(t\in[0,1]:\bold{e}_t\in A)=\int_A \frak{l}(x) dx, \]
for any Borelian $A$.
Using that $\hat{w}_n$ scales to the normalized Brownian excursion $\bold{e}$ (see \cite{Al3} Theorem 23) and that $H^n$ scales to $\frak{l}$ (see \cite{drmota1997profile} Theorem 1.1), using the same idea as in Theorem 3 of \cite{aldous1998Brownian}, we can get the joint convergence 
\begin{equation}\label{eq:jointconvergence}
\left((n^{-1/2} \hat{w}_n(t))_{t\in[0,1]},(n^{-1/2}H^n(\sqrt{n}x)_{x\geq0})\right)\stackrel{n\to\infty}{\to} \left(\left(\frac{2}{\sigma}\bold{e}_t\right)_{t\in[0,1]},\left(\frac{\sigma}{2} \frak{l}(x\sigma/2)\right)_{
x\in\R}\right)
\end{equation} 
weakly in $C([0,1],\R)\times C(\R_+,\R)$ endowed with the uniform topology.
Since $\bold{e}$ is an excursion of duration $1$ and $\frak{l}$ is the local time of $\bold{e}$, it follows that
\[\int_{\R}\bold{l}(x) dx =1.\]
In particular,
\begin{equation}\label{eq:mes}
\int_{0}^{h} \frac{\sigma}{2}\frak{l}(t\sigma/2) dt=1.
\end{equation}



By \eqref{eq:jointconvergence} we have that
\[
\limsup_{n\to\infty}{\bf P}_n[d^{\T^{\text{GW}}_n}(0,U_1^n) \leq \delta n^{1/2} ]\leq E\left[\int_0^{\delta} \frac{\sigma}{2}\frak{l}(x\sigma/2)dx\right],
\]
for any $\delta>0$. The result follows by taking $\delta$ to $0$.
\end{proof}

Finally, we do the third argument showing that $\overline{R}(\cdot,\cdot)$ and $R(\cdot,\cdot)$ are close. Fix $\epsilon>0$, write $y_1^{(n,\epsilon)},\dots, y_{l(n,\epsilon)}^{(n,\epsilon)}$ for the vertices at generation $\abs{U_1^n}-\frac{\epsilon}2 n^{1/2}$ that have descendants at generation $\abs{U_1^n}$. If $\abs{U_1^n}<\frac{\epsilon}2 n^{1/2}$ we set $l(n,\epsilon)=0$.

  \begin{lemma} \label{FUCKKKKKKKKKKK}
For any $\epsilon>0$, we have that 
\begin{equation}\label{eq:barriewhite}
\limsup_{K\to \infty} \limsup_{n\to \infty}{\bf P}_n[ l(n,\epsilon) >K]=0,
\end{equation}
and
\begin{equation}\label{eq:lordquas}
\limsup_{\delta \to 0}\limsup_{n\to \infty} {\bf P}_n\Bigl[ \min_{i=1,\ldots, l(n,\epsilon)} \abs{\overrightarrow{\T^{\text{GW}}_n}_{y_i^{(n,\epsilon)}}} <\delta n\Bigr] =0.
\end{equation}

The same result holds under the measure ${\bf P}[\cdot \mid \abs{\T} \geq n]$.
\end{lemma}
\begin{proof}
\emph{Proof of display \eqref{eq:barriewhite}}
Let $h_{\T^\text{GW}_n}$ be the height function associated to $\T^{\text{GW}}_n$ as in the proof of Lemma \ref{uniform_points_bulk}.
It is not hard to notice that $l(n,\epsilon)$ equals the number of up-crossings of the interval $[|U_1|-\frac{\epsilon}{2} n^{-1/2}, |U_1|]$ by $h_{\T^{\text{GW}}_n}$. Therefore, from 
Theorem 1.15 of \cite{le2005random} we have that
\begin{equation}\label{eq:olodum} (n^{-1/2}h_{\T^\text{GW}_n}(\lfloor tn\rfloor))_{t\in [0,1]} \stackrel{n\to\infty}{\to} (\frac{2}{\sigma_Z} \bold{e}_t)_{t\in[0,1]} 
\end{equation} in distribution, where $\bold{e}$ is a normalized Brownian excursion (and we recall that $\sigma_Z^2$ is the variance of the offspring distribution of $\T^{\text{GW}}$). Hence, $l(n,\epsilon)$ converges in distribution, as $n\to\infty$, to the number of up-crossings of $[\bold{e}(U)-\frac{\sigma_Z}{2}\frac{\epsilon}{2},\bold{e}(U)]$ by $\bold{e}$, where $U$ is a uniformly chosen point of $[0,1]$ independent of $\bold{e}$. But, since $\bold{e}$ is continuous, the said number of up-crossings is almost surely finite. This finishes the proof.
 
\emph{Proof of display \eqref{eq:lordquas}} 
Using, as above, that each $y_i^{(n,\epsilon)}$ is associated to an up-crossing of the interval $[|U_1|-\frac{\epsilon}{2},|U_1|]$ by $h_{\T^{\text{GW}}_n}$, we see that $\abs{\overrightarrow{\T^{\text{GW}}_n}_{y_i^{(n,\epsilon)}}}$ equals the duration of the excursion above level $|U_1|$ corresponding to the said up-crossing.
Again, by the convergence of $h_{\T^{\text{GW}}_n}$ in display \eqref{eq:olodum}, we get that $n^{-1}\abs{\overrightarrow{\T^{\text{GW}}_n}_{y_i^{(n,\epsilon)}}}$ converges in distribution of the duration of the excursion of $\bold{e}$ associated to an up-crossing of $[|U|-\frac{\sigma_Z}{2}\frac{\epsilon}{2},|U|]$ by $\bold{e}$ (where $U$ is an uniformly chosen point of $[0,1]$ independent of $\bold{e}$). Since $\bold{e}$ is uniformly continuous, the minimum duration of those excursions is positive. The result follows. 
\end{proof}

This allows us to compare $\overline{R}(\cdot,\cdot)$ and $R(\cdot,\cdot)$.
\begin{lemma}\label{go_france_boo_germany}
For any $\epsilon>0$ we have
\[
\limsup_{n\to \infty}{\bf P}_n\Bigl[\frac{\abs{\overline{R}^{\omega_n}(0, \phi_{\T^{\text{GW}}_n}(U_1^n))-R^{\omega_n}(0, \phi_{\T^{\text{GW}}_n}(U_1^n))}}{ d^{\T}(0,U_1^n)}  >\epsilon\Bigr]=0,
\]
and
\[
\limsup_{n\to \infty}{\bf P}_n\Bigl[\frac{\abs{\overline{d}_{\omega}(0, \phi_{\T^{\text{GW}}_n}(U_1^n))-d_{\omega}(0, \phi_{\T^{\text{GW}}_n}(U_1^n))}}{ d^{\T}(0,U_1^n)}  >\epsilon\Bigr]=0.
\]

The same result holds under ${\bf P}[\cdot \mid \abs{T}\geq n]$. 
\end{lemma}

\begin{proof}
We only prove the result for the resistance since the other result is similar but easier.

 On $\max_{x\prec U_1^n} (\abs{x}-\abs{\overline{\pi}_{\T^{\text{GW}}_n}(x)}) \leq n^{(1-\epsilon')/2}$ for $\epsilon'>0$, we know that for any $\epsilon_1>0$ there exists a point with loopless image $y_{\epsilon_1} \prec U_1^n$ such that $\epsilon_1 n^{1/2}\leq \abs{U_1^n}-\abs{y_{\epsilon_1}} \leq 2\epsilon_1 n^{1/2}$. Obviously, by Rayleigh's monotonicity principle and the fact that $\phi_{\T^{\text{GW}}_n}(y_{\epsilon_1})$ is (the image of a loopless point and hence) a cut-point lying between $0$ and $U_1^n$, we have that
\begin{equation}\label{olamexicano}
R^{\omega_n}(0, \phi_{\T^{\text{GW}}_n}(y_{\epsilon_1}))\leq  R^{\omega_n}(0, \phi_{\T^{\text{GW}}_n}(U_1^n))\leq  R^{\omega_n}(0, \phi_{\T^{\text{GW}}_n}(y_{\epsilon_1})) +2\epsilon_1 n^{1/2}.
 \end{equation}

  We will now turn our gaze to $R^{\omega_n}(0, \phi_{\T^{\text{GW}}_n}(y_{\epsilon_1}))$. For this, we define  the graph 
 \begin{align*}
  \mathcal{B}_{\text{ubble}}(\epsilon_1)  :=  & \{z\in \omega,\ z\text{ is connected without using cut-points to } \phi_{\T^{\text{GW}}_n}(z')  \\
& \text{ for some $z'\prec U_1^n$ with $\abs{U_1^n}-\abs{z'} \geq \epsilon_1 n^{1/2}$}\},
 \end{align*}
 that is the graph formed by the bubbles linking $0$ to the $\epsilon_1 n^{1/2}$-th ancestor of $U_1^n$. 
 
 Take a cut-bond $e$ that is not in $ \mathcal{B}_{\text{ubble}}(\epsilon_1)$. It doesn't lie on any simple path from $0$ to $y_{\epsilon_1}$, so when a current flows from $0$ to $\phi_{\T^{\text{GW}}_n}(y_{\epsilon_1})$  we know that the current travelling through $e$ is $0$ (this follows from Proposition 2.2.~in~\cite{lyons2005probability}). In particular this implies that
\[
\overline{R}^{\omega_n}(0, \phi_{\T^{\text{GW}}_n}(y_{\epsilon_1})\mid \abs{U_1^n})=
\overline{R}_{\mathcal{B}_{\text{ubble}}(\epsilon_1)}(0, \phi_{\T^{\text{GW}}_n}(y_{\epsilon_1})\mid \abs{U_1^n}),
\]
where we used a natural extension of the notation $\overline{R}^{\omega_n}$, and also
\[
R_{\overline{R}^{\omega_n}\omega}(0, \phi_{\T^{\text{GW}}_n}(y_{\epsilon_1}))=
R_{\mathcal{B}_{\text{ubble}}(\epsilon_1)}(0, \phi_{\T^{\text{GW}}_n}(y_{\epsilon_1})),
\]

Introduce 
\[
A_{\epsilon_1}(n):=\{\text{for any } z\in \T^{\text{GW}}_n \text{ with } \abs{z}\geq\abs{U_1^n}, \text{ we have } \phi_{\T^{\text{GW}}_n}(z) \notin  \mathcal{B}_{\text{ubble}}(\epsilon_1)\},
\]
 which means that $ \mathcal{B}_{\text{ubble}}(\epsilon_1)$  are only made of images of points in $\T^{\text{GW}}_n$ whose generation is lower than that of $U_1^n$. The key observation is that on $A_{\epsilon_1}(n)$ we have
 \[
 R_{\mathcal{B}_{\text{ubble}}(y_{\epsilon_1})}(0, \phi_{\T^{\text{GW}}_n}(y_{\epsilon_1}))=\overline{R}_{\mathcal{B}_{\text{ubble}}(y_{\epsilon_1})}(0, \phi_{\T^{\text{GW}}_n}(y_{\epsilon_1})\mid \abs{U_1^n}),
 \]
 which implies by the two previous equations that, on $A_{\epsilon_1}(n)$
 \begin{equation}\label{olamexicano2}
 \overline{R}^{\omega_n}(0, \phi_{\T^{\text{GW}}_n}(y_{\epsilon_1})\mid \abs{U_1^n})=R^{\omega_n}(0, \phi_{\T^{\text{GW}}_n}(y_{\epsilon_1})).
 \end{equation}

Finally, we can notice that $y_{\epsilon_1}$ is a cut-point for the restricted subgraph generated by embedding the $\abs{U_1^n}$ first levels of $\T^{\text{GW}}_n$. This and Rayleigh's monotonicity principle imply that
\begin{equation}\label{olamexicano3}
 \overline{R}^{\omega_n}(0, \phi_{\T^{\text{GW}}_n}(y_{\epsilon_1})\mid \abs{U_1^n})\leq  \overline{R}^{\omega_n}(0, \phi_{\T^{\text{GW}}_n}(U_1^n))\leq  \overline{R}^{\omega_n}(0, \phi_{\T^{\text{GW}}_n}(y_{\epsilon_1})\mid \abs{U_1^n}) +2\epsilon_1 n^{1/2}.
 \end{equation}

The combination of~\eqref{olamexicano}, \eqref{olamexicano2} and~\eqref{olamexicano3} imply that on $A_{\epsilon_1}(n)$, we see that
\[
\abs{\overline{R}^{\omega_n}(0, \phi_{\T^{\text{GW}}_n}(U_1^n))-R^{\omega_n}(0, \phi_{\T^{\text{GW}}_n}(U_1^n))} \leq 2\epsilon_1 n^{1/2}.
\]

Using this equation it means that, for $\epsilon_1>0$ and for some $\epsilon'>0$ 
\begin{align*}
& {\bf P}_n\Bigl[\frac{\abs{\overline{R}^{\omega_n}(0, \phi_{\T^{\text{GW}}_n}(U_1^n))-R^{\omega_n}(0, \phi_{\T^{\text{GW}}_n}(U_1^n))}}{ d^{\T}(0,U_1^n)}  >\epsilon\Bigr] \\
\leq & {\bf P}_n[A_{\epsilon_1}(n)] +{\bf P}_n\Bigl[\max_{x\prec U_1^n} (\abs{x}-\abs{\overline{\pi}_{\T^{\text{GW}}_n}(x)}) \geq n^{(1-\epsilon')/2}\Bigr]  \\
  & \qquad \qquad +{\bf P}_n[d^{\T^{\text{GW}}_n}(0,U_1^n) \leq \Bigl(\frac{\epsilon}4\Bigr)^{-1} \epsilon_1 n^{1/2} ],
\end{align*} 
and taking $n$ to infinity and then $\epsilon_1$ to $0$, we may use~Lemma~\ref{close_cut_point} and Lemma~\ref{stupid_estimate_n_plus_two} to see that for any $\epsilon>0$.
\begin{equation}\label{okkk}
\limsup_{n\to \infty}{\bf P}_n\Bigl[\frac{\abs{R^{\omega_n}(0, \phi_{\T^{\text{GW}}_n}(U_1^n))}}{ d^{\T}(0,U_1^n)}  >\epsilon\Bigr] \leq \liminf_{\epsilon_1\to 0} \limsup_{n\to\infty}{\bf P}_n[A_{\epsilon_1}(n)].
\end{equation}

Hence it would be enough to show that for any $\epsilon>0$ we have that 
\[
 \limsup_{n\to\infty}{\bf P}_n[A_{\epsilon}(n)]=0,
 \]
 i.e.~showing that the ancestors of $U_1^n$ which are at a generation lower than $\abs{U_1^n}-\epsilon n^{1/2}$ are not in bubbles that contain images of points at a generation greater than $\abs{U_1^n}$, except with vanishing probability.
 
Fix $\epsilon>0$, write $y_1^{(n,\epsilon)},\dots, y_{l(n,\epsilon)}^{(n,\epsilon)}$ for the vertices at generation $\abs{U_1^n}-\frac{\epsilon}2 n^{1/2}$ that have descendants at generation $\abs{U_1^n}$. For any $\eta>0$, we can choose, by Lemma~\ref{FUCKKKKKKKKKKK}, $M<\infty$ and $\delta>0$ such that the event
\[
B(n)=\{l(n,\epsilon) \leq M\} \cap \Bigl\{\min_{i=1,\ldots, l(n,\epsilon)} \abs{\overrightarrow{\T^{\text{GW}}_n}_{y_i^{(n,\epsilon)}}} \geq \delta n\Bigr\},
\]
verifies
 \[
\limsup_{n\to \infty} {\bf P}_n[B(n)^c ] <\eta.
\]

On $B(n)$, by successively choosing uniform points $(U_j^n)_{j\geq 0}$ in the tree, we know that at each step we have probability at least $\delta$ to pick a uniform point which has $y_i^{(n,\epsilon)}$ as an ancestor. Since there are at most $M$ such $y_i^{(n,\epsilon)}$, we know there exists $J<\infty$ such that
\[
\limsup_{n\to \infty} {\bf P}_n[B(n), \text{for any $i \leq l(n,\epsilon)$, } y_i^{(n,\epsilon)}\prec U_j^n \text{ for some $j\leq J$}]>1-\eta,
\]
and denote $C(n)^c$ the event $\{\text{for any $i \leq l(n,\epsilon)$, } y_i^{(n,\epsilon)}\prec U_j^n \text{ for some $j\leq J$}\}$.

Now, by applying Lemma~\ref{close_cut_point} to the points $U_j^n$ for $j\leq J$, we know that for some $\epsilon'>0$ we have
\[
\lim_{n\to \infty} {\bf P}_n\Bigl[ \max_{x\prec U_j^n} (\abs{x}-\abs{\overline{\pi}_{\T^{\text{GW}}_n}(x)}) \geq n^{(1-\epsilon')/2} \text{ for some $j\leq J$}\Bigr] =0.
\]

But we can notice that if $B(n)$ and $C(n)$ occur along with $\max_{x\prec U_j^n} (\abs{x}-\abs{\overline{\pi}_{\T^{\text{GW}}_n}(x)}) \leq n^{(1-\epsilon')/2} \text{ for all $j\leq J$}$, then we can see that 
\begin{itemize}
\item the only paths leading to generation $U_1^n$ go through the points $ y_i^{(n,\epsilon)}$ for $ i \leq l(n,\epsilon)$,
\item any $y_i^{(n,\epsilon)}$ is an ancestor of some $U_j^n$ for $j\leq J$ and as such $y_i^{(n,\epsilon)}$  has an ancestor which is a cut-point at distance less than $n^{(1-\epsilon')/2}$ (and hence this cut-point is at a generation at least $\abs{U_1^n}-\frac{\epsilon}2 n^{1/2} -n^{(1-\epsilon')/2}$)
\end{itemize}

  This means that any path from level $\abs{U_1^n}-\epsilon n^{1/2}$ that reaches level $n$ has to cross one of the cut-points on the path to some $y_i^{(n,\epsilon)}$ with $ i \leq l(n,\epsilon)$. This is not compatible with $A_{\epsilon}(n)$. This means that
  \begin{align*}
  &  \limsup_{n\to\infty}{\bf P}_n[A_{\epsilon}(n)] \\
  \leq & \limsup_{n\to \infty} {\bf P}_n[B(n)^c ]+ \limsup_{n\to \infty} {\bf P}_n[B(n),C(n)^c] \\ 
  & \qquad + \lim_{n\to \infty} {\bf P}_n\Bigl[ \max_{x\prec U_j^n} (\abs{x}-\abs{\overline{\pi}_{\T^{\text{GW}}_n}(x)}) \geq n^{(1-\epsilon')/2} \text{ for some $j\leq J$}\Bigr] \\
  \leq & 2\eta
  \end{align*}
  for any $\eta>0$. Since $\eta$ is arbitrary this means that
  \[
  \limsup_{n\to\infty}{\bf P}_n[A_{\epsilon}(n)]=0,
  \]
  which is what we needed to complete the proof.
\end{proof}

Recall the definition of $\overline{\pi}_{\T^{\text{GW}}_n}$ below Remark~\ref{rem_strongcut}.
\begin{proposition}\label{prop_res00}
Fix $d>14$. For $\epsilon>0$, we have 
\[
\lim_{n\to \infty} {\bf P}_n\Bigl[\abs{\frac{R^{\omega_n}(0, \phi_{\T^{\text{GW}}_n}(\overline{\pi}_{\T^{\text{GW}}_n}(U_1^n)))}{ d^{\T^{\text{GW}}_n}(0,U_1^n)} -\rho_1} >\epsilon\Bigr]=0,
\]
and
\[
\lim_{n\to \infty} {\bf P}_n\Bigl[\abs{\frac{d^{\omega_n}(0, \phi_{\T^{\text{GW}}_n}(\overline{\pi}_{\T^{\text{GW}}_n}(U_1^n)))}{ d^{\T^{\text{GW}}_n}(0,U_1^n)} -\rho_2} >\epsilon\Bigr]=0.
\]
\end{proposition}

\begin{proof}
Once again we limit ourselves to the proof of the first result.

By Lemma~\ref{go_france_boo_germany} it is enough to prove that
\begin{equation}\label{simard}
\lim_{n\to \infty} {\bf P}_n\Bigl[\abs{\frac{\overline{R}^{\omega_n}(0, \phi_{\T^{\text{GW}}_n}(\overline{\pi}_{\T^{\text{GW}}_n}(U_1^n)))}{ d^{\T^{\text{GW}}_n}(0,U_1^n)} -\rho_1} >\epsilon\Bigr]=0,
\end{equation}

Fix $\epsilon'>0$, by Lemma~\ref{stupid_estimate_n_plus_one} there exists $\eta>0$ such that the event $C(\eta)$ that was defined above Lemma~\ref{stupid_estimate_n_plus_one} to have
\begin{align}\label{etape1}
& \limsup_{n\to \infty} {\bf P}_n\Bigl[\abs{\frac{\overline{R}^{\omega_n}(0, \phi_{\T^{\text{GW}}_n}(U_1^n))}{ d^{\T^{\text{GW}}_n}(0,U_1^n)} -\rho_1} >\epsilon\Bigr] \\ \nonumber
 \leq & \limsup_{n\to \infty} {\bf P}_n\Bigl[\abs{\frac{\overline{R}^{\omega_n}(0, \phi_{\T^{\text{GW}}_n}(U_1^n))}{ d^{\T^{\text{GW}}_n}(0,U_1^n)} -\rho_1} >\epsilon,C(\eta)\Bigr] +\epsilon'.
\end{align}

Now, we can notice that, by definition of $C(\eta)$, 
\begin{align}\label{etape2}
 & {\bf P}_n\Bigl[\abs{\frac{\overline{R}^{\omega_n}(0, \phi_{\T^{\text{GW}}_n}(U_1^n))}{ d^{\T^{\text{GW}}_n}(0,U_1^n)} -\rho_1} >\epsilon,C(\eta)\Bigr] \\ \nonumber
 \leq & {\bf E}_n\Bigl[\sum_{i=0}^{(1-\eta)n-1} {\bf 1}\Bigl\{U_1^n=v_i^{\T^{\text{GW}}_n}, \abs{\frac{\overline{R}^{\omega_n}(0, \phi_{\T^{\text{GW}}_n}(U_1^n))}{ d^{\T^{\text{GW}}_n}(0,U_1^n)} -\rho_1} >\epsilon,C(\eta)\Bigr\}\Bigr]\\ \nonumber
 \leq  &{\bf P}_n\Bigl[\abs{\frac{\overline{R}^{\omega_n}(0, \phi_{\T^{\text{GW}}_n}(U_1^{(1-\eta)n}))}{ d^{\T^{\text{GW}}_n}(0,U_1^{(1-\eta)n})} -\rho_1} >\epsilon\Bigr],
 \end{align}
 where $U_1^{(1-\eta)n}$ is a uniform variable on the $(1-\eta)n$ first vertices (in the breadth first order). We can notice that the event on the right-hand side is measurable with respect to the variables $(Q_i)_{i\leq (1-\eta) n}$ and the embedding of those points.

We are now going to use a strategy employed in~\cite{LG1} (p.741-743) and used again in~\cite{LGL} (proof of Theorem 7) to compare the probability of events measurable with respect $(Q_i)_{i\leq (1-\eta) n}$ under the measures ${\bf P}_n$ and ${\bf P}[\cdot \mid \abs{\T^{\text{GW}}_n}\geq n]$. In this context, the technique was applied for the lexicographical ordering of the tree instead of the breadth first ordering but this does not change the details of the proof. In this context, we are going to use the proof from~\cite{LGL} which is written more rapidly but provides a good description of this method. 

Consider a bounded function on $\Z^{(1-\eta) n+1}$ and for any $j\in  \Z$ write $P_j[\cdot]$ for the law of a random walk $(Q_i)_{i\in \N}$ on $\Z$ with jump distribution $Z-1$ started at $j$ and set 
\[
T:=\inf\{k\geq 0, Q_k=0\},
\]

Using Markov's property at time $(1-\eta)n$  for a $P$-random walk (see~\cite{LG1}, p.742-743 for details), one can verify that 
\begin{equation}\label{damn1}
{\bf E}_n[f((Q_i)_{i\leq (1-\eta) n})] = {\bf E}\Bigl[ f((Q_i)_{i\leq (1-\eta) n}) \frac{P_0[T\geq n]}{P_0[T=n]}\frac{\Xi_n(X_{(1-\eta)n})}{\Xi^{'}_n(X_{(1-\eta)n})}\mid \abs{\T^{\text{GW}}}\geq n\Bigr],
\end{equation}
where, for every integer $j\geq 0$, 
\[
\Xi_n(j)=P_j[T=n-(1-\eta)n] \text{ and } \Xi_n^{'}(j)=P_j[T\geq n-(1-\eta)n].
\]

Using Kemperman's formula (see~\cite{P}, p.122) and the local limit theorem, it is possible to obtain (see~\cite{LG1}, p.742-743 for details) that for any $\delta>0$,
\begin{equation}\label{damn2}
\lim_{n\to \infty} \Bigl(\sup_{j\geq c\sqrt n} \abs{ \frac{P_0[T\geq n]}{P_0[T=n]}\frac{\Xi_n(j)}{\Xi^{'}_n(j)}- \Gamma_{1-\eta}\Bigl( \frac j{\sigma_Z\sqrt n}\Bigr)}\Bigr)=0,
\end{equation}
where we define for $x\geq 0$,
\[
\Gamma_{1-\eta}(x)= \frac{ 2(2\pi \eta^3)^{-1/2} \exp(-x^2/(2\eta))}{\int_{\eta}^{\infty} (2\pi s^3)^{-1/2} \exp(-x^2/(2s)) ds}.
\]

This last function is bounded over $\R^+$ and one can verify with the local limit theorem that 
\begin{equation}\label{damn3}
\lim_{\delta \to 0} \limsup_{n\to \infty}{\bf P}_n[Q_{(1-\eta)n} \leq \delta \sqrt{n}\mid \abs{\T^{\text{GW}}}=n] =0,
\end{equation}
and
\begin{equation}\label{damn35}
 \lim_{\delta \to 0} \limsup_{n\to \infty}{\bf P}[Q_{(1-\eta)n} \leq \delta \sqrt{n}\mid \abs{\T^{\text{GW}}}\geq n] =0.
\end{equation}

Now using~\eqref{damn1}, \eqref{damn2}, \eqref{damn3} and~\eqref{damn35} we can see that for any uniformly bounded sequence of functions $(f_n)_{n\geq 1}$ on $\Z^{(1-\eta)n}$ we have
\[
\lim_{n\to \infty} \abs{{\bf E}_n[f((Q_i)_{i\leq (1-\eta) n})] - {\bf E}\Bigl[ f((Q_i)_{i\leq (1-\eta) n}\Gamma_{1-\eta}\Bigl( \frac {Q_{(1-\eta)n}}{\sigma_Z\sqrt n}\Bigr)\mid \abs{\T^{\text{GW}}}\geq n \Bigr]}=0,
\]
hence using the fact that $\Gamma_{1-\eta}\leq K(\eta)<\infty$, we see that for any sequence of events $(A_n)_{n\geq 1}$ which are measurable with respect to $(Q_i)_{i\leq (1-\eta) n}$ and the embedding of those points we have
\[
\limsup_{n\to \infty} {\bf P}_n[A_n] \leq K(\eta) \limsup_{n\to \infty}{\bf P}[A_n\mid \abs{\T^{\text{GW}}}\geq n ],
\]
which is a formula that can be applied for the event on the right-hand side of~\eqref{etape2}
\begin{align}\label{etape3}
&\limsup_{n\to \infty} {\bf P}_n\Bigl[\abs{\frac{\overline{R}^{\omega_n}(0, \phi_{\T^{\text{GW}}_n}(U_1^{(1-\eta)n}))}{ d^{\T^{\text{GW}}_n}(0,U_1^{(1-\eta)n})} -\rho_1} >\epsilon\Bigr] \\ \nonumber
\leq & K(\eta) \limsup_{n\to \infty}{\bf P}\Bigl[\abs{\frac{\overline{R}^{\omega_n}(0, \phi_{\T^{\text{GW}}_n}(U_1^{(1-\eta)n}))}{ d^{\T^{\text{GW}}_n}(0,U_1^{(1-\eta)n})} -\rho_1} >\epsilon\mid \abs{\T^{\text{GW}}}\geq n \Bigr],
\end{align}
 where $U_1^{(1-\eta)n}$ is a uniform variable on the $(1-\eta)n$ first vertices of $\T^{\text{GW}}_n$ (in the breadth first order).
 
 Fix $\epsilon''>0$,  by using that ${\bf P}[\abs{\T^{\text{GW}}}\geq n] \sim C n^{-1/2}$ (Lemma 2.1.4~\cite{Kolchin}),   there exists $K<\infty$ such that ${\bf P}[\abs{\T^{\text{GW}}}\geq Kn]/{\bf P}[\abs{\T^{\text{GW}}}\geq n]\leq \epsilon''$ for $n$ large. Hence
  \begin{align*}
 & {\bf P}\Bigl[\abs{\frac{\overline{R}^{\omega_n}(0, \phi_{\T^{\text{GW}}_n}(U_1^{(1-\eta)n}))}{ d^{\T^{\text{GW}}_n}(0,U_1^{(1-\eta)n})} -\rho_1} >\epsilon\mid \abs{\T^{\text{GW}}}\geq n \Bigr]\\
 \leq & {\bf P}\Bigl[\abs{\frac{\overline{R}^{\omega_n}(0, \phi_{\T^{\text{GW}}_n}(U_1^{(1-\eta)n}))}{ d^{\T^{\text{GW}}_n}(0,U_1^{(1-\eta)n})} -\rho_1} >\epsilon, n\leq \abs{\T^{\text{GW}}} \leq  n K \mid \abs{\T^{\text{GW}}} \geq n\Bigr] \\
 & \qquad \qquad \qquad \qquad \qquad \qquad +{\bf P}[\abs{\T^{\text{GW}}_n} \geq Kn\mid \abs{\T^{\text{GW}}}\geq n] \\
 \leq & {\bf E}\Bigl[\sum_{i=0}^{(1-\eta)n} \frac 1{(1-\eta)n} {\bf 1}\bigl\{\abs{\frac{\overline{R}^{\omega_n}(0, \phi_{\T^{\text{GW}}_n}(v_i^{\T^{\text{GW}}_n}))}{ d^{\T^{\text{GW}}_n}(0,Q_i)} -\rho_1} >\epsilon, n\leq \abs{\T^{\text{GW}}} \leq  n K\Bigr\} \mid \abs{\T^{\text{GW}}} \geq n\Bigr] +\epsilon'' \\
\leq  & {\bf E}\Bigl[\frac {\abs{T}}{(1-\eta)n}\sum_{i=0}^{\abs{T}} \frac 1{\abs{T}} {\bf 1}\Bigl\{\abs{\frac{\overline{R}^{\omega_n}(0, \phi_{\T^{\text{GW}}_n}(v_i^{\T^{\text{GW}}_n}))}{ d^{\T^{\text{GW}}_n}(0,Q_i)} -\rho_1} >\epsilon, n\leq \abs{\T^{\text{GW}}} \leq  n K \Bigr\}\mid \abs{\T^{\text{GW}}} \geq n\Bigr] +\epsilon'' \\
\leq  & \frac{K}{(1-\eta)}{\bf E}\Bigl[\sum_{i=0}^{\abs{T}} \frac 1{\abs{T}} {\bf 1}\Bigl\{\abs{\frac{\overline{R}^{\omega_n}(0, \phi_{\T^{\text{GW}}_n}(v_i^{\T^{\text{GW}}_n}))}{ d^{\T^{\text{GW}}_n}(0,Q_i)} -\rho_1} >\epsilon, n\leq \abs{\T^{\text{GW}}} \leq  n K\Bigr\} \mid \abs{\T^{\text{GW}}} \geq n\Bigr] +\epsilon'' \\
\leq  & \frac{K}{(1-\eta)}{\bf P}\Bigl[\abs{\frac{\overline{R}^{\omega_n}(0, \phi_{\T^{\text{GW}}_n}(U_1^n))}{ d^{\T^{\text{GW}}_n}(0,U_1^n)} -\rho_1} >\epsilon, n\leq \abs{\T^{\text{GW}}} \leq  n K \mid \abs{\T^{\text{GW}}} \geq n\Bigr] +\epsilon''.
\end{align*}

Using Lemma~\ref{go_france_boo_germany} means that  Lemma~\ref{prop_res0000} holds with $\overline{R}(\cdot,\cdot)$ instead of $R(\cdot,\cdot)$. We can use that result to upper-bound the previous equation by letting $n$ go to infinity and taking $\epsilon''$ to 0, which yields that
\begin{equation}\label{etape4}
\limsup_{n\to \infty}{\bf P}\Bigl[\abs{\frac{\overline{R}^{\omega_n}(0, \phi_{\T^{\text{GW}}_n}(U_1^{(1-\eta)n}))}{ d^{\T^{\text{GW}}_n}(0,U_1^{(1-\eta)n})} -\rho_1} >\epsilon\mid \abs{\T^{\text{GW}}}\geq n \Bigr]=0,
\end{equation}
 where $U_1^{(1-\eta)n}$ is a uniform variable on the $(1-\eta)n$ first vertices of $\T^{\text{GW}}_n$ (in the breadth first order).

Finally, using~\eqref{etape1}, \eqref{etape2}, \eqref{etape3} and~\eqref{etape4} and letting $\epsilon'$ go to 0 we see that
\begin{equation}\label{almost_f-ing_finished}
 \limsup_{n\to \infty} {\bf P}_n\Bigl[\abs{\frac{\overline{R}^{\omega_n}(0, \phi_{\T^{\text{GW}}_n}(U_1^n))}{ d^{\T^{\text{GW}}_n}(0,U_1^n)} -\rho_1} >\epsilon\Bigr]=0,
 \end{equation}
 where $U_1^n$ is uniformly distributed on $\T^{\text{GW}}_n$. This is~\eqref{simard} which is what we needed to prove the lemma.
\end{proof}




\subsection{Condition $(R)$}\label{sect_condition_R_proof}

Recall the definition of $\pi_n$ at Definition~\ref{defpi}. We can prove condition $(R)$ with $V_i^n:=\pi_n(\phi_{\T^{\text{GW}}_n}(U_i^n))$ where $(U_i^n)_{i\in \N}$ is an i.i.d.~sequence of random variables on $\T^{\text{GW}}_n$.
\begin{proposition}\label{prop_res0}
Fix $d>14$. For $\epsilon>0$, we have 
\[
\lim_{n\to 0} {\bf P}_n\Bigl[\abs{\frac{R^{\omega_n}(0, \pi_n(\phi_{\T^{\text{GW}}_n}(U_1^n)))}{ d^{\mathcal{\omega}_n}(0,\pi_n(\phi_{\T^{\text{GW}}_n}(U_1^n)))} -\rho} >\epsilon\bigr]=0,
\]
where $U_1^n$ is uniformly chosen among the vertices of $\T^{\text{GW}}_n$ where
\begin{equation}\label{def_rho}
\rho:=\frac{{\bf E}^{-\infty,\infty}[R_{\text{eff}}^{-\infty,\infty}(0, \alpha(\gamma_1))\mid 0\text{ is a pivotal point}] }{{\bf E}^{-\infty,\infty}[d_{\omega_{{-\infty,\infty}}}(0, \alpha(\gamma_1))\mid 0\text{ is a pivotal point}] }.
\end{equation}
\end{proposition}

\begin{proof}

{\it Step 1}

\vspace{0.5cm}

Let us start with the case where there is no cut-point between $U_1^n$ and $0$. In this case, by definition of $\pi_n$ and $\overline{\pi}^n$, we know that 
\[
d_{\omega_n}(0,U_1^n)= \max_{x\prec U_1^n} \abs{x}-\abs{\overline{\pi}^n(x)},
\]
because in this case $\overline{\pi}^n(U_1)=\text{root}$

Hence, we know that for all $\epsilon>0$ and $\eta>0$,we have
\begin{align}\label{damn_step1}
& \{\text{there is no cut-point between $U_1^n$ and $0$}\} \\ \nonumber
\subset & \{ \max_{x\prec U_1^n} \abs{x}-\abs{\overline{\pi}^n(x)} \geq n^{(1-\epsilon')/2}\} \cup \{d^{\T^{\text{GW}}_n}(0, U_1^n) \leq n^{(1-\epsilon'/2)/2} \}.
\end{align}

\vspace{0.5cm}

{\it Step 2}

\vspace{0.5cm}

If there is a cut-point between $U_1^n$ and $0$, the we know that $\pi_n(\phi_{\T^{\text{GW}}_n}(U_1^n))$ is a cut-point lying on any path from $0$ to $\phi_{\T^{\text{GW}}_n}(U_1^n)$. Furthermore, by Remark~\ref{rem_strongcut}, we know that $\phi_{\T^{\text{GW}}_n}(\overline{\pi}_{\T^{\text{GW}}_n}(U_1^n)))$ is a cut-point on any path from $0$ to $\pi_n(\phi_{\T^{\text{GW}}_n}(U_1^n))$. By the law of resistances in series this means that
\[
R^{\omega_n}(0, \phi_{\T^{\text{GW}}_n}(\overline{\pi}_{\T^{\text{GW}}_n}(U_1^n)))\leq R^{\omega_n}(0, \pi_n(\phi_{\T^{\text{GW}}_n}(U_1^n)))  \leq R^{\omega_n}(0, \phi_{\T^{\text{GW}}_n}(U_1^n)),
\]
which means that
\begin{align*}
& \abs{R^{\omega_n}(0, \phi_{\T^{\text{GW}}_n}(\overline{\pi}_{\T^{\text{GW}}_n}(U_1^n))) -R^{\omega_n}(0, \pi_n(\phi_{\T^{\text{GW}}_n}(U_1^n)))} \\
\leq& R^{\omega_n}(0, \phi_{\T^{\text{GW}}_n}(U_1^n))- R^{\omega_n}(0, \phi_{\T^{\text{GW}}_n}(\overline{\pi}_{\T^{\text{GW}}_n}(U_1^n))),
\end{align*}
but if $ \max_{x\prec U_1^n} (\abs{x}-\abs{\overline{\pi}_{\T^{\text{GW}}_n}(x)}) \leq n^{(1-\epsilon')/2}$ for some $\epsilon'>0$,  then using that
\[
d_{\omega_n}(\phi_{\T^{\text{GW}}_n}(U_1^n), \phi_{\T^{\text{GW}}_n}(\overline{\pi}_{\T^{\text{GW}}_n}(U_1^n))) \leq d^{\T^{\text{GW}}_n}(U_1^n,\overline{\pi}_{\T^{\text{GW}}_n}(U_1^n)),
\]
and using Rayleigh's monotonicity principle, we see
\begin{equation}\label{grrr}
\abs{R^{\omega_n}(0, \phi_{\T^{\text{GW}}_n}(\overline{\pi}_{\T^{\text{GW}}_n}(U_1^n))) -R^{\omega_n}(0, \pi_n(\phi_{\T^{\text{GW}}_n}(U_1^n)))} \leq n^{(1-\epsilon')/2},
\end{equation}
furthermore, considering the event $\{d^{\T^{\text{GW}}_n}(0,U_1^n) \leq n^{(1-\epsilon'/2)/2} \}$, we can see that
\begin{align}\label{damn_step2}
& \Bigl \{\abs{\frac{R^{\omega_n}(0, \pi_n(\phi_{\T^{\text{GW}}_n}(U_1^n)))}{ d^{\mathcal{\omega}_n}(0,\pi_n(\phi_{\T^{\text{GW}}_n}(U_1^n)))} -\frac{R^{\omega_n}(0, \phi_{\T^{\text{GW}}_n}(\overline{\pi}_{\T^{\text{GW}}_n}(U_1^n)))}{ d^{\mathcal{\omega}_n}(0,\phi_{\T^{\text{GW}}_n}(\overline{\pi}_{\T^{\text{GW}}_n}(U_1^n))))} } >n^{-\epsilon'/4} \Bigr\} \\ \nonumber
& \subset \Bigl\{  \max_{x\prec U_1^n} (\abs{x}-\abs{\overline{\pi}_{\T^{\text{GW}}_n}(x)}) \leq n^{(1-\epsilon')/2}\Bigr\} \cup \Bigl\{d^{\T^{\text{GW}}_n}(0,U_1^n) \leq n^{(1-\epsilon'/2)/2} \Bigr\}.
\end{align}

\vspace{0.5cm}

{\it Step 3}

\vspace{0.5cm}

Recalling~\eqref{def_rho1} and~\eqref{def_rho2}, Proposition~\ref{prop_res00} implies
\begin{equation}\label{nminus1}
\lim_{n\to 0} {\bf P}_n\Bigl[\abs{\frac{R^{\omega_n}(0, \phi_{\T^{\text{GW}}_n}(\overline{\pi}_{\T^{\text{GW}}_n}(U_1^n)))}{ d^{\mathcal{\omega}_n}(0,\phi_{\T^{\text{GW}}_n}(\overline{\pi}_{\T^{\text{GW}}_n}(U_1^n))))} -\rho} >\epsilon\bigr]=0.
\end{equation}

Recalling~\eqref{damn_step1}, \eqref{damn_step2} ant the first step we see that
\begin{align*}
& {\bf P}_n\Bigl[\abs{\frac{R^{\omega_n}(0, \pi_n(\phi_{\T^{\text{GW}}_n}(U_1^n)))}{ d^{\mathcal{\omega}_n}(0,\pi_n(\phi_{\T^{\text{GW}}_n}(U_1^n)))} -\rho} >\epsilon\bigr] \\
&\leq  {\bf P}_n\Bigl[\abs{\frac{R^{\omega_n}(0, \phi_{\T^{\text{GW}}_n}(\overline{\pi}_{\T^{\text{GW}}_n}(U_1^n)))}{ d^{\mathcal{\omega}_n}(0,\phi_{\T^{\text{GW}}_n}(\overline{\pi}_{\T^{\text{GW}}_n}(U_1^n))))} -\rho} >\epsilon-n^{-\epsilon'/4}\Bigr]  \\
& +{\bf P}_n\Bigl[\max_{x\prec U_1^n} (\abs{x}-\abs{\overline{\pi}_{\T^{\text{GW}}_n}(x)}) \geq n^{(1-\epsilon')/2}\Bigr] +{\bf P}_n[d^{\T^{\text{GW}}_n}(0,U_1^n) \leq n^{(1-\epsilon'/2)/2} ]
\end{align*}
 the result follows from Lemma~\ref{stupid_estimate_n_plus_two}, Lemma~\ref{close_cut_point} and~\eqref{nminus1}.
\end{proof}

We will need the following lemma which can be proved in the exact same manner as above
\begin{lemma}\label{prop_dist0}
Fix $d>14$. For $\epsilon>0$, we have 
\[
\lim_{n\to 0} {\bf P}_n\Bigl[\abs{\frac{d^{\omega_n}(\text{root}, \pi_n(\phi_{\T^{\text{GW}}_n}(U_1^n)))}{ d^{\T^{\text{GW}}_n}(0,U_1^n)} -\sigma_G} >\epsilon\bigr]=0,
\]
where $U_1^n$ is uniformly chosen among the vertices of $\T^{\text{GW}}_n$ where
\begin{equation}\label{def_sigmaG}
\sigma_G:=\frac{{\bf E}^{-\infty,\infty}[d_{\omega_{-\infty,\infty}}(0, \alpha(\gamma_1))\mid 0\text{ is a pivotal point}] }{{\bf E}^{-\infty,\infty}[\gamma_1\mid 0\text{ is a pivotal point}] }.
\end{equation}
\end{lemma}

\section{Geometric estimate and Asymptotic thinness}\label{sect:condG}
Recall that $\omega_n$ is the random graph whose set of vertices $V(\omega_n)$ is $\phi_{n}(V(\T^{\text{GW}}_n))$ and its set of edges $E(\omega_n)$ is $\{(x_1,x_2)\in E(\Z^d): \exists  e=(v_1,v_2)\in E(\T^{\text{GW}}_n) : \phi_{\T^{\text{GW}}_n}(v_1)=x_1\text{ and }\phi_{\T^{\text{GW}}_n}(x_2)=x_2\}$. 
Recall from Section \ref{sect_abstract_thm} that $V^n_i:=\pi_n(\phi_{\T^{\text{GW}}_n}(U^n_i))$, where, for each $n\in\N$, $(U^n_i)_{i\in\N}$ is a sequence uniform, i.i.d.~vertices of $\T^{\text{GW}}_n$.        
 In this section we will prove condition $(G)$ and asymptotic thinness for the sequence of random augmented graphs $(\omega_n,(V_i^n)_{i\in\N})_{n\in\N}$.
 
Recall that $\sigma=2/\sigma_Z$, where $\sigma_Z$ is the variance of the offspring distribution of the Galton-Watson tree and that $\sigma_G$ was defined at~\eqref{def_sigmaG}.
The main result of this section is
 \begin{lemma}\label{lem:condG}
 The sequence of random augmented graphs $(\omega_n,(V_i^n)_{i\in\N})_{n\in\N}$ satisfies condition $(G)_{\sigma_d,\sigma_\phi}$ with $\sigma_d=\sigma\sigma_G$ and $\sigma_{\phi}=\sigma_G^{-1/2}$.
 \end{lemma}
 
  In \cite{janson2005convergence} Janson and Marckert proved that  $(\T^{\text{GW}}_n, d^{\T^{\text{GW}}_n},\phi_{\T^{\text{GW}}_n})$ scales to the ISE when regarded as a spatial tree. Roughly speaking, the main task of this section is to transfer that result from $(\T^{\text{GW}}_n,d^{\T^{\text{GW}}_n})$ to $(\frak {T}^{(n,K)},d^{(n,K)})$.
  The first and more relevant issue to address is to transfer the convergence from the underlying tree distance $d^{\T^{\text{GW}}_n}$ to the embedded graph distance $d^{(n,K)}$. A secondary issue is the fact that the vertices and edges of $\T^{\text{GW}}_n$ are not the same as those of $\T^{(n,K)}$.

\subsection{A useful coupling}\label{s:usefcoupl}

 Since the proof of the lemma above will use some previous results about branching random walks, we will introduce the notation for the statement of those results. 
 Recall that $w_n$ is the depth first search around $\T^{\text{GW}}_n$ and $\hat{w}_n$ is the search depth process of $\T^{\text{GW}}_n$, which were given at the beginning of Section~\ref{subsubsection_searchdepth}.
 
 Let us define the discrete tour process $r_n:[0,1]\to\R^d$ by setting
 \[r_n\left(\frac{i}{2(n-1)}\right)=\phi_{\T^{\text{GW}}_n}\left(w_n\left(\frac{i}{2(n-1)}\right)\right) \qquad {i=0,\dots,2(n-1)}\]
 and interpolating linearly between those values.
 Similarly, let us define the Brownian tour $\bold{r}:[0,1]\to\R^d$ as
 \[\bold{r}(x)=\phi_{\frak{T}}([x]).\]
 Recall that $(\bold{e}_t)_{t\in[0,1]}$ is a normalized Brownian excursion.
 \begin{proposition}[Theorem 2 of \cite{janson2005convergence}]\label{prop:tourconvergence}
 \[(n^{-1/2}\hat{w}_n,n^{-1/4}r_n)\stackrel{n\to\infty}{\to}\left(\sigma\bold{e},\sqrt{\sigma}\bold{r}\right)
\]
weakly in the space $C([0,1],\R_+)\times C([0,1],\R^d)$ endowed with the topology of uniform convergence.
 \end{proposition}
 
During this section we will work under a specific coupling that we describe next:

By Proposition \ref{prop:tourconvergence} and the Skorohod representation Theorem we can assume that
\begin{equation}\label{eq:papapa}
(n^{-1/2}\hat{w}_n,n^{-1/4}r_n)\stackrel{n\to\infty}{\to}\left(\sigma\bold{e},\sqrt{\sigma}\bold{r}\right)
, \quad {\bf P}\text{-a.s.}
\end{equation}
Let $\gamma_n:[0,1]\to[0,1]$ be defined by
\begin{equation}\label{eq:defofgamma}
\gamma_n(x)=
\begin{cases}
\frac{\lfloor 2(n-1)x \rfloor}{2(n-1)} \quad &\text{ if }\hat{w}_n(\frac{\lfloor 2(n-1)x \rfloor}{2(n-1)})\geq \hat{w}_n(\frac{\lceil 2(n-1)x \rceil}{2(n-1)}),\\
\frac{\lceil 2(n-1)x \rceil}{2(n-1)} \quad &\text{ otherwise.}
\end{cases}
\end{equation}
The function $\gamma_n$ is constructed so that $w_n\circ\gamma_n\circ \text{Leb}$ is the uniform measure in $V(\T^{\text{GW}}_n)\setminus \{\text{root}\}$.
Therefore, recalling that $(U_i)_{i\in\N}$ is the i.i.d.~sequence of random variables uniformly distributed over $[0,1]$ which is used to construct $\frak{T}^{(K)}$, we have that $(w_n(\gamma_n(U_i)))_{i\in\N}$ is an independent sequence of uniform (non-root) vertices of $\T^{\text{GW}}_n$. 

Hence, we can (and will) assume that
\begin{equation}\label{eq:secondcoupling}
(U^i_n)_{i\in\N}=(w_n(\gamma_n(U_i)))_{i\in\N}.
\end{equation}
There is a small issue since $(w_n(\gamma_n(U_i)))_{i\in\N}$ is an i.i.d.~sequence of uniform \emph{non-root} vertices, while the original $(U^i_n)_{i\in\N}$ were uniform on the whole set of vertices of $\T^{\text{GW}}_n$. But, since the probability that one of the original $U^i_n,i=1,\dots,K$ is the root goes to $0$ with $n$, we can dismiss this issue.

\subsection{An equivalent formulation of tour convergence}
We will need to use an equivalent form for the convergence at display \eqref{eq:papapa}. 
Recall from Section \ref{sect_graph_tree} the notation for reduced sub-trees. 
Let $K\in\N$ and define
\begin{equation}\label{eq:defoftgwk}
\T^{\text{GW}}_n(K):=\T^{\text{GW}}_n(U^n_i,i=1,\dots,K),
\end{equation}
that is, $\T^{\text{GW}}_n(K)$ is the minimal subtree of $\T^{\text{GW}}_n$ containing the root and $U_i^n,i=1,\dots,K$. Let $\pi_{\T^{\text{GW}}_n(K)}:V(\T^{\text{GW}}_n)\to V(\T^{\text{GW}}_n(K))$ be the projection of $\T^{\text{GW}}_n$ onto $\T^{\text{GW}}_n(K)$. More precisely, for all $v\in V(\T^{\text{GW}}_n)$, $\pi_{\T^{\text{GW}}_n(K)}(v)$ is the vertex of $\T^{\text{GW}}_n(K)$ which is the closest to $v$ under $d^{\T^{\text{GW}}_n}$.
Let $\frak{T}^{\text{GW}}_{n}(K)$ be constructed from $\T^{\text{GW}}_n(K)$ in the same way that $\frak{T}^{(n,K)}$ was constructed from $\T^{(n,K)}$ in Section \ref{sect:skeleton} and regard $(\frak{T}_n^{\text{GW}}(K), d^{\T^{\text{GW}}_n}, n^{-1/4}\phi_{\T^{\text{GW}}_n})$ as a graph spatial tree. 
 \begin{proposition}[Proposition 8.2 of \cite{Croydon_arc}] \label{prop:croydonequivalence}
For $\bold{P}$-a.e.~realization of $((\bold{e}_t)_{t\in[0,1]},(U_i)_{i\in\N})$, the convergence at Proposition \ref{prop:tourconvergence} is equivalent to the conjunction of the two following conditions 
 \begin{enumerate}
 \item For all $K\in\N$
 \begin{equation}\label{eq:equation1}
\lim_{K\to\infty}\limsup_{n\to\infty}n^{-1/2}\max_{v\in V(\T^{\text{GW}}_n)} d^{\T^{\text{GW}}_n}(v,\pi_{\T^{\text{GW}}_n(K)}(v))=0
\end{equation}
and
\begin{equation}\label{eq:equation2}
\lim_{\delta\to0}\limsup_{n\to\infty}n^{-1/4} \max_{d^{\T^{\text{GW}}_n}(x,y)\leq \delta n^{-1/2}} d_{\Z^d}(\phi_{\T^{\text{GW}}_n}(x),\phi_{\T^{\text{GW}}_n}(y))=0.
\end{equation}

 \item For all $K\in\N$
 \begin{equation}\label{eq:underlyinggeometricestimate}
(\frak{T}^{\text{GW}}_n(K), n^{-1/2}d^{\T^{\text{GW}}_n}, n^{-1/4}\phi_{\T^{\text{GW}}_n})\stackrel{n\to\infty}{\to}\left(\frak{T}^{(K)},\sigma d_{\frak{T}},\sqrt{\sigma}\phi_{\frak{T}}\right)
\end{equation}
under the topology given by the metric in graph spatial trees $D$.
 
 \end{enumerate}
 \end{proposition}

\subsection{Proof of Condition (G)} 
Before presenting the proof of Lemma \ref{lem:condG} we will need four auxiliary results. 
\begin{lemma}\label{compacity_tnk}
For each $\epsilon,\eta>0$, there exists $K'$ such that
\[\limsup_{n\to\infty}{\bf P}\left[ \max_{x\in V(\T^{\text{GW}}_n)}\min_{i=1,\dots,K'}n^{-1/2}d^{\T^{\text{GW}}_n}(U^n_i,x)\geq \epsilon \right]\leq \eta.\]
\end{lemma}
 
 \begin{lemma}\label{lem:analogoustores_2a}
Let $K\in\N$ and $\epsilon>0$, then
\begin{equation} 
\lim_{n\to\infty} \bold{P}\left[\max_{v\in V(\T^{\text{GW}}_n(K)) }\abs{d^{(n,K)}(\text{root}^\ast,\phi_{\T^{\text{GW}}_n}(\bar{\pi}_{\T^{\text{GW}}_n}(v)))-\sigma_Gn^{-1/2}d^{\T^{\text{GW}}_n}(\text{root},v)} \geq \epsilon \right]=0.
\end{equation}
 \end{lemma}

 \begin{lemma}\label{lem:vhart}
For $i,j\leq K$, we have that 
 \begin{equation}\label{eq:mosa}
d^{(n,K)}(\text{root}^*,b^{\T^{(n,K)}}(V^i_n, V^j_n))\to \sigma_dd_{\frak{T}}(\text{root},b^{\frak{T}}([U_i],[U_j])).
 \end{equation}

\end{lemma}

\begin{lemma}\label{lem:eleele}
Let $K\in\N$ and $i,j,k\leq K$.
\begin{enumerate}
\item If 
\[d_{\frak{T}}(\text{root},b^{\T^{\text{GW}}_n}([U_i], [U_j]))>d_{\frak{T}}(\text{root},b^{\frak{T}}([U_i], [U_k]))\] then \[d^{(n,K)}(\text{root}^*,b^{\T^{(n,K)}}(V^i_n, V^j_n))>d^{(n,K)}(\text{root}^*,b^{\T^{(n,K)}}(V^i_n, V^k_n))\] for $n$ large enough. 
\item
If \[d_{\frak{T}}(\text{root},b^{\frak T}([U_i], [U_j]))=d_{\frak{T}}(\text{root},b^{\frak{T}}([U_i], [U_k]))\]
 then \[d^{(n,K)}(\text{root}^\ast,b^{\T^{(n,K)}}(V^i_n, V^j_n))=d^{(n,K)}(\text{root}^*,b^{\T^{(n,K)}}(V^i_n, V^k_n))\] for $n$ large enough.  
\end{enumerate}
\end{lemma}
 
 Assuming those four lemmas we prove Lemma \ref{lem:condG}.
 \begin{proof}[Proof of Lemma \ref{lem:condG}]
We have to show that
\begin{equation}\label{eq:...}
(\frak{T}^{(n,K)}, d^{(n,K)}, n^{-1/4}\phi_{\T^{\text{GW}}_n})\stackrel{n\to\infty}{\to} \left(\frak{T}^{(K)},\sigma\sigma_Gd_{\frak{T}},\sqrt{\sigma}\phi_{\frak{T}}\right)
\end{equation}
in distribution in the $D$ distance.

\vspace{0.5cm}

{\it Step 1: Showing that the trees are asymptotically homeomorphic}

\vspace{0.5cm}

We will start by showing that, for $n$ large enough, the tree $\frak{T}^{(n,K)}$ is homeomorphic to $\frak{T}^{(K)}$ and there exists a genealogical order preserving homeomorphism.
The construction depends on Lemma \ref{lem:eleele}

We can construct the homeomorphism $\Upsilon_{n,K}:\frak{T}^{(n,K)}\to \frak{T}^{(K)}$ as follows. Let $i\leq K$ be fixed. We will first define $\Upsilon_{n,K}$ on the interval $[\text{root}^\ast,V^i_n]$. The interval $[\text{root}^\ast,V^i_n]$ splits on a set of edges $e_1^n(i),e_2^n(i),\dots, e^n_{l(n,i)}(i)$  $\in E(\frak{T}^{(n,K)})$, each one having as endpoints on the set $\{\text{root}^\ast, V^i_n, b^{\T^{(n,K)}}(V^i_n,V^j_n), j\leq K\}$.
Also, $[\text{root}, [U_i]]$ splits on a set of edges $e_1(i),e_2(i),\dots, e_{l(i)}(i)$ of $\frak{T}^{(K)}$. Moreover, by Lemma \ref{lem:eleele} it is not hard to see that, for $n$ large enough, the number $l(n,i)$ of edges in which $[\text{root}^\ast, V^i_n]$ decomposes is the same as the number $l(i)$ in which $[\text{root}, [U_i]]$ does. 
Let us assume that $e_1^n(i),e_2^n(i),\dots, e_{l(n,i)}^n(i)$ and $e_1(i),e_2(i),\dots, e_{l(i)}(i)$ are indexed according to increasing genealogical order. So, for $i\leq K$ and $j\leq l(n,i)$ we set 
$\Upsilon_{n,K}$ as mapping $e_j^n(i)$ onto $e_j(i)$ in such a way that the point at distance $\alpha |e^n_j(i)|_{d^{(n,K)}}$ along the edge $e^n_j(i)$ is mapped to the point at distance $\alpha |e_j(i)|_{d_{\frak{T}}}$ along the edge $e_j(i)$, for all $\alpha \in [0,1]$.
By Lemma \ref{lem:eleele} we have that this construction is consistent for different indices $i$, for $n$ large enough, thus providing a homeomorphism $\Upsilon_{n,K}:\frak{T}^{(n,K)}\to\frak{T}^{(K)}$.

\vspace{0.5cm}

{\it Step 2: Showing the convergence of the metric}

\vspace{0.5cm}

To show that $(\T^{(n,K)}_n,d^{(n,K)})$ converges to $(\frak{T}^{(K)},\sigma_dd_{\frak{T}})$ under the $d_1$ distance (where $d_1$ is as in \eqref{eq:defofd1} 
) it only remains to show that 
\begin{equation}
\abs{e_j^n(i)}_{d^{(n,K)}} \to \sigma_d\abs{e_j(i)}_{d_{\frak{T}}}
\end{equation}
for all $i\leq K$ and $j\leq l(i)$, but this can be easily deduced from
\eqref{eq:mosa}.
That is, we have shown that $(\frak{T}^{(n,K)},d^{(n,K)})$ converges under $d_1$ to $(\frak{T}^{(K)},\sigma_d d_{\frak{T}})$.

\vspace{0.5cm}

{\it Step 3: Showing that the embeddings are asymptotically close}

\vspace{0.5cm}

The last remaining step to conclude the proof is to show that 
  \begin{equation}\label{eq:Gfordres}
  \lim_{n\to\infty }\sup_{x\in\frak{T}^{(n,K)}}d_{\R^d}(n^{-1/4}\phi_{\T^{\text{GW}}_n}(x),\sqrt{\sigma}\phi_{\frak{T}}(\Upsilon_{n,K}(x)))=0,
    \end{equation}
    almost surely. 
 Let  $\hat{\Upsilon}_{n,K}$ be the lexicographical-order preserving homeomorphism between $(\frak{T}^{\text{GW}}_n,n^{-1/2}d^{\T^{\text{GW}}_n})$ and $(\frak{T}^{(K)},\sigma d_{\frak{T}})$ which is linear along the edges (and which exists by \eqref{eq:underlyinggeometricestimate}). By  \eqref{eq:underlyinggeometricestimate} we have
  \begin{equation}\label{eq:Gford}
  \lim_{n\to\infty }\sup_{x\in\frak{T}^{(n,K)}}d_{\R^d}(n^{-1/4}\phi_{\T^{\text{GW}}_n}(x),\sqrt{\sigma}\phi_{\frak{T}}(\hat{\Upsilon}_{n,K}(x)))=0,
    \end{equation} 
    almost surely.
  It can be deduced from Lemma \ref{lem:analogoustores_2a} that 
    \begin{equation}
  \lim_{n\to\infty} \sup_{x\in\frak{T}^{(n,K)}} d_{\frak{T}^{(K)}}(\Upsilon_{n,K}(x),\hat{\Upsilon}_{n,K}(x))=0.
  \end{equation}
  That, together with \eqref{eq:Gford} and the continuity of $\phi_{\frak{T}}$ yields \eqref{eq:Gfordres}.

\end{proof}

\subsection{Proof of the preliminary lemmas for Condition (G)}
\begin{proof}[Proof of Lemma \ref{compacity_tnk}]
For any $w:[0,1]\to \R$, let 
\begin{equation}\label{eq:defofmin}
m_{w}(x,y):=\min\{w(t):t\in [x,y]\}.
\end{equation}

For any $x\in[0,1]$, let
\begin{equation}\label{eq:defof[]}
[x]_n:=w_n(\gamma_n(x)),
\end{equation}
where $w_n$ is the search depth process associated to $\T^{\text{GW}}_n$. First, note that for any $x,y\in[0,1]$
\[d^{\T^{\text{GW}}_n}([x]_n,[y]_n)=\hat{w}_n(\gamma_n(x))+\hat{w}_n(\gamma_n(y))-2m_{\hat{w}_n}(\gamma_n(x),\gamma_n(y)).\]

Let $0\leq U_{(1)}\leq U_{(2)}\leq \dots \leq U_{(K')}\leq 1$ be the points $U_i,i=1,\dots,K'$ ordered in increasing order.
Therefore
\[ 
\begin{aligned}
&\max_{x\in V(\T^{\text{GW}}_n)}\min_{i=1,\dots,K'}d^{\T^{\text{GW}}_n}(U^n_i,x)=\\ 
\leq & \max_{i=1\dots K'-1}\left[2\max_{s\in [\gamma_n(U_{(i)}),\gamma_n(U_{(i+1)})]} \hat{w}_n(s)-2m_{\hat{w}_n}(\gamma_n(U_{(i)}),\gamma_n(U_{(i+1)}))\right].
\end{aligned}
\]
On the other hand, from \eqref{eq:lastlines} and \eqref{eq:papapa} we get that 
\[
\begin{aligned}
&n^{-1/2}\max_{i=1\dots K'-1}\left[2\max_{s\in [\gamma_n(U_{(i)}),\gamma_n(U_{(i+1)})]} \hat{w}_n(s)-2m_{\hat{w}_n}(\gamma_n(U_{(i)}),\gamma_n(U_{(i+1)}))\right]\\
\stackrel{n\to\infty}{\to}& \frac{2}{\sigma_Z} \max_{i=1,\dots K'-1}\left[2\max_{s\in[U_{(i)},U_{(i+1)}]}{\bold e}_{s}-2m_{\bold{e}}(U_{(i)},U_{(i+1)})\right]
\end{aligned}
\]
But, since for any $\eta$, there exists $K'$ large enough such that $U_i,i=1,\dots K'$ is $\eta$-dense\footnote{i.e., for all $x\in[0,1]$ there exists an $i\in\{1,\dots,K'\}$ such that the distance from $x$ to $U_i$ is smaller than $\eta$.} in $[0,1]$ and the fact that $\bold{e}$ is uniformly continuous, we get that 
\[
\lim_{K'\to\infty} {\bf P} \left[ \max_{i=1,\dots K'-1}\left[2\max_{s\in[U_{(i)},U_{(i+1)}]}{\bold e}_{s}-2m_{\bold{e}}(U_{(i)},U_{(i+1)})\right]\geq \epsilon\right]=0
\]
for all $\epsilon>0$. This finishes the proof.
\end{proof}

 \begin{proof}[Proof of Lemma \ref{lem:analogoustores_2a}]
 By Lemma \ref{prop_dist0} we get that
\begin{equation}\label{eq:embeddedoverunderlying-1}
\lim_{n\to\infty}\bold{P}_n\left[\abs{\frac{d^{(n,K)}(\text{root}^{\ast},V^i_n)}{n^{-1/2}d^{\T^{\text{GW}}_n}(\text{root},U^i_n)}-\sigma_{G}}\geq\epsilon\right]=0,
\end{equation}
for all $\epsilon>0$, where $V^i_n$ is as in the beginning of Section \ref{sect:condG}.

We will start deducing that
\begin{equation}\label{eq:embeddedoverunderlying}
\lim_{n\to\infty}\bold{P}_n\left[\abs{d^{(n,K)}(\text{root}^{\ast},V^i_n)-\sigma_{G}n^{-1/2}d^{\T^{\text{GW}}_n}(\text{root},U^i_n)}\geq\epsilon\right]=0,
\end{equation}
for all $\epsilon>0$.

 Since $d^{\T^{\text{GW}}_n}(\text{root},V^i_n)\leq \sup_{x\in[0,1]}\hat{w}_n(x)$, by virtue of Proposition \ref{prop:tourconvergence} we can say that, for each $\eta>0$, there exists $C$ such that
\begin{equation}\label{eq:jejenolemma}
\limsup_{n\to\infty}{\bf P}_n\left[ n^{-1/2} d^{\T^{\text{GW}}_n}(\text{root}, U^i_n) \geq C\right]\leq \eta.
\end{equation}
 Once we know this term cannot be larger than $C$, we can apply~\eqref{eq:embeddedoverunderlying-1} with $\epsilon/C$ instead of $\epsilon$ to obtain~\eqref{eq:embeddedoverunderlying}. 

 By Lemma ~\ref{compacity_tnk} we have that for any $v\in V(\T^{\text{GW}}_n)$ there exists $i=i(v)\leq K'$ such that $d^{\T^{\text{GW}}_n}(U_n^i,v)$ is small.
 More precisely, for any $\epsilon,\eta>0$ there exists $K'$ such that
 \begin{equation}\label{eq:eleven}
 \bold{P}_n\left[\max_{v\in V(\T^{\text{GW}}_n(K'))}n^{-1/2} d^{\T^{\text{GW}}_n}(v, U^{i(v)}_n)\geq \epsilon\right]\leq \eta.
 \end{equation} 
Therefore, by the display above and the fact that \begin{equation}
 d^{(n,K)}(\phi_{\T^{\text{GW}}_n}(\cdot),\phi_{\T^{\text{GW}}_n}(\cdot))\leq n^{-1/2}d^{\T^{\text{GW}}_n}(\cdot,\cdot)
 \end{equation}
  we get that 
 \begin{equation}\label{eq:mgg}
 \bold{P}_n\left[\max_{v\in V(\T^{\text{GW}}_n(K'))}d^{(n,K)}(\phi_{\T^{\text{GW}}_n}(\bar{\pi}_{\T^{\text{GW}}_n}(U^{i(v)}_n)),\phi_{\T^{\text{GW}}_n}(\bar{\pi}_{\T^{\text{GW}}_n}(v)))\geq \epsilon \right]\leq \eta.
 \end{equation}
 On the other hand 
 \begin{equation}\label{eq:piojore}
\begin{aligned}
 &\abs{\sigma_G d^{\T^{\text{GW}}_n}(\text{root},v)- d^{(n,K')}(\text{root}^\ast,\phi_{\T^{\text{GW}}_n}(\bar{\pi}_{\T^{\text{GW}}_n}(v)))}\\
\leq &\abs{\sigma_G d^{\T^{\text{GW}}_n}(\text{root},v)-\sigma_G d^{\T^{\text{GW}}_n}(\text{root},U^{i(v)}_n)}\\
& \qquad + \abs{\sigma_G d^{\T^{\text{GW}}_n}(\text{root},U^{i(v)}_n)-d^{(n,K')}(\text{root}^\ast ,\phi_{\T^{\text{GW}}_n}(\bar{\pi}_{\T^{\text{GW}}_n}(U^{i(v)}_n)))}\\
&\qquad +\abs{d^{(n,K')}(\text{root}^\ast ,\phi_{\T^{\text{GW}}_n}(\bar{\pi}_{\T^{\text{GW}}_n}(U^{i(v)}_n)))- d^{(n,K')}(\text{root}^\ast,\phi_{\T^{\text{GW}}_n}(\bar{\pi}_{\T^{\text{GW}}_n}(v)))} \\
\leq & \sigma_G d^{\T^{\text{GW}}_n}(U^{i(v)}_n,v)+\abs{\sigma_G d^{\T^{\text{GW}}_n}(\text{root},U^{i(v)}_n)-d^{(n,K')}(\text{root}^\ast ,V_{i(v)}^n)}\\
& \qquad+d^{(n,K')}(\phi_{\T^{\text{GW}}_n}(\bar{\pi}_{\T^{\text{GW}}_n}(v)),\phi_{\T^{\text{GW}}_n}(\bar{\pi}_{\T^{\text{GW}}_n}(U^{i(v)}_n))).
\end{aligned}
\end{equation}
The first and third summands of the right hand side in the display above can be controlled by displays \eqref{eq:eleven} and \eqref{eq:mgg} respectively, while the second summand can be controlled by \eqref{eq:embeddedoverunderlying}. This proves the lemma.
 \end{proof}

\begin{proof}[Proof of Lemma \ref{lem:vhart}]
Since we are under the assumption that $(U^i_n)_{i\in\N}$ (the uniform vertices in $\T^{\text {GW}}_n$) are chosen as $w_n(\gamma_n(U_i))$ (where we recall that $w_n$ is the depth-first search around $\T^{\text{GW}}_n$ and $\gamma_n$ is as in \eqref{eq:defofgamma}) we have that
 \begin{equation}\label{eq:holahola}
d^{\T^{\text{GW}}_n}(\text{root},U^i_n)=\hat{w}_n(\gamma_n(U_i)).
\end{equation}
On the other hand
\begin{equation}\label{eq:holachao}
d_{\frak{T}}(\text{root}, [U_i])= \bold{e}_{U_i}.
\end{equation}
Furthermore, it follows from the definition of $\gamma_n$ that 
\begin{equation}\label{eq:lastlines}
\sup_{x\in[0,1]} \abs{ \gamma_n(x)-x }\to0, \quad {\bf P}\text{-a.s.}
\end{equation} as $n\to\infty$. 
Hence, we get that 
\begin{equation}\label{eq:gammantoid}
 \gamma_n(U_i)\to U_i  \quad \text{ for all } i\in\N.
\end{equation}

Therefore it follows from \eqref{eq:papapa}, \eqref{eq:holahola}, \eqref{eq:holachao} and \eqref{eq:gammantoid} that
\begin{equation}
n^{-1/2} d^{\T^{\text{GW}}_n}(\text{root},U^i_n)\to \sigma d_{\frak{T}}(\text{root}, [U_i])
 \end{equation}
We would also like to have,  for all $i,j\leq K$, that 
\begin{equation}\label{eq:premosa}
n^{-1/2}d^{\T^{\text{GW}}_n}(\text{root},b^{\T^{\text{GW}}_n}(U_i^n, U_j^n))\to \sigma d_{\frak{T}}(\text{root},b^{\frak{T}}([U_i],[U_j])).
\end{equation} 

Since 
\[
d^{\T^{\text{GW}}_n}(\text{root},b^{\T^{\text{GW}}_n}(U_i^n, U_j^n))= m_{\hat{w}_n}(\gamma_n(U_i),\gamma_n(U_j))\]
and
\[ d_{\frak{T}}(\text{root},b^{\frak{T}}([U_i],[U_j]))=m_{\bold{e}}(U_i,U_j),
\]
display \eqref{eq:premosa} can be deduced from 
  \begin{equation}\label{eq:nopuedodarotrabatalla}
\lim_{n\to\infty}\sup_{x,y\in[0,1]} \abs{n^{-1/2}m_{\hat{w}_n}(\gamma_n(x),\gamma_n(y))- \sigma m_{\bold{e}}(x,y)} = 0 \quad {\bf P}\text{-a.s. }
\end{equation}

\vspace{0.5cm}

{\it Step 1: Proof of \eqref{eq:nopuedodarotrabatalla}.}

\vspace{0.5cm}

For $x,y\in[0,1]$, $x<y$, let $z_{\text{min}}\in[x,y]$ be the point where $m_{\bold{e}}(x,y)$ is attained. We have
\begin{equation}
n^{-1/2}\hat{w}_n(z_{\text{min}})\leq \sigma \bold{e}_{z_{\text{min}}}+\abs{n^{-1/2}\hat{w}_n(z_{\text{min}})-\sigma\bold{e}_{z_{\text{min}}}}
\end{equation}
Therefore
\begin{equation}\label{eq:lasfloresbarce}
n^{-1/2}m_{\hat{w}_n}(x,y)\leq \sigma m_{\bold{e}}(x,y)+\abs{n^{-1/2}\hat{w}_n(z_{\text{min}})-\sigma\bold{e}_{z_{\text{min}}}}
\end{equation}
Similarly, letting $z_{\text{min}}^n$ be the point where $m_{\hat{w}_n}(x,y)$ is attained. 
We have 
\begin{equation}
\sigma \bold{e}_{z_{\text{min}}^n}\leq n^{-1/2}\hat{w}_n(z^n_{\text{min}})+\abs{n^{-1/2}\hat{w}_n(z^n_{\text{min}})-\sigma\bold{e}_{z^n_{\text{min}}}}
\end{equation}
and
\begin{equation}\label{eq:lasfloresbarce2}
\sigma m_{\bold{e}}(x,y)\leq n^{-1/2}m_{\hat{w}_n}(x,y)+\abs{n^{-1/2}\hat{w}_n(z_{\text{min}}^n)-\sigma\bold{e}_{z^n_{\text{min}}}}
\end{equation}
From \eqref{eq:lasfloresbarce} and \eqref{eq:lasfloresbarce2}, we get that
\begin{equation}
\abs{\sigma m_{\bold{e}}(x,y)-n^{-1/2}m_{\hat{w}_n}(x,y)}\leq\sup_{z\in[0,1]}\abs{n^{-1/2}\hat{w}_n(z)-\sigma \bold{e}_{z}}
\end{equation}
Therefore, it follows from \eqref{eq:papapa} that 
\[
\lim_{n\to\infty}\sup_{x,y\in[0,1]} \abs{n^{-1/2}m_{\hat{w}_n}(x,y)- \sigma m_{\bold{e}}(x,y)} = 0 \quad {\bf P}\text{-a.s. }
\]
The display above, together with the continuity of $m_{\bold{e}}(\cdot,\cdot)$, (which, in turn, follows from the continuity of $\bold{e}$) give that
\[\lim_{\delta\to0}\limsup_{n\to\infty}\sup\left\{ \abs{n^{-1/2}m_{\hat{w}_n}(x',y')-\frac{2}{\sigma_Z}m_{\bold{e}}(x,y)}: \abs{y'-y}\leq\delta,\abs{x'-x}\leq\delta\right\}=0 .\]
Display \eqref{eq:nopuedodarotrabatalla} follows from the display above and\eqref{eq:lastlines}.

\vspace{0.5cm}

{\it Step 2: Proving the main result}

\vspace{0.5cm}

Our next goal is to prove that \eqref{eq:mosa}
 Fix $i,j\leq K$ distinct and $\epsilon>0$. By Lemma \ref{close_cut_point}, with high probability there is a vertex $v^+\in V(\T^{\text{GW}}_n)$ with $b^{\T^{(n,K)}}(U_i^n,U_j^n)\prec \bar{\pi}_{\T^{\text{GW}}_n}(v^+)$ and $n^{-1/2}d^{\T^{\text{GW}}_n}(\bar{\pi}_{\T^{\text{GW}}_n}(v^+),b^{\T^{\text{GW}}_n}(U_i^n,U_j^n))\leq \epsilon$. Also by Lemma \ref{close_cut_point}, there is, with high probability, $v^-\in V(\T^{\text{GW}}_n)$ with $b^{\T^{\text{GW}}_n}(U_i^n,U_j^n)\succ \bar{\pi}_{\T^{\text{GW}}_n}(v^-)$ and $n^{-1/2}d^{\T^{\text{GW}}_n}(\bar{\pi}_{\T^{\text{GW}}_n}(v^-),b^{\T^{\text{GW}}_n}(U_i^n,U_j^n))\leq \epsilon$.  
 Therefore, with high probability
 \begin{equation}
 n^{-1/2}d^{\T^{\text{GW}}_n}(\bar{\pi}_{\T^{\text{GW}}_n}(v^-),\bar{\pi}_{\T^{\text{GW}}_n}(v^+))\leq 2\epsilon
 \end{equation} 
 and consequently,
 \begin{equation}\label{eq:toto1}
 d^{(n,K)}(\phi_{\T^{\text{GW}}_n}(\bar{\pi}_{\T^{\text{GW}}_n}(v^-)),\phi_{\T^{\text{GW}}_n}(\bar{\pi}_{\T^{\text{GW}}_n}(v^+)))\leq 2\epsilon
 \end{equation}
 with high probability.
 
  Note that, by Lemmas \ref{close_cut_point} and \ref{lem:analogoustores_2a}, we get that
\begin{equation}\label{eq:toto2}
\lim_{n\to\infty} \bold{P}_n\left[\abs{d^{(n,K)}(\text{root}^\ast,\phi_{\T^{\text{GW}}_n}(\bar{\pi}_{\T^{\text{GW}}_n}(v^+)))-\sigma_Gn^{-1/2}d^{\T^{\text{GW}}_n}(\text{root},v^+)} \geq \epsilon \right]=0
\end{equation}
and
\begin{equation}\label{eq:toto3}
\lim_{n\to\infty} \bold{P}_n\left[\abs{d^{(n,K)}(\text{root}^\ast,\phi_{\T^{\text{GW}}_n}(\bar{\pi}_{\T^{\text{GW}}_n}(v^-)))-\sigma_Gn^{-1/2}d^{\T^{\text{GW}}_n}(\text{root},v^-)} \geq \epsilon \right]=0
\end{equation}

Moreover, since $b^{\T^{\text{GW}}_n}(U_i^n,U_j^n)\prec \bar{\pi}_{\T^{\text{GW}}_n}(v^+)$ and $b^{\T^{\text{GW}}_n}(U_i^n,U_j^n)\succ \bar{\pi}_{\T^{\text{GW}}_n}(v^-)$, $b^{\T^{(n,K)}}(V^i_n,V^j_n)$ lies in the interval $[\bar{\pi}_{\T^{\text{GW}}_n}(v^-),\bar{\pi}_{\T^{\text{GW}}_n}(v^+)]$.
Therefore, it follows from \eqref{eq:toto1}, \eqref{eq:toto2} and \eqref{eq:toto3} that
\begin{equation}
\lim_{n\to\infty} \bold{P}_n\left[\abs{d^{(n,K)}(\text{root}^\ast,b^{\T^{(n,K)}}(V^i_n,V^j_n))-\sigma_Gn^{-1/2}d^{\T^{\text{GW}}_n}(\text{root},b^{\T^{\text{GW}}_n}(U_i^n,U_j^n))} \geq 4 \epsilon \right]=0.
\end{equation}
This, together with \eqref{eq:premosa}, shows \eqref{eq:mosa} for the case $i\neq j$.

For the case $i=j$, note that $b^{\frak{T}}([U_i],[U_i])=[U_i]$ and $b^{\T^{\text{GW}}_n}(U^i_n,U^i_n)=U^i_n$.
Therefore, in the case $i=j$ \eqref{eq:mosa} reduces to
\begin{equation}
n^{-1/2}d^{(n,K)}(\text{root},V_i^n)\to \sigma_dd_{\frak{T}}(\text{root},[U_i]).
\end{equation}
But the display above follows from \eqref{eq:premosa} and 
\eqref{eq:embeddedoverunderlying}.
  \end{proof}
  \begin{proof}[Proof of Lemma \ref{lem:eleele}]
 It follows directly from \eqref{eq:mosa} that if \[d_{\frak{T}}(\text{root},b^{\frak{T}}([U_i], [U_j]))>d_{\frak{T}}(\text{root},b^{\frak{T}}([U_i], [U_k]))\] then \[d^{(n,K)}(\text{root}^\ast,b^{\T^{(n,K)}}(V^i_n, V^j_n))>d^{(n,K)}(\text{root}^*,b^{\T^{(n,K)}}(V^i_n, V^k_n))\] for $n$ large enough. It remains to treat the case when 
\begin{equation}\label{eq:case}
d_{\frak{T}}(\text{root},b^{\frak{T}}([U_i], [U_j]))=d_{\frak{T}}(\text{root},b^{\frak{T}}([U_i], [U_k])).
\end{equation} 

It follows from Theorem 4.6 of \cite{DuLG} that $\frak{T}$ is a tree in which no vertex has degree larger than $3$. In particular, for all different $i,j,k\leq K$ we have that $b^{\frak{T}}([U_i],[U_j]),b^{\frak{T}}([U_j],[U_k])$ and $b^{\frak{T}}([U_i],[U_k])$ cannot be all equal. 

Therefore in case \eqref{eq:case}, since $b^{\frak{T}}([U_i], [U_j])$ and $b^{\frak{T}}([U_i], [U_k])$ lie in $[\text{root},[U_i]]$, we have \[b^{\frak{T}}([U_i], [U_j])=b^{\frak{T}}([U_i], [U_k]).\] Hence, \[b^{\frak{T}}([U_k],[U_j])\neq b^{\frak{T}}([U_i],b[U_j]).\]  Therefore, since both $b^{\frak{T}}([U_i],[U_j])$ and $b^{\frak{T}}([U_k],[U_j])$ lie in $[\text{root},[U_j]]$, either $b^{\frak{T}}([U_j], [U_k])\prec b^{\frak{T}}([U_i], [U_j])$ or $b^{\frak{T}}([U_j], [U_i])\prec b^{\frak{T}}([U_k], [U_j])$. But, since by assumption $b^{\frak{T}}([U_i], [U_j])=b^{\frak{T}}([U_i], [U_k])$  we cannot have $b^{\frak{T}}([U_j], [U_k])\prec b^{\frak{T}}([U_i], [U_j])$, therefore \[b^{\frak{T}}([U_j], [U_j])\prec b^{\frak{T}}([U_k], [U_j])\] and
\begin{equation}
d_{\frak{T}}(\text{root}, b^{\frak{T}}([U_j], [U_i]))< d_{\frak{T}}(\text{root}, b^{\frak{T}}([U_j], [U_k])).
\end{equation}
Hence, by \eqref{eq:mosa} we get that
\begin{equation}
d^{(n,K)}(\text{root}^\ast, b^{\T^{(n,K)}}(V^j_n, V^i_n))< d^{(n,K)}(\text{root}^\ast, b^{\T^{(n,K)}}(V^j_n, V^k_n))
\end{equation}
for $n$ large enough. This, together with the fact that both $b^{\T^{(n,K)}}(V^j_n, V^i_n)$ and $b^{\T^{(n,K)}}(V^j_n, V^k_n)$ lie in $[\text{root}^*,V^j_n]$, yields that
\begin{equation}
b^{\T^{(n,K)}}(V^j_n, V^i_n)\prec b^{\T^{(n,K)}}(V^j_n, V^k_n).
\end{equation}
In particular, $b^{\T^{(n,K)}}(V^j_n, V^k_n)$ lies in $[b^{\T^{(n,K)}}(V^j_n, V^i_n),V^k_n]$.
From this, it follows from the tree-structure that 
\[
 b^{\T^{(n,K)}}(V^j_n, V^i_n)= b^{\T^{(n,K)}}(V^i_n, V^k_n)
\]
and 
\begin{equation}
d^{(n,K)}(\text{root}^\ast, b^{\T^{(n,K)}}(V^j_n, V^i_n))= d^{(n,K)}(\text{root}^\ast, b^{\T^{(n,K)}}(V^i_n, V^k_n))
\end{equation}
for $n$ large enough.
This finishes the proof of the lemma.
\end{proof}

  
\subsection{Asymptotic thinness}

It remains to show that the graphs are asymptotically thin.
Our goal is to prove the two following lemmas.
\begin{lemma}\label{lem:underlyingbubblesaremicroscopic} For all $\epsilon>0$,
\[
\lim_{K\to\infty}\limsup_{n\to\infty}\bold{P}_n\left[n^{-1/2}\Delta^{(n,K)}_{\omega_n}\geq \epsilon\right]=0
\]
\end{lemma}
\begin{lemma}\label{lem:asymthin}
For all $\epsilon>0$
\[\lim_{K\to\infty}\lim_{n\to\infty}\bold{P}_n[n^{-1/4}\Delta^{(n,K)}_{\Z^d}\geq\epsilon]=0.\]
\end{lemma}

The strategy to prove the two lemmas above is to adapt \eqref{eq:equation1} and \eqref{eq:equation2} from $(\T^{\text{GW}}_n,\phi_{\T^{\text{GW}}_n})$ to $(\T^{(n,K)},\phi_{\T^{\text{GW}}_n})$.  We will need the following result.

\subsubsection{A preliminar lemma}
\begin{lemma}\label{lem:preasymthin}
 For each $v\in V(\T^{\text{GW}}_n)$, we can find a vertex $v'=v'(v,K)\in  V(\T^{\text{GW}}_n(K))$ that satisfies $\phi_{\T^{\text{GW}}_n}(v')=\pi^{(n,K)}(\phi_{\T^{\text{GW}}_n}(v))$ and such that 
\begin{equation}
\lim_{K\to\infty}\limsup_{n\to \infty} {\bf P}_n\Bigl[\sup_{v\in V(\T^{\text{GW}}_n)} n^{-1/2}d^{\T^{\text{GW}}_n}(v,v') \geq \epsilon \Bigr]=0,
\end{equation}
for all $\epsilon>0$.
\end{lemma}
\begin{proof}[Proof of Lemma \ref{lem:preasymthin}]
For each $v\in V(\T^{\text{GW}}_n)$, consider the vertex $v''$ of $\T^{\text{GW}}_n(K)$ given by $v''=\overline{\pi}_{\T^{\text{GW}}_n}(\pi_{\T^{\text{GW}}_n(K)}(v)))$. It follows directly from the definitions of $\overline{\pi}_{\T^{\text{GW}}_n}$ (see Definition \ref{def_strongcut}) and $\pi_{\T^{\text{GW}}_n(K)}$ (see display \eqref{eq:defoftgwk}) that $v''\in [\text{root},v]$. 

Observe that, by definition, $\pi^{(n,K)}(\phi_{\T^{\text{GW}}_n}(v))$ has a pre-image which is in $[\text{root},v]$, i.e., there exists $v'\in [\text{root},v]$ with $\phi_{\T^{\text{GW}}_n}(v')=\pi^{(n,K)}(\phi_{\T^{\text{GW}}_n}(v))$. 
Moreover, by the definition of $\pi^{(n,K)}$, we have that $v'$ can be chosen to belong in $[\text{root},\pi_{\T^{\text{GW}}_n(K)}(v)]$ and therefore $v'\in V(\T^{\text{GW}}_n(K))$.

Our next goal is to show that $v''\preceq v'$. Assume for contradiction that $v'\prec v''$. Since $v''$ has a loopless image (recall Definition \ref{def_strongcut}), and we are assuming $v'\prec v''$, we have that $\phi_{\T^{\text{GW}}_n}(v')$ is outside the connected component of $\phi_{\T^{\text{GW}}_n}(v)$ when $\phi_{\T^{\text{GW}}_n}(v'')$ is removed from $\phi_{\T^{\text{GW}}_n}(\T^{\text{GW}}_n)$. Moreover, using again that $v''$ has a loopless image, we get that $\phi_{\T^{\text{GW}}_n}(v'')$ is one of the cut-points which is in the path from $\phi_{\T^{\text{GW}}_n}(v)$ to $\text{root}^\ast$. Furthermore, since $v''\in \T^{\text{GW}}_n(K)$, we have that $\phi_{\T^{\text{GW}}_n}(v'')\in V^\ast(\T^{(n,K)})$. 

Since, by definition of $v'$, we have $\phi_{\T^{\text{GW}}_n}(v')=\pi^{(n,K)}(\phi_{\T^{\text{GW}}_n}(v))$, the two previous facts are in contradiction with the definition of $\pi^{(n,K)}(\phi_{\T^{\text{GW}}_n}(v))$ as the first cut-point in $V^\ast(\T^{(n,K)})$ which is in the path from $\phi_{\T^{\text{GW}}_n}(v)$ to $\text{root}^*$. 

We have shown that $v''\preceq v' \preceq v$, which, in turn, implies that $d^{\T^{\text{GW}}_n}(v',v)\leq d^{\T^{\text{GW}}_n}(v'',v)$. Therefore, since we already know that $v'$ satisfies $\phi_{\T^{\text{GW}}_n}(v')=\pi^{(n,K)}(\phi_{\T^{\text{GW}}_n}(v))$, the lemma will follow once we have shown that
\begin{equation}\label{eq:vprimeprime}
\lim_{K\to\infty}\limsup_{n\to \infty} {\bf P}_n\Bigl[\sup_{v\in V(\T^{\text{GW}}_n)} n^{-1/2}d^{\T^{\text{GW}}_n}(v,v'') \geq \epsilon \Bigr]=0,
\end{equation}
for all $\epsilon>0$.

We have that
\begin{align} \label{last_thursday1}
\max_{v\in V(\T^{\text{GW}}_n)} d^{\T^{\text{GW}}_n}(v,v'') & \leq \max_{v\in V(\T^{\text{GW}}_n)} d^{\T^{\text{GW}}_n}(v,\pi_{\T^{\text{GW}}_n(K)}(v)) \\ \nonumber
   & \qquad +\max_{v\in V(\T^{\text{GW}}_n)} d^{\T^{\text{GW}}_n}(\pi_{\T^{\text{GW}}_n(K)}(v),v'').
\end{align}

By~\eqref{eq:equation2} (which hold as a consequence of Propositions \ref{prop:tourconvergence} and \ref{prop:croydonequivalence}), we know that for any $\epsilon >0$,
\begin{equation} \label{last_thursday2}
\lim_{K\to\infty}\limsup_{n\to\infty}\bold{P}_n\left[ n^{-1/2} \max_{v\in V(\T^{\text{GW}}_n)} d^{\T^{\text{GW}}_n}(v,\pi_{\T^{\text{GW}}_n(K)}(v))\geq \epsilon \right]=0.
\end{equation}

Since $v''=\overline{\pi}_{\T^{\text{GW}}_n}(\pi_{\T^{\text{GW}}_n(K)}(v)))$, it follows from Lemma \ref{close_cut_point} and a straightforward union bound over $K$ events to see that
\[
\limsup_{n\to \infty} {\bf P}_n\Bigl[ n^{-1/2} \max_{v\in V(\T^{\text{GW}}_n)} d^{\T^{\text{GW}}_n}(\pi_{\T^{\text{GW}}_n(K)}(v), v'') \geq \epsilon \Bigr]=0,
\]
for any $\epsilon>0$. 

The two last displays prove the lemma.

\subsubsection{Proof of asymptotic thinnes}
\end{proof}
\begin{proof}[Proof of Lemma \ref{lem:underlyingbubblesaremicroscopic}]
Since 
\[
 \Delta^{(n,K)}_{\omega_n}\leq 2\max_{x\in V(\omega_n)} d_{\omega_n}(x,\pi^{(n,K)}(x)),
 \]
 it is enough to control the maximum in the right hand side above. 
 
 For any $x\in V(\omega_n)$ consider $v=v(x)$ one of the pre-images of $x$ under $\phi_{\T^{\text{GW}}_n}$. By Lemma \ref{lem:preasymthin}, there also exists $v'(x)$ such that $\phi_{\T^{\text{GW}}_n}(v'(x))=\pi^{(n,K)}(\phi_{\T^{\text{GW}}_n}(v))$ and 
 \[
 \lim_{K\to\infty}\limsup_{n\to \infty} {\bf P}_n\Bigl[\max_{x\in V(\omega_n)} n^{-1/2}d^{\T^{\text{GW}}_n}(v(x),v'(x)) \geq \epsilon \Bigr]=0
 \]
 for all $\epsilon>0$.
 
 Since $d_{\omega_n}(\phi_{\T^{\text{GW}}_n}(v_1),\phi_{\T^{\text{GW}}_n}(v_2))\leq d^{\T^{\text{GW}}_n}(v_1,v_2)$ for any $v_1,v_2\in V(\T^{\text{GW}}_n)$,
 we get that
 \begin{align*}
\max_{x\in V(\omega_n)} d_{\omega_n}(x,\pi^{(n,K)}(x)) & =  \max_{x\in V(\omega_n)} d_{\omega_n}(\phi_{\T^{\text{GW}}_n}(v(x)),\phi_{\T^{\text{GW}}_n}(v'(x))) \\
& \leq \max_{x\in V(\omega_n)} d^{\T^{\text{GW}}_n}(v(x),v'(x)),
 \end{align*}
  the lemma follows directly from the previous equation.

\end{proof}

\begin{proof}[Proof of Lemma \ref{lem:asymthin}]
By the triangular inequality 
\[\Delta^{(n,K)}_{\Z^d}\leq 2\max_{v\in V(\T_n^{\text{GW}})} d_{\Z^d}(\phi_{\T_n^{\text{GW}}}(v),\pi^{(n,K)}(\phi_{\T_n^{\text{GW}}}(v))),\]
therefore, it suffices to control the maximum in the right hand side above.
For any $v\in V(\T_n^{\text{GW}})$, let $v'$ be as in Lemma \ref{lem:preasymthin}. We have that
\begin{equation}\label{eq:fargodakota}
\max_{v\in V(\T_n^{\text{GW}})} d_{\Z^d}(\phi_{\T_n^{\text{GW}}}(v),\pi^{(n,K)}(\phi_{\T_n^{\text{GW}}}(v)))=\max_{v\in V(\T_n^{\text{GW}})} d_{\Z^d}(\phi_{\T_n^{\text{GW}}}(v),\phi_{\T_n^{\text{GW}}}(v')).
\end{equation}

By Lemma \ref{lem:preasymthin} we get that, for any $\delta>0$,
\[
\lim_{K\to\infty}\limsup_{n\to \infty} {\bf P}_n\Bigl[\sup_{v\in V(\T^{\text{GW}}_n)} n^{-1/2}d^{\T^{\text{GW}}_n}(v,v') \geq \delta \Bigr]=0.
\]
Hence, by display \eqref{eq:equation2} (which hold as a consequence of Propositions \ref{prop:tourconvergence} and \ref{prop:croydonequivalence}) we get that
\[
\lim_{K\to\infty}\limsup_{n\to \infty} {\bf P}_n\Bigl[\sup_{v\in V(\T^{\text{GW}}_n)} n^{-1/4}d_{\Z^d}(\phi_{\T_n^{\text{GW}}}(v),\phi_{\T_n^{\text{GW}}}(v')) \geq \epsilon \Bigr]=0.
\]
The lemma follows from the display above and \eqref{eq:fargodakota}. 

\end{proof}

\section{Volume measure convergence}\label{sect:condV}
Recall our choice of random augmented graphs $(\omega_n,( V^i_n)_{i\in\N})_{n\in\N}$ fom Section \ref{sect_abstract_thm}. That is, $\omega_n$ is the trace of critical branching random walks and $V^i_n=\pi_n(\phi_{\T^{\text{GW}}_n}(U^i_n))$, where $(U^i_n)_{i\in\N}$ is an i.i.d.~sequence of uniform vertices of $\T^{\text{GW}}_n$.  
The purpose of this section is to prove condition $(V)_\nu$ for  $(\omega_n,(V^i_n)_{i\in\N})_{n\in\N}$.
We will start by defining the constant $\nu$ of condition $(V)_\nu$.

Recall that Theorem \ref{teo:legalllin} (which was proved in ~\cite{LGL}) states that, in some sense, the vertex-cardinal of the graph obtained by the embedding the tree $\T^{\text{GW}}_n$ is linear in $n$. The strategy of proof used in Section 2 of~\cite{LGL} can be adapted without serious difficulties to prove that the edge-cardinal of the embedded graphs is linear in $n$. It can be shown that 

\begin{theorem}\label{teo:legalllinedge}
Assume ${\bf E}[Z^2]<\infty$ and $d\geq 5$. Let $v_1<v_2<\dots<v_n$ be the vertices of $\T^{\text{GW}}_n$ labeled in increasing lexicographical order. Let us denote 
\begin{align*}
E(\text{vol})_n(k)&=\{(x_1,x_2)\in E(\Z^d): \exists  0\leq i,j\leq k \text{ such that }(v_i,v_j)\in E(\T^{\text{GW}}_n) \\
  &  \qquad \qquad  \text{ and }   \phi_{n}(v_i)=x_1,\phi_{\T^{\text{GW}}_n}(v_j)=x_2\}.
\end{align*}

 Then, there exists $\nu>0$ such that for all $a\in [0,1]$, we have that
\begin{equation}\label{def_nu}
\frac 1n \abs{E(\text{vol})_n(\lfloor an\rfloor)} \to_{n\to \infty} a \nu,
\end{equation}
in probability where $\omega_n$ is chosen under ${\bf P}_n$.
\end{theorem}

The main result of this section is the following:
\begin{proposition}\label{lem:condV}
The sequence of random augmented graphs $(\omega_n,(V^i_n)_{i\in\N})_{n\in\N}$ satisfies condition $(V)_{\nu}$, where the constant $\nu$ is as in Theorem \ref{teo:legalllinedge}.
\end{proposition}
 During the whole of this section we will be working under the couplings at \eqref{eq:papapa} and \eqref{eq:secondcoupling}.


\subsection{Proof of condition $(V)$ assuming three estimates}

For the proof of Proposition \ref{lem:condV} we will need some preliminary definitions. Recall the definition of $\lambda^{\frak{T}}$ from \eqref{eq:lambdaT}. Let $\pi_{\frak{T}^{(K)}}$ be the projection of $\frak{T}$ onto $\frak{T}^{(K)}$ and $\mu^{(K)}$ be the probability measure on $\frak{T}^{(K)}$ defined as $\mu^{(K)}=\pi_{\frak{T}^{(K)}}\circ\lambda^{\frak{T}}$. We use the notation $\mu^{(K)}$ instead of the more natural $\lambda^{(K)}$ since $\lambda^{(K)}$ denotes the (normalized) Lebesgue measure on $\frak{T}^{(K)}$ (recall Section \ref{sect_KISE}). 

Since $\frak{T}^{(n,K)}$ and $\T^{(n,K)}$ are the same objects when viewed as real trees the measures $\mu^{(n,K)}$ and  $\lambda^{(n,K)}$ can also be viewed as defined on $\frak{T}^{(n,K)}$.

\subsubsection{Three technical estimates}

For any $x\in\frak{T}^{(K)}$ let $\overrightarrow{\frak{T}^{(K)}_x}$ be the set of descendants of $x$ (including $x$) in $\frak{T}^{(K)}$.
\begin{lemma}\label{l:convergenceofmuk} 
Let $\mathcal{A}^{n,K}(\epsilon)$ be the event in which
\[
\nu\mu^{(K)}(\overrightarrow{\frak{T}^{(K)}_x})- \nu\mu^{(K)}(x)-\epsilon\leq \mu^{(n,K)}(\Upsilon_{n,K}^{-1}(\overrightarrow{\frak{T}^{(K)}_x})) \leq \nu\mu^{(K)}(\overrightarrow{\frak{T}^{(K)}_x}) +\epsilon
\]
holds for all $x\in \frak{T}^{(K)} $. For all $K\in\N$, 
\[\lim_{n\to\infty}{\bf P}_n\left[\mathcal{A}^{n,K}(\epsilon)\right]=1,\]
for all $\epsilon >0$.
\end{lemma}

 For any $x\in\frak{T}^{(n,K)}$ let $\overrightarrow{\frak{T}^{(n,K)}_x}$ be the set of descendants of $x$ (including $x$) in $\frak{T}^{(n,K)}$. Recall the coupling between $\T^{(n,K)}$ and $\frak{T}^{(K)}$ described in Section \ref{s:usefcoupl}.
\begin{lemma}\label{l:lambdarestolambdak}
For all $K\in\N$,
\[\lim_{n\to\infty}{\bf P}_n\left[\sup_{x\in \frak{T}^{(n,K)}}\abs{\lambda^{(n,K)}(\overrightarrow{\frak{T}_x^{(n,K)}})  - \lambda^{(K)}(\Upsilon_{n,K}(\overrightarrow{\frak{T}^{(n,K)}_x}))}\geq\epsilon\right]=0,\]
for all $\epsilon>0$.
\end{lemma}

Finally we have
\begin{lemma}\label{lem:lambdaktomuk}
For all $\epsilon>0$,
\[\lim_{K\to\infty}\sup_{x\in\frak{T}^{(K)}}\abs{\lambda^{(K)}(\overrightarrow{\frak{T}_x^{(K)}})-\mu^{(K)}(\overrightarrow{\frak{T}_x^{(K)}})}=0.\quad{\bf P}_n\text{-a.s.}\]
\end{lemma}

\subsubsection{Proof of condition (V)}

Here we will prove condition $(V)$ using the results of the previous section.
\begin{proof}[Proof of Proposition \ref{lem:condV}]
We have that
\begin{align*}
&\sup_{x\in V(\T^{(n,K)})}\abs{\mu^{(n,K)}(\overrightarrow{\T_x^{(n,K)}})-\nu\lambda^{(n,K)}(\overrightarrow{\T_x^{(n,K)}})}\\
\leq& \sup_{x\in\frak{T}^{(n,K)}}\abs{\mu^{(n,K)}(\overrightarrow{\frak{T}_x^{(n,K)}})-\nu\mu^{(K)}(\Upsilon_{n,K}(\overrightarrow{\frak{T}_x^{(n,K)}}))}\\
&+\nu\sup_{x\in\frak{T}^{(n,K)}}\abs{\mu^{(K)}(\Upsilon_{n,K}(\overrightarrow{\frak{T}_x^{(n,K)}}))-\lambda^{(K)}(\Upsilon_{n,K}(\overrightarrow{\frak{T}_x^{(n,K)}}))}\\
&+\nu\sup_{x\in\frak{T}^{(n,K)}}\abs{\lambda^{(K)}(\Upsilon_{n,K}(\overrightarrow{\frak{T}_x^{(n,K)}}))-\lambda^{(n,K)}(\overrightarrow{\frak{T}_x^{(n,K)}})}.
\end{align*}

The second summand in the display above converges to $0$ in ${\bf P}_n$-probability as $K\to\infty$ by virtue of Lemma \ref{lem:lambdaktomuk} and the third summand converges to $0$ in ${\bf P}_n$-probability as $n\to\infty$ by virtue of Lemma \ref{l:lambdarestolambdak}. Therefore, it is enough to control the first summand.

It can be deduced from Lemma \ref{lem:lambdaktomuk} and the fact that $\lambda^{(K)}$ has no atoms that, for all $\epsilon>0$,
\[\lim_{K\to\infty}\limsup_{n\to\infty}{\bf P}_n\left[\sup_{x\in\frak{T}^{(n,K)}(x)}\mu^{(K)}(\Upsilon_{n,K}(x))\geq\epsilon\right]=0.\]

Therefore, it follows from Lemma \ref{l:convergenceofmuk} and the fact that $\Upsilon_{n,K}(\overrightarrow{\frak{T}_x^{(n,K)}})=\overrightarrow{\frak{T}_{\Upsilon_{n,K}(x)}^{(K)}}$  that
\[\lim_{K\to\infty}\limsup_{n\to\infty} {\bf P}_n\left[ \sup_{x\in\frak{T}^{(n,K)}}\abs{\mu^{(n,K)}(\overrightarrow{\frak{T}_x^{(n,K)}})-\nu\mu^{(K)}(\Upsilon_{n,K}(\overrightarrow{\frak{T}_x^{(n,K)}}))}\geq\epsilon
\right]=0,\]
for all $\epsilon>0$ and this finishes the proof.
\end{proof}

\subsection{Technical estimates}

The goal of this section is to prove Lemma~\ref{l:convergenceofmuk}, Lemma~\ref{l:lambdarestolambdak} and Lemma~\ref{lem:lambdaktomuk}.

\subsubsection{Proof of Lemma~\ref{l:convergenceofmuk}}

For each $v\in V^\ast(\T^{(n,K)})$, let
\[
\hat{\mu}^{(n,K)}(v):=n^{-1}\abs{\{(e_-,e_+)\in E(\T^{\text{GW}}_n): e_- \prec e_+ \text{ and } \pi^{(n,K)}(\phi_{\T^{\text{GW}}_n}(e_+))=v \}}.
\]
and
\begin{equation}\label{eq:defofmuhat}
\hat{\mu}^{(n,K)}:=\sum_{v\in V^{\ast}(\T^{(n,K)})} \hat{\mu}^{(n,K)}(v)\delta_v.\end{equation}

Notice that this definition counts the edge-volume of $\T^{\text{GW}}_n$ whereas $\mu^{(n,K)}$ counts the edge-volume of $\omega_n$.
As we argued before, we can regard $\hat{\mu}^{(n,K)}$ as a measure on $\frak{T}^{(n,K)}$.
 Lemma \ref{l:convergenceofmuk} will be a direct consequence of the following lemma.
\begin{lemma}\label{lem:croydonconvergenceofmeasures}
Let $\mathcal{B}^{n,K}(\epsilon)$ be the event in which
\[
\mu^{(K)}(\overrightarrow{\frak{T}^{(K)}_x})- \mu^{(K)}(x)-\epsilon\leq \hat{\mu}^{(n,K)}(\Upsilon_{n,K}^{-1}(\overrightarrow{\frak{T}^{(K)}_x})) \leq \mu^{(K)}(\overrightarrow{\frak{T}^{(K)}_x}) +\epsilon
\]
holds for all $x\in \frak{T}^{(K)} $. For all $K\in\N$,
\[\lim_{n\to\infty}{\bf P}_n\left[\mathcal{B}^{n,K}(\epsilon)\right]=1,\]
for all $\epsilon >0$.
\end{lemma}

\begin{proof}

Recall the definition of $\hat{\mu}_n^{(K)}$ in \eqref{eq:defofmuhat} and that of $[\cdot]_n:[0,1]\to V(\T^{\text{GW}}_n)$ from \eqref{eq:defof[]}. We have that
\[
\hat{\mu}_n^{(K)}=\pi^{(n,K)}\circ\phi_{\T^{\text{GW}}_n}\circ[\cdot]_n\circ\text{Leb} ,
\]
and
\[
\mu^{(K)}=\pi_{\frak{T}^{(K)}}\circ[\cdot]\circ \text{Leb},
\]
The last two displays show that both $\hat{\mu}_n^{(K)}$ and $\mu^{(K)}$ are push forwards of the Lebesgue measure in $[0,1]$ through the functions $\pi^{(n,K)}\circ\phi_{\T^{\text{GW}}_n}\circ[\cdot]_n$ and $\pi_{\frak{T}^{(K)}}\circ[\cdot]$ respectively. Therefore, the convergence  of $\Upsilon_{n,K}\circ \pi^{(n,K)}\circ\phi_{\T^{\text{GW}}_n}\circ[\cdot]_n$ towards $\pi_{\frak{T}^{(K)}}\circ[\cdot]$ implies the weak convergence of $\Upsilon_{n,K}\circ\hat{\mu}^{(n,K)}$ towards $\mu^{(K)}$. Therefore, recalling that $\overrightarrow{\frak{T}_x^{(K)}}$ is a closed set of $\frak{T}^{(K)}$, we can see that to prove the lemma it suffices to show that
\begin{equation}\label{eq:phinktophik}
\lim_{n\to\infty}{\bf P}_n\left[\sup_{x\in[0,1]}d_{\mathfrak{T}}(\Upsilon_{n,K}(\pi^{(n,K)}(\phi_{n}([x]_n))),\pi_{\frak{T}^{(K)}}([x]))\geq \epsilon \right] = 0,
\end{equation}
for all $\epsilon>0$.

By \eqref{eq:underlyinggeometricestimate}, ${\bf P}_n$-almost surely, for $n$ large enough, there exists $\hat{\Upsilon}_{n,K}:\frak{T}^{\text{GW}}_{n}(K)\to \frak{T}^{(K)}$ a genealogical order preserving homeomorphism between $(\frak{T}^{\text{GW}}_n(K), n^{-1/2}d^{\T^{\text{GW}}_n})$ and $(\frak{T}^{(K)},\sigma d_{\frak{T}})$ which is linear along the edges of $\frak{T}^{\text{GW}}_n(K)$.
The display above will be proved after we have proved the following two statements
\begin{equation}\label{eq:croydonmeasure}
\sup_{x\in[0,1]}d_{\mathfrak{T}^{(K)}}(\hat{\Upsilon}_{n,K}(\pi_{\T^{\text{GW}}_n(K)}([x]_n)),\pi_{\frak{T}^{(K)}}([x])) \to 0 \quad \text{as } n \to \infty, \quad  {\bf P}_n\text{-a.s.}
\end{equation}
and
\begin{equation}\label{eq:fromcroydontoours}
\lim_{n\to\infty}{\bf P}_n\left[\sup_{v\in V(\T^{\text{GW}}_n)}d_{\mathfrak{T}^{(K)}}(\hat{\Upsilon}_{n,K}(\pi_{\T^{\text{GW}}_n(K)}(v)),\Upsilon_{n,K}(\pi^{(n,K)}(\phi_{\T^{\text{GW}}_n}(v)))) \geq \epsilon \right] =0.
\end{equation}

\vspace{0.2cm}

{\it Proof of equation~\eqref{eq:croydonmeasure}}

\vspace{0.2cm}

Let us denote the non-root leaves of $\frak{T}^{(K)}$ as $l_i:=[U_i],i=1,\dots,K$. The branching point $b^{\frak{T}}(l_i,l_j)$ between $l_i$ and $l_j$, $i\neq j$ will be denoted by $b_{i,j}$. Setting $U_{-1}=1$ and $U_0=0$, we have that for any $x\in[0,1]$ there exists unique $i=i(x),j=j(x)\in\{-1,\dots,K\}$ such that $x\in[U_i,U_j]$ and $(U_i,U_j)\cap\{U_i:i=-1,\dots,K\}=\emptyset$. For $\alpha\in[0,d_{\frak{T}}(b_{i,j},l_i)]$, let $\alpha \langle b_{i,j},l_i \rangle $ be the point in $\frak{T}^{(K)}$ at distance $\alpha$ of $b_{i,j}$ and lying in the line segment whose endpoints are $b_{i,j}$ and $l_i$.
 Recall that, for any $w:[0,1]\to \R$, $m_{w}(x,y)$ denotes $\min\{w(t):t\in [x,y]\}$.
  It follows (see Figure 8) that
\begin{equation}\label{eq:nicaragua1}
\pi_{\frak{T}^{(K)}}([x])=
\begin{cases}
 (m_{\bold{e}}(x,U_i)-m_{\bold{e}}(x,U_j)) \langle b_{i,j},l_i \rangle \quad \text{if } m_{\bold{e}}(x,U_i)\geq m_{\bold{e}}(x,U_j)\\
 (m_{\bold{e}}(x,U_j)-m_{\bold{e}}(x,U_i)) \langle b_{i,j},l_j \rangle \quad \text{otherwise}.
\end{cases}
\end{equation}

\begin{figure}
  \includegraphics[width=\linewidth]{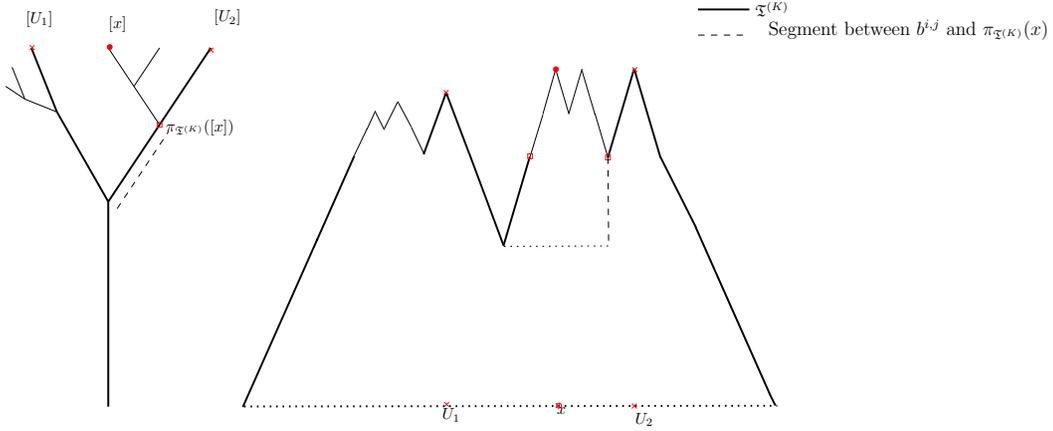}
  \caption{Expresion \eqref{eq:nicaragua1}: The tree $\frak{T}$ and its corresponding first-depth process $\bold{e}$. The subtree $\frak{T}^{(K)}$ is highlighted by bold dark lines. In the figure $K=2$. In this case the point $\pi_{\frak{T}^{(K)}}([x])$ lies in the edge $\langle b^{1,2},[U_1] \rangle$. The segment between $b^{1,2}$ and $\pi_{\frak{T}^{(K)}}([x])$ is depicted by a dashed line. The length of said interval corresponds to $m_{\bold{e}}(x,U_1)-m_{\bold{e}}(x,U_2)$.}
\end{figure}
A similar expression can be deduced for $\pi_{\T^{\text{GW}}_n(K)}( [x]_n)$ in terms of $\hat{w}_n$ instead of $\bold{e}$: Let $l_i^n:=[U_i]_n,i=1,\dots,K$ be the non-root leaves\footnote{ It might be that, for $n$ small, some of the $l_i^n,i=1,\dots,K$ are not actually leaves, but in virtue of \eqref{eq:underlyinggeometricestimate}, for $n$ sufficiently large, all of the $l_i^n,n=1,\dots,K$ are leaves } of $\frak{T}^{\text{GW}}_n(K)$ and let us denote the branching point  $b^{\frak{T}^{\text{GW}}_n}([U_i]_n,[U_j]_n)$ between $l_i^n$ and $l_j^n$ as $b_{i,j}^n$. As before, set $U_{-1}=1$ and $U_0=0$. For $x$ in $[0,1]$, let $i,j$ be the indices such that $x\in [U_i,U_j]$ and $(U_i,U_j)\cap\{U_i:i=-1,\dots,K\}=\emptyset$.
For $\alpha\in[0,d_{\T^{\text{GW}}_n(K)}(b_{i,j}^n,l_i^n)]$, let $\alpha\langle b^n_{i,j},l_i^n \rangle$ be the point which is at $n^{-1/2}d^{\T^{\text{GW}}_n}$-distance $\alpha$ of $b^n_{i,j}$ and lying in the line segment whose endpoints are $b_{i,j}^n$ and $l_i^n$. Then
\begin{equation}\label{eq:nicaragua2}
\begin{aligned}
&\pi_{\T^{\text{GW}}_n(K)}([x]_n)=\\
&\begin{cases}
 n^{-1/2}(m_{\hat{w}_n}(\gamma_n(x),\gamma_n(U_i))-m_{\hat{w}_n}(\gamma_n(x),\gamma_n(U_j))) \langle b^n_{i,j},l^n_i \rangle \quad \text{if } m_{\hat{w}_n}(x,U_i)\geq m_{\hat{w}_n}(x,U_j)\\
 n^{-1/2}(m_{\hat{w}_n}(\gamma_n(x),\gamma_n(U_j))-m_{\hat{w}_n}(\gamma_n(x),\gamma_n(U_i))) \langle b_{i,j}^n,l^n_j \rangle \quad \text{otherwise.}
\end{cases}
\end{aligned}
\end{equation}
Let $x\in[0,1]$ and $e(x)$ the edge of $\frak{T}^{(K)}$ which contains $[x]$ and $e_n(x)$ be the edge of $\frak{T}^{\text{GW}}_n$ which contains $[x]_n$.
It follows from the definition of $\hat{\Upsilon}_{n,K}$ that 
\begin{equation}\label{eq:werring}
\hat{\Upsilon}_{n,K}(\pi_{\T^{\text{GW}}_n(K)}([x]_n))
=
\frac{d^{\T^{\text{GW}}_n}(b_{i,j}^n,\pi_{\T^{\text{GW}}_n(K)}([x]_n))}{d^{\T^{\text{GW}}_n}(e(x))} d_{\frak{T}}(e(x))\langle e(x)\rangle,
\end{equation}
where $\alpha\langle e\rangle$ denotes the point at distance $\alpha $ along $e$.
On the other hand, \[\pi_{\frak{T}^{(K)}}([x])=d_{\frak{T}}(b_{i,j},\pi_{\frak{T}^{(K)}}([x]))\langle e(x)\rangle
.\]
Therefore, \eqref{eq:croydonmeasure} follows after we have showed that
\begin{equation}
\lim_{n\to\infty}\sup_{x\in[0,1]} \abs{d_{\frak{T}}(b_{i,j},\pi_{\frak{T}^{(K)}}([x]))-\frac{d^{\T^{\text{GW}}_n}(b_{i,j}^n,\pi_{\T^{\text{GW}}_n(K)}([x]_n))}{d^{\T^{\text{GW}}_n}(e)} d_{\frak{T}}(e)}=0.
\end{equation}
 By \eqref{eq:underlyinggeometricestimate} we have that
\begin{equation}\label{eq:thisshitiseverlasting}
\lim_{n\to\infty} n^{-1/2}d^{\T^{\text{GW}}_n}(e)=\sigma d_{\frak{T}}(e).
\end{equation}

On the other hand, it follows from \eqref{eq:nopuedodarotrabatalla}
that
\begin{equation}
\lim_{n\to\infty} \abs{n^{-1/2}(m_{\hat{w}_n}(\gamma_n(x),\gamma_n(U_i))-m_{\hat{w}_n}(\gamma_n(x),\gamma_n(U_j)))-\sigma (m_{\bold{e}}(x,U_i)-m_{\bold{e}}(x,U_j))}=0
\end{equation}
and 
\begin{equation}
\lim_{n\to\infty} \abs{n^{-1/2}(m_{\hat{w}_n}(\gamma_n(x),\gamma_n(U_j))-m_{\hat{w}_n}(\gamma_n(x),\gamma_n(U_i)))-\sigma(m_{\bold{e}}(x,U_j)-m_{\bold{e}}(x,U_i))}=0.
\end{equation}
 Therefore, from the expressions \eqref{eq:nicaragua1} and \eqref{eq:nicaragua2} we get that 
 \begin{equation}\label{eq:fasho}
\lim_{n\to\infty}\sup_{x\in[0,1]}\abs{n^{-1/2}d^{\T^{\text{GW}}_n}(b_{i,j}^n,\pi_{\T^{\text{GW}}_n(K)}([x]_n))-\sigma d_{\frak{T}}(b_{i,j},\pi_{\frak{T}^{(K)}}([x])) }=0
\end{equation}

Display \eqref{eq:croydonmeasure} follows from \eqref{eq:thisshitiseverlasting} and \eqref{eq:fasho}.

\vspace{0.2cm}

{\it Proof of equation~\eqref{eq:fromcroydontoours}}

\vspace{0.2cm}

Fix $K\in\N$. For any $v\in V(\T^{\text{GW}}_n)$,  consider $v''=\bar{\pi}_{\T^{\text{GW}}_n}(\pi_{\T^{\text{GW}}_n(K)}(v))$ and $v'$ as in Lemma \ref{lem:preasymthin}.
It is clear that $v'\preceq\pi_{\T^{\text{GW}}_n}(v)$. 
We have that
\begin{equation}
\begin{aligned}
&\abs{d^{(n,K)}(\text{root}^*, \pi^{(n,K)}(\phi_{\T^{\text{GW}}_n}(v)))-\sigma_Gn^{-1/2}d^{\T^{\text{GW}}_n}(\text{root},\pi_{\T^{\text{GW}}_n(K)}(v))}\\
\leq
&\abs{d^{(n,K)}(\text{root}^*, \pi^{(n,K)}(\phi_{\T^{\text{GW}}_n}(v)))- d^{(n,K)}(\text{root}^\ast,\phi_{\T^{\text{GW}}_n}(v'))}\\
+&\abs{d^{(n,K)}(\text{root}^\ast,\phi_{\T^{\text{GW}}_n}(v'))-d^{(n,K)}(\text{root}^\ast,\phi_{\T^{\text{GW}}_n}(\bar{\pi}_{\T^{\text{GW}}_n}(v')))}\\
+&\abs{d^{(n,K)}(\text{root}^\ast,\phi_{\T^{\text{GW}}_n}(\bar{\pi}_{\T^{\text{GW}}_n}(v')))-\sigma_Gn^{-1/2}d^{\T^{\text{GW}}_n}(\text{root},v')}\\
+& \sigma_Gn^{-1/2}\abs{d^{\T^{\text{GW}}_n}(\text{root},\pi_{\T^{\text{GW}}_n(K)}(v))-d^{\T^{\text{GW}}_n}(\text{root},v')}.
\end{aligned}
\end{equation}

Since $\pi^{(n,K)}(\phi_{\T^{\text{GW}}_n}(v))=\phi_{\T^{\text{GW}}_n}(v')$ we have that the first summand above is identically $0$.
Moreover, since $v'\in V(\T^{\text{GW}}_n(K)),$ we have by Lemma \ref{lem:analogoustores_2a},
\begin{equation}
\bold{P}\left[ \sup_{v\in V(\T^{\text{GW}}_n)} \abs{d^{(n,K)}(\text{root}^\ast,\phi_{\T^{\text{GW}}_n(K)}(\bar{\pi}_{\T^{\text{GW}}_n}(v')))-\sigma_Gn^{-1/2}d^{\T^{\text{GW}}_n}(\text{root},v')}\geq \epsilon \right]\stackrel{n\to\infty}{\to} 0
\end{equation}
for all $\epsilon>0$.
On the other hand, since $v'\in[\bar{\pi}_{\T^{\text{GW}}_n}(\pi_{ \T^{\text{GW}}_n(K)}(v)),\pi_{ \T^{\text{GW}}_n(K)}(v)]$ and $\pi_{ \T^{\text{GW}}_n(K)}(v)\in \T^{\text{GW}}_n(K)$, it follows from Lemma \ref{close_cut_point} (and a simple union bound on $K$ events) that
\begin{equation}
\bold{P}\left[\max_{v\in V(\T^{\text{GW}}_n(K))}n^{-1/2}\abs{d^{\T^{\text{GW}}_n}(\text{root},\pi_{\T^{\text{GW}}_n(K)}(v))-d^{\T^{\text{GW}}_n}(\text{root},v')}\geq\epsilon\right]\stackrel{n\to\infty}{\to} 0
\end{equation}
and
\begin{equation}
\bold{P}\left[\max_{v\in V(\T^{\text{GW}}_n(K))}n^{-1/2}\abs{d^{(n,K)}(\text{root}^\ast,\phi_{\T^{\text{GW}}_n}(v'))-d^{(n,K)}(\text{root}^\ast,\phi_{\T^{\text{GW}}_n}(\bar{\pi}_{\T^{\text{GW}}_n}(v')))}\geq\epsilon\right]\stackrel{n\to\infty}{\to} 0
\end{equation}
for all $\epsilon>0$.
Therefore we get that
\begin{equation}\label{eq:olimpic}
\bold{P}\left[\sup_{v\in V(\T^{\text{GW}}_n)}\abs{d^{(n,K)}(\text{root}^*, \pi^{(n,K)}(\phi_{\T^{\text{GW}}_n}(v)))-\sigma_Gn^{-1/2}d^{\T^{\text{GW}}_n}(\text{root},\pi_{\T^{\text{GW}}_n}(v))}\geq\epsilon\right]\stackrel{n\to\infty}{\to}0.
\end{equation}

Let $(a^{n,K}(v),b^{n,K}(v))$ be the edge of $\frak{T}^{\text{GW}}_n$ which contains $\pi_{\T^{\text{GW}}_n(K)}(v)$ and $(\alpha^{(n,K)}(v),\beta^{(n,K)}(v))$ be the edge of $\frak{T}^{(n,K)}$ which contains $\pi^{(n,K)}(\phi_{\T^{\text{GW}}_n}(v))$. Since for any $v\in V(\T^{\text{GW}}_n)$,
\[d^{\T^{\text{GW}}_n}(a^{n,K}(v),\pi_{\T^{\text{GW}}_n}(v))=d^{\T^{\text{GW}}_n}(\text{root},\pi_{\T^{\text{GW}}_n}(v))-d^{\T^{\text{GW}}_n}(\text{root}, a^{n,K}(v))\]
(and a similar expression holds for $d^{(n,K)}(\alpha^{n,K}(v), \pi^{(n,K)}(\phi_{\T^{\text{GW}}_n}(v)))$), it can be shown from 
\eqref{eq:olimpic} that 
\begin{equation}\label{eq:kendrick}
\bold{P}\left[\sup_{v\in V(\T^{\text{GW}}_n)}\abs{d^{(n,K)}(\alpha^{n,K}(v), \pi^{(n,K)}(\phi_{\T^{\text{GW}}_n}(v)))-\sigma_Gn^{-1/2}d^{\T^{\text{GW}}_n}(a^{n,K}(v),\pi_{\T^{\text{GW}}_n}(v))}\geq\epsilon\right]\stackrel{n\to\infty}{\to}0.
\end{equation}

On the other hand, it follows from \eqref{eq:underlyinggeometricestimate} and Lemma \ref{lem:eleele} that
\begin{equation}\label{eq:madgibbs}
\bold{P}\left[\sup_{v\in V(\T^{\text{GW}}_n)}\abs{|(\alpha^{n,K}(v),\beta^{n,K}(v)|_{d^{(n,K)}}-\sigma_Gn^{-1/2}|(a^{n,K}(v),b^{(n,K)}(v))|_{d^{\T^{\text{GW}}_n}}}\geq\epsilon\right]\to0.
\end{equation}

On the other hand, it follows from the definition of $\Upsilon$ that 
\begin{equation}
\Upsilon_{n,K}(\pi^{(n,K)}(\phi_{n}([x]_n)))= 
\frac{d^{(n,K)}(e_-^n(x),\pi^{(n,K)}(\phi_{\T^{\text{GW}}_n}([x]_n)))}{d^{(n,K)}(e^n(x))} d_{\frak{T}}(e(x))\langle e(x)\rangle
\end{equation}

Therefore, in view of \eqref{eq:werring}, display \eqref{eq:fromcroydontoours} follows from \eqref{eq:kendrick} and \eqref{eq:madgibbs}.

\end{proof}

The next lemma shows that the measure $\mu^{(n,K)}$ is approximately equal to $\hat{\mu}^{(n,K)}$, scaled by the constant $\nu$ defined above. 
\begin{lemma}\label{lem:Prokhorov} Let $\nu$ be the constant appearing in Lemma \ref{l:convergenceofmuk}. For all $K\in\N$, 
\[\lim_{n\to\infty} \sup_{x\in \frak{T}^{(n,K)}}\abs{ \mu^{(n,K)}(\overrightarrow{\frak{T}_x^{(n,K)}})-\nu\hat{\mu}^{(n,K)}(\overrightarrow{\frak{T}_x^{(n,K)}})} =0,\]
where the convergence is in $ {\bf P}_n$-probability.
\end{lemma}

\begin{proof}
Begin noticing that the collection $\mu^{(n,K)}(\{x':x'\leq x\}), x\in V(\T^{\text{GW}}_n(K))$, where $\leq$ denotes the lexicographical total order in $\T^{\text{GW}}_n$, determines completely the measure $\mu^{(n,K)}$. 
Therefore, it is enough to show that
\[\lim_{n\to\infty}\sup_{x\in V(T_n(k))} n^{-1}|\mu^{(n,K)}(\{x':x'\leq x\})-c_2\hat{\mu}_n^{(K)}(\{x':x'\leq x\})|= 0,\]
in ${\bf P}_n$-probability.
Notice that the descendants of a given point $x$ is always an interval with respect  to the lexicographical order. 
Hence, the display above follows by Theorem \ref{teo:legalllinedge}.
\end{proof}

We can now prove  Lemma \ref{l:convergenceofmuk}.
\begin{proof}[Proof of Lemma \ref{l:convergenceofmuk}]
Note that
$\Upsilon^{-1}_{n,K}(\overrightarrow{\frak{T}_x^{(K)}})=\overrightarrow{\frak{T}_{\Upsilon^{-1}_{n,K}(x)}^{(n,K)}}$. Therefore, by Lemma \ref{lem:Prokhorov}, it follows that 
\[\lim_{n\to\infty} \sup_{x\in \frak{T}^{(K)}}\abs{\nu\hat{\mu}^{(n,K)}(\Upsilon_{n,K}^{-1}(\overrightarrow{\frak{T}_x^{(K)}}))-\mu^{(n,K)}(\Upsilon_{n,K}^{-1}(\overrightarrow{\frak{T}_x^{(K)}})) }=0. \]
Therefore, the result is a consequence of Lemma \ref{lem:croydonconvergenceofmeasures}. 
\end{proof}

\subsubsection{Proof of Lemma \ref{l:lambdarestolambdak}}

Let us prove the second technical estimate.

\begin{proof}[Proof of Lemma \ref{l:lambdarestolambdak}]
Let $e^n_j=(a_j^n,b_j^n), a_j\prec b_j$ be the edges to $\frak{T}^{(n,K)}$. Let $x\in\frak{T}^{(n,K)}$ and $e^n(x)=(a^n(x),b^n(x))$ be the edge to which $x$ belongs. 
We have that
\[
\lambda^{(n,K)}(\overrightarrow{\frak{T}_x^{(n,K)}})=\frac{\sum_{j: x\prec a_j^n} |e^n_j|_{d^{(n,K)}}+d^{(n,K)}(x,b^n(x))}{| \frak{T}^{(n,K)}|_{d^{(n,K)}}}.
\]
On the other hand, let $e_j=(a_j,b_j), a_j\prec b_j$ be the edges to $\frak{T}^{(K)}$. Let $x\in\frak{T}^{(K)}$ and $e(x)=(a(x),b(x))$ be the edge to which $\Upsilon_{n,K}(x)$ belongs. Then
\[
\lambda^{(K)}(\Upsilon_{n,K}(\overrightarrow{\frak{T}_x^{(n,K)}}))=\frac{\sum_{j: \Upsilon_{n,K}(x)\prec a_j} |e_j|_{d_{\frak{T}}}+d_{\frak{T}}(\Upsilon_{n,K}(x),b(x))}{| \frak{T}^{(K)}|_{d_{\frak{T}}}}.
\]
It follows from \eqref{eq:underlyinggeometricestimate} that
\[\lim_{n\to\infty} \sup_{x\in\frak{T}^{(n,K)}}\abs{\frac{\sum_{j: x\prec a_j^n} |e^n_j|_{d^{(n,K)}}}{| \frak{T}^{(n,K)}|_{d^{(n,K)}}}-\frac{\sum_{j: \Upsilon_{n,K}(x)\prec a_j} |e_j|_{d_{\frak{T}}}}{| \frak{T}^{(K)}|_{d_{\frak{T}}}}}=0, \quad {\bf P}_n\text{-a.s.}\]
It remains to show that
\begin{equation}\label{eq:todoeshorizonte}
\lim_{n\to\infty}\sup_{x\in\frak{T}^{(n,K)}}\abs{\frac{d^{(n,K)}(x,b^n(x))}{| \frak{T}^{(n,K)}|_{d^{(n,K)}}}-\frac{d_{\frak{T}}(\Upsilon_{n,K}(x),b(x))}{| \frak{T}^{(K)}|_{d_{\frak{T}}}}},\quad {\bf P}_n\text{-a.s.}
\end{equation}
Note that, by the definition of $\Upsilon_{n,K}$, 
\[ d_{\frak{T}}(\Upsilon_{n,K}(x),b(x))= d^{(n,K)}(x,b^n(x))\frac{|e_i|_{d_{\frak{T}}}}{|e_i^n|_{d^{(n,K)}}}.\]
Therefore \eqref{eq:todoeshorizonte} follows from \eqref{eq:underlyinggeometricestimate}.
\end{proof}

\subsubsection{Proof of Lemma \ref{lem:lambdaktomuk}}

Let us prove the third technical estimate.
\begin{proof}[Proof of Lemma \ref{lem:lambdaktomuk}]
It follows from \cite[Theorem 3, (ii)]{Al1} that $\lambda^{(K)}$ converges weakly to $\lambda_{\frak{T}}$ as $K\to\infty$. On the other hand, by \cite[Lemma 2.4]{Croydon_crt} we know that the projection $\pi_{\frak{T}^{(K)}}$ satisfies
\[
\lim_{K\to\infty}\sup_{x\in\frak{T}}d_{\frak{T}}(\pi_{\frak{T}^{(K)}}(x),x)=0, \quad {\bf P}_n\text{-a.s.}.
\]
Therefore, it follows from the definition of $\mu^{(K)}$ and $\lambda^{\frak{T}}$ that
$\mu^{(K)}$ converges weakly to $\lambda^{\frak{T}}$ as $K\to\infty$. The two previous facts together with the fact that $\lambda^{\frak{T}}$ has no atoms imply the lemma.
\end{proof}



\appendix

\section{Basic estimates on branching random walks}\label{sect_be_brw}

The  goal for this section is to prove the existence of pivotal points for branching random walks (see Section~\ref{sect_res_est_not} for notations).

Before starting, let us recall some estimates on simple random walks. We denote $p_n(x,y)=P_x[X_n=y]$ the heat kernel of the simple random walk on $\Z^d$. It verifies (see e.g.~\cite{Varo})
\begin{equation}\label{gauss_bound}
 p_n(x,y) \leq C n^{-d/2} \exp\bigl(-\frac{\abs{x-y}^2}n\bigr).
\end{equation}
and (see e.g.~\cite{Grigoryan})
\begin{equation} \label{gauss_bound2}
c n^{-d/2} \exp\bigl(-\frac{\abs{x-y}^2}n\bigr)\leq  p_n(x,y)+ p_{n+1}(x,y)  
\end{equation}

We also set $G_{d}(x,y)={\bf E}_x\Bigl[\sum_{n\in \N} \1{X_n=y}\Bigr]=\sum_{n\in \N} p_n(x,y)$ the Green function in $\Z^d$. For $d\geq 3$, we have (see e.g.~Theorem 3.5 in~\cite{LG4})
\begin{equation}
\label{bound_green}
G_{d}(x,y)\leq C \abs{x-y}^{-(d-2)}.
\end{equation}

\subsection{Probability of intersection of critical branching random walks}

We recall that $\mathcal{B}(x)$ and $\T^{\text{GW}}_*$ were defined in Section~\ref{sect_res_est_not} . We have
\begin{lemma}\label{appendix_lemma}
Fix $d\geq 3$. For any $x,y\in \Z^d$, we have
\[
{\bf E}[ \abs{\mathcal{B}(x)\cap \mathcal{B}(y)}] \leq C \abs{x-y}^{-(d-4)}.
\]
\end{lemma}
\begin{proof}
In order to estimate the number of intersections between $\mathcal{B}(x)$ and $\mathcal{B}(y)$, we shall do a union bound on all generations $i^{(x)}$ of $\T^{\text{GW}}_*(x)$ (resp.~$i^{(y)}$ of $\T^{\text{GW}}_*(y)$) and on all points $z\in \T^{\text{GW}}_*(x)$ with $\abs{z}= i^{(x)}$ (resp.~$z\in \T^{\text{GW}}_*(y)$ with $\abs{z}=i^{(y)}$). This yields, averaging over the randomness of the embedding,
\begin{align*}
& {\bf E}[ \abs{\mathcal{B}(x)\cap \mathcal{B}(y)}]\\
\leq & {\bf E}\Bigl[ \sum_{\substack{i^{(x)} \in [0,H(\T^{\text{GW}}_*(x))] \\ i^{(y)}\in [0,H(\T^{\text{GW}}_*(y))]}} \sum_{\substack{z\in \T^{\text{GW}}_*(x),\ \abs{z}=i^{(x)} \\z\in \T^{\text{GW}}_*(y),\ \abs{z}=i^{(y)} }}\sum_{z\in \Z^d}p_{i^{(x)}}(x,z)p_{i^{(y)}}(y,z)\Bigr] \\
\leq &  {\bf E}\Bigl[ \sum_{\substack{i^{(x)} \in [0,H(\T^{\text{GW}}_*(x))] \\ i^{(y)}\in [0,H(\T^{\text{GW}}_*(y))]}}  \sum_{\substack{z\in \T^{\text{GW}}_*(x),\ \abs{z}=i^{(x)} \\z\in \T^{\text{GW}}_*(y),\ \abs{z}=i^{(y)} }}p_{i^{(x)}+i^{(y)}}(x,y)\Bigr] \\
\leq &   {\bf E}\Bigl[ \sum_{\substack{i^{(x)} \in [0,H(\T^{\text{GW}}_*(x))] \\ i^{(y)}\in [0,H(\T^{\text{GW}}_*(y))]}} Z_{i^{(x)}}(\T^{\text{GW}}_*(x)) Z_{i^{(y)}}(\T^{\text{GW}}_*(y)) p_{i^{(x)}+i^{(y)}}(x,y)\Bigr], \end{align*}
where $Z_{n}(\T^{\text{GW}}_*(z))$ denotes the cardinal of $\T^{\text{GW}}_*(z)$ at generation $n$.

Since the trees $\T^{\text{GW}}_*(x)$ and $\T^{\text{GW}}_*(y)$ are critical Galton-Watson trees with a special initial distribution $\tilde{Z}_0$. The distribution of $1+\tilde{Z}_0$ is a size-biased version of $Z_1$ under ${\bf P}$ and hence ${\bf E}[\tilde{Z}_1]<\infty$ since ${\bf E}[Z_1^2]<\infty$. Recalling that $E[Z_1]=1$, we know that $E[Z_{i^{(x)}}(\T^{\text{GW}}_*(x))]=Z_{i^{(y)}}(\T^{\text{GW}}_*(y))={\bf E}[\tilde{Z}_1]<\infty$, which implies
\begin{align*}
{\bf E}[ \abs{\mathcal{B}(x)\cap \mathcal{B}(y)}] & \leq  C \sum_{i^{(x)}, i^{(y)}\in \N^2}  p_{i^{(x)}+i^{(y)}}(x,y)\\
&\leq  C\sum_{k=0}^{\infty} k p_k(x,y) \\
&\leq  C \sum_{k=0}^{\infty} k k^{-d/2} \exp\bigl(-\frac{\abs{x-y}^2}k\bigr),
\end{align*}
where we used~\eqref{gauss_bound}. By using~\eqref{gauss_bound2} as well as and~\eqref{bound_green}, we see that
\[
{\bf E}[ \abs{\mathcal{B}(x)\cap \mathcal{B}(y)}] \leq  CG_{d-2}(x,y) \leq  C\abs{x-y}^{-(d-4)}.
\]
\end{proof}

Let us set
\[
q(x,y)={\bf P}[ \mathcal{B}(x)\cap \mathcal{B}(y)\neq \emptyset].
\]

By Markov's inequality, we have
\begin{corollary}\label{bound_q}
Fix $d\geq 3$.  For any $x,y\in \Z^d$, we have
\[
q(x,y) \leq C \abs{x-y}^{-(d-4)}.
\]
\end{corollary}

\subsection{Probability of intersection of two IICBRW}\label{proof_intersect}

We recall that $\mathcal{B}(X_{(R,\infty)})$ for $R>0$ was defined in Section~\ref{sect_res_est_not}. Let us first prove the following lemma
\begin{lemma}\label{cutp_0}
Fix $d\geq 3$. For $\alpha(n)$ a simple random walk started at $0$, we have for any $R>0$
\[
{\bf P}[\mathcal{B}(0)\cap \mathcal{B}(\alpha((R,\infty))) \neq \emptyset] \leq C R^{-(d-6)/2}.
\]
\end{lemma}
\begin{proof}

We have by~\eqref{gauss_bound} and Corollary~\ref{bound_q}
\begin{align*}
{\bf P}[\mathcal{B}(0)\cap \mathcal{B}(\alpha((R,\infty))) \neq \emptyset]  &\leq  \sum_{n>R} \sum_{y\in\Z^d} p_n(0,y) q(0,y) \\
&\leq C \sum_{n>R} \sum_{y\in\Z^d}  n^{-d/2} \exp\bigl(-\frac{\abs{y}^2}n\bigr) \abs{y}^{-(d-4)} \\
&\leq  C \sum_{n>R} n^{-d/2} \sum_{k=0}^{\infty} k^{d-1} k^{-(d-4)} \exp\bigl(-\frac{\abs{k}^2}n\bigr),
\end{align*}
where, to obtain the last line, we decided to partition $\Z^d$ into spheres. From here, we may notice that
\[
\sum_{k=0}^{\infty} k^3 \exp\bigl(-\frac{\abs{k}^2}n\bigr) \leq C \int_0^{\infty} x^3\exp\bigl(-\frac{x^2}n\bigr) dx \leq C n^{2}\int_0^{\infty} u^3\exp(-u^2)du,
\]
which leads us to the following bound
\begin{align*}
{\bf P}[\mathcal{B}(0)\cap \mathcal{B}(X_{(R,\infty)}) \neq \emptyset] &\leq C  \sum_{n>R} n^{-d/2} n^{2} \\
 &\leq  C\sum_{n>R} n^{-(d-4)/2} \\
 &\leq  CR^{-(d-6)/2}.
\end{align*}
\end{proof}

Our main estimate for this section is the following
\begin{lemma}\label{cutp_1b}
Fix $d\geq 3$. For $(\alpha(n))_{n\in \Z}$ a simple random walk started at $0$, we have,
\[
{\bf P}[\mathcal{B}(\alpha((-\infty, 0)))\cap \mathcal{B}(\alpha((R,\infty))) \neq \emptyset] \leq C R^{-(d-8)/2}.
\]
\end{lemma}
\begin{proof}
We have
\begin{align*}
{\bf P}[\mathcal{B}(\alpha((-\infty, 0)))\cap \mathcal{B}(\alpha((R,\infty))) \neq \emptyset]  & \leq \sum_{l=0}^{\infty} {\bf P}[\mathcal{B}(\alpha(-l))\cap \mathcal{B}(\alpha((R,\infty))) \neq \emptyset]  \\
& \leq   \sum_{l=0}^{\infty} {\bf P}[\mathcal{B}(0)\cap \mathcal{B}(\alpha((R+l,\infty))) \neq \emptyset],
\end{align*}
where we use translation invariance. Then, we can see by Lemma~\ref{cutp_0}
\begin{align*}
 \sum_{l=0}^{\infty} {\bf P}[\mathcal{B}(0)\cap \mathcal{B}(\alpha((R+l,\infty))) \neq \emptyset] & \leq  C\sum_{l=0}^{\infty} (R+l)^{-(d-6)/2} \\
& \leq  C \int_0^{\infty} (R+x)^{-(d-6)/2} dx \\
&\leq  C \int_R^{\infty} x^{-(d-6)/2} dx \\
& \leq  C R^{-(d-8)/2},
\end{align*}
and the lemma follows.
\end{proof}

\subsection{Existence of pivotal points}~\label{proof_pivotal}

The notion of pivotal point was defined in Section~\ref{sect_res_est_not}. The goal of this part is to prove the existence of such points in high enough dimensions, more specifically we want to prove Lemma~\ref{cutp_2}. Our proof is inspired by~\cite{erdHos1960some}.
\begin{proof}
Denote
\[
r(n)={\bf P}[\mathcal{B}(\alpha((-\infty, 0]))\cap \mathcal{B}(\alpha([1,n])) = \emptyset],
\]
and 
\[
 s(k)={\bf P}[\mathcal{B}(\alpha((-\infty, 0]))\cap \mathcal{B}(\alpha(k)) \neq \emptyset] .
\]

Notice that $r(n)$ is a non-negative non-increasing function of $n$ and as such admits a limit $\alpha:=\lim_{n\to \infty} r(n)$. It turns out that $\alpha=P[0\text{ is a pivotal point}]$ and hence our goal is to show that $\alpha>0$. We see that
\begin{equation}\label{cut_series}
\sum_{k=1}^{\infty} s(k) <\infty,
\end{equation}
since $s(k)\leq C k^{-(d-8)/2}$ by Lemma~\ref{cutp_1b} and $d>10$.

On the event $\{\mathcal{B}(\alpha((-\infty, 0]))\cap \mathcal{B}(\alpha([1,n])) \neq \emptyset \}$, we know that there exists $k\in [1,n]$ such that $\mathcal{B}(\alpha(k))\cap \mathcal{B}(\alpha((-\infty, 0])) \neq \emptyset$. Decomposing along the largest such $k$, we see that
\[
\sum_{k=1}^n s(k)r(n-k)=1-r(n).
\]

Moreover $r(n)$ is a non-increasing function of $n$ with $r(n)\leq 1$, which implies that
\[
r(n-\lfloor n/2\rfloor ) \sum_{k=1}^{\lfloor n/2\rfloor } s(k) +\sum_{\lfloor n/2\rfloor }^n s(k) \geq 1-r(n).
\]

Recalling~\eqref{cut_series}, we can take the limit as $n\to \infty$ in the previous equation, which allows us to see that
\[
\alpha \sum_{k=1}^{\infty} s(k)\geq 1-\alpha,
\]
which implies that $\alpha>0$. As mentioned at the beginning of the proof, this proves the lemma.
\end{proof}

\section*{Glossary of notations}

\vspace{0.5cm}

\def\qq{&}
\begin{center}
\halign{
#\quad\hfill&#\quad\hfill&\quad\hfill#\cr
${\bf P}$ \qq Law of the Galton-Watson tree \fff{P}
$Z$ \qq Number of offspring in a Galton-Watson \fff{Z}
$\T^{\text{GW}}$ \qq a Galton-Watson tree \fff{gw}
${\bf P}_n$ \qq law of the Galton-Watson tree with $n$ vertices \fff{Pn}
$\T^{\text{GW}}_n$ \qq a Galton-Watson tree conditioned to have $n$ vertices \fff{phin}
$\phi_{\T^{\text{GW}}_n}$ \qq random embedding of $\T^{\text{GW}}_n$ \fff{phin}
$\omega_n$ \qq graph of the critical branching random walk from $\T^{\text{GW}}_n$ \fff{omega}
$\mathfrak{T}$ \qq The Continuum random tree (CRT) \fff{crt}
$\phi_{\mathfrak{T}}(\mathfrak{T})$ \qq  integrated super-Brownian excursion (ISE) \fff{ise}
$\sigma_Z$ \qq variance of $Z$ \fff{sigZ}
$B^{ISE}$  \qq Brownian motion on the ISE \fff{bise}
$B^{CRT}$  \qq Brownian motion on the CRT \fff{bise}
$\mathfrak{T}^{(K)}$ \qq $K$-CRT \fff{kcrt}
$B^{(K)}$ \qq Brownian motion on $K$-CRT \fff{bcrtk}
$B^{K-ISE}$\qq Brownian motion on $K$-ISE \fff{bisek} 
$(G_n,(V_i^n)_{i\in \N})_{n\in \N}$ \qq sequence of random augmented graphs \fff{gv}
$\T^{(n,K)}$ \qq  $K$-skeleton of our $G_n$ \fff{mass}
$\text{root}^*$ \qq root of the skeleton of $\T^{(n,K)}$ \fff{rootstar}
$\mathfrak{T}^{(n,K)}$ \qq  version of $\T^{(n,K)}$ with fewer vertices \fff{mass}
$V^*(\T^{(n,K)})$ \qq  vertices of $\T^{(n,K)}$ corresponding to cut-points \fff{mass}
$E^*(\T^{(n,K)})$ \qq  edges of $\T^{(n,K)}$  between vertices of $V^*(\T^{(n,K)})$ \fff{mass}
$\phi^{(n,K)}$ \qq random embedding of $\T^{(n,K)}$ \fff{mass}
$v^{(n,K)}(x)$ \qq volume of the sausage attached at $x\in V^*(\T^{(n,K)})$ \fff{mass}
$\reff^{(n,K)}(e)$ \qq resistance of the edge $e \in E(\T^{(n,K)})$ \fff{mass}
$\pi^{(n,K)}(x)$ \qq projection of $x\in G_n$ onto $\T^{(n,K)}$ \fff{mass}
$d^{(n,K)}(\cdot,\cdot)$ \qq rescaled intrinsic distance on $\T^{(n,K)}$ \fff{mass}
 $d^{(n,K)}_{\text{res}}(\cdot,\cdot)$ \qq rescaled resistance distance on $\T^{(n,K)}$ \fff{mass}
  $\mu^{(n,K)}$ \qq volume measure on $\T^{(n,K)}$ \fff{mass}
  $\lambda_{T}$  \qq renormalized Lebesgue measure on $T$ \fff{lambdaT}
$\Delta^{(n,K)}_{\Z^d}$ \qq maximal $\Z^d$-diameter of sausages \fff{delta1}
$\Delta^{(n,K)}_{\Z^d}$ \qq maximal $\omega_n$-diameter of sausages \fff{delta2}
$\overrightarrow{\T^{(n,K)}_{x}}$ \qq are the descendants of $x$ in $\T^{(n,K)}$ \fff{arrowt}
$\pi_n(x)$ \qq  first cut-point encountered on any path from $x$ to $0$ \fff{pin}

\cr}\end{center}

\bibliographystyle{plain}
\bibliography{RWBRW}

\end{document}